\pdfoutput=1
\documentclass[11pt, a4paper, twoside,leqno]{amsart}
\usepackage[centering, totalwidth = 380pt, totalheight = 600pt]{geometry}
\usepackage{amssymb, amsmath, amsthm}
\usepackage[bbgreekl]{mathbbol}
\usepackage{microtype, stmaryrd, url, lmodern, eucal, cmll, physics}
\usepackage{tikz}
\usepackage[shortlabels]{enumitem}
\usepackage[latin1]{inputenc}
\usepackage{color}
\definecolor{darkgreen}{rgb}{0,0.45,0}
\usepackage[colorlinks,citecolor=darkgreen,linkcolor=darkgreen]{hyperref}
\usepackage[british]{babel}
\usetikzlibrary{shapes,calc}

\makeatletter
\def\@cite#1#2{[{#1\if@tempswa ,~#2\fi}]}
\makeatother

\DeclareMathAlphabet{\mathbf}{OT1}{cmr}{b}{n}

\usepackage[arrow, matrix, tips, curve, graph, rotate]{xy}
\SelectTips{cm}{10}

\makeatletter
\def\matrixobject@{%
  \edef \next@{={\DirectionfromtheDirection@ }}%
  \expandafter \toks@ \next@ \plainxy@
  \let\xy@@ix@=\xyq@@toksix@
  \xyFN@ \OBJECT@}
\let\xy@entry@@norm=\entry@@norm
\def\entry@@norm@patched{%
  \let\object@=\matrixobject@
  \xy@entry@@norm }
\AtBeginDocument{\let\entry@@norm\entry@@norm@patched}
\makeatother

\newcommand{\twocong}[2][0.5]{\ar@{}[#2] \save ?(#1)*{\cong}\restore}
\newcommand{\twoeq}[2][0.5]{\ar@{}[#2] \save ?(#1)*{=}\restore}
\newcommand{\rtwocell}[3][0.5]{\ar@{}[#2] \ar@{=>}?(#1)+/l 0.2cm/;?(#1)+/r 0.2cm/^{#3}}
\newcommand{\ltwocell}[3][0.5]{\ar@{}[#2] \ar@{=>}?(#1)+/r 0.2cm/;?(#1)+/l 0.2cm/^{#3}}
\newcommand{\ltwocello}[3][0.5]{\ar@{}[#2] \ar@{=>}?(#1)+/r 0.2cm/;?(#1)+/l 0.2cm/_{#3}}
\newcommand{\dtwocell}[3][0.5]{\ar@{}[#2] \ar@{=>}?(#1)+/u  0.2cm/;?(#1)+/d 0.2cm/^{#3}}
\newcommand{\dltwocell}[3][0.5]{\ar@{}[#2] \ar@{=>}?(#1)+/ur  0.2cm/;?(#1)+/dl 0.2cm/^{#3}}
\newcommand{\drtwocell}[3][0.5]{\ar@{}[#2] \ar@{=>}?(#1)+/ul  0.2cm/;?(#1)+/dr 0.2cm/^{#3}}
\newcommand{\dthreecell}[3][0.5]{\ar@{}[#2] \ar@3{->}?(#1)+/u  0.2cm/;?(#1)+/d 0.2cm/^{#3}}
\newcommand{\utwocell}[3][0.5]{\ar@{}[#2] \ar@{=>}?(#1)+/d 0.2cm/;?(#1)+/u 0.2cm/_{#3}}
\newcommand{\dtwocelltarg}[3][0.5]{\ar@{}#2 \ar@{=>}?(#1)+/u  0.2cm/;?(#1)+/d 0.2cm/^{#3}}
\newcommand{\utwocelltarg}[3][0.5]{\ar@{}#2 \ar@{=>}?(#1)+/d  0.2cm/;?(#1)+/u 0.2cm/_{#3}}

\newdir{(}{{}*!<0em,-.14em>-\cir<.14em>{l^r}}
\newdir{ (}{{}*!/-5pt/\dir{(}}
\newdir{ >}{{}*!/-5pt/\dir{>}}

\DeclareMathOperator{\id}{id}
\DeclareMathOperator{\ov}{\ovee}

\newcommand{\cat}[1]{\mathrm{\mathcal #1}}
\newcommand{\thg}{{\mathord{\text{--}}}}

\newcommand{\res}[2]{\left.{#1}\right|_{#2}}

\newcommand{\spn}[1]{{\langle{#1}\rangle}}

\newcommand{\defeq}{\mathrel{\mathop:}=}
\newcommand{\cd}[2][]{\vcenter{\hbox{\xymatrix#1{#2}}}}

\renewcommand{\phi}{\varphi}

\newcommand{\C}{{\mathcal C}}
\newcommand{\D}{{\mathcal D}}
\newcommand{\E}{{\mathcal E}}

\newcommand{\G}{{\mathcal G}}
\renewcommand{\H}{{\mathcal H}}

\newcommand{\K}{{\mathcal K}}

\newcommand{\M}{{\mathcal M}}
\newcommand{\N}{{\mathcal N}}

\renewcommand{\P}{{\mathcal P}}

\newcommand{\R}{{\mathcal R}}
\let\sec=\S
\renewcommand{\S}{{\mathcal S}}

\newcommand{\U}{{\mathcal U}}

\newcommand{\xtor}[1]{\cdl[@1]{{} \ar[r]|-{\object@{|}}^{#1} & {}}}
\newcommand{\tor}{\ensuremath{\relbar\joinrel\mapstochar\joinrel\rightarrow}}

\makeatletter

\def\hookleftarrowfill@{\arrowfill@\leftarrow\relbar{\relbar\joinrel\rhook}}
\def\twoheadleftarrowfill@{\arrowfill@\twoheadleftarrow\relbar\relbar}
\def\leftbararrowfill@{\arrowdoublefill@{\leftarrow\mkern-5mu}\relbar\mapstochar\relbar\relbar}
\def\Leftbararrowfill@{\arrowdoublefill@{\Leftarrow\mkern-2mu}\Relbar\Mapstochar\Relbar\Relbar}
\def\leftringarrowfill@{\arrowdoublefill@{\leftarrow\mkern-3mu}\relbar{\mkern-3mu\circ\mkern-2mu}\relbar\relbar}
\def\lefttriarrowfill@{\arrowfill@{\mathrel\triangleleft\mkern0.5mu\joinrel\relbar}\relbar\relbar}
\def\Lefttriarrowfill@{\arrowfill@{\mathrel\triangleleft\mkern1mu\joinrel\Relbar}\Relbar\Relbar}

\def\hookrightarrowfill@{\arrowfill@{\lhook\joinrel\relbar}\relbar\rightarrow}
\def\twoheadrightarrowfill@{\arrowfill@\relbar\relbar\twoheadrightarrow}
\def\rightbararrowfill@{\arrowdoublefill@{\relbar\mkern-0.5mu}\relbar\mapstochar\relbar\rightarrow}
\def\Rightbararrowfill@{\arrowdoublefill@{\Relbar\mkern-2mu}\Relbar\Mapstochar\Relbar\Rightarrow}
\def\rightringarrowfill@{\arrowdoublefill@\relbar\relbar{\mkern-2mu\circ\mkern-3mu}\relbar{\mkern-3mu\rightarrow}}
\def\righttriarrowfill@{\arrowfill@\relbar\relbar{\relbar\joinrel\mkern0.5mu\mathrel\triangleright}}
\def\Righttriarrowfill@{\arrowfill@\Relbar\Relbar{\Relbar\joinrel\mkern1mu\mathrel\triangleright}}

\def\leftrightarrowfill@{\arrowfill@\leftarrow\relbar\rightarrow}
\def\mapstofill@{\arrowfill@{\mapstochar\relbar}\relbar\rightarrow}

\renewcommand*\xleftarrow[2][]{\ext@arrow 20{20}0\leftarrowfill@{#1}{#2}}
\providecommand*\xLeftarrow[2][]{\ext@arrow 60{22}0{\Leftarrowfill@}{#1}{#2}}
\providecommand*\xhookleftarrow[2][]{\ext@arrow 10{20}0\hookleftarrowfill@{#1}{#2}}
\providecommand*\xtwoheadleftarrow[2][]{\ext@arrow 60{20}0\twoheadleftarrowfill@{#1}{#2}}
\providecommand*\xleftbararrow[2][]{\ext@arrow 10{22}0\leftbararrowfill@{#1}{#2}}
\providecommand*\xLeftbararrow[2][]{\ext@arrow 50{24}0\Leftbararrowfill@{#1}{#2}}
\providecommand*\xleftringarrow[2][]{\ext@arrow 10{26}0\leftringarrowfill@{#1}{#2}}
\providecommand*\xlefttriarrow[2][]{\ext@arrow 80{24}0\lefttriarrowfill@{#1}{#2}}
\providecommand*\xLefttriarrow[2][]{\ext@arrow 80{24}0\Lefttriarrowfill@{#1}{#2}}

\renewcommand*\xrightarrow[2][]{\ext@arrow 01{20}0\rightarrowfill@{#1}{#2}}
\providecommand*\xRightarrow[2][]{\ext@arrow 04{22}0{\Rightarrowfill@}{#1}{#2}}
\providecommand*\xhookrightarrow[2][]{\ext@arrow 00{20}0\hookrightarrowfill@{#1}{#2}}
\providecommand*\xtwoheadrightarrow[2][]{\ext@arrow 03{20}0\twoheadrightarrowfill@{#1}{#2}}
\providecommand*\xrightbararrow[2][]{\ext@arrow 01{22}0\rightbararrowfill@{#1}{#2}}
\providecommand*\xRightbararrow[2][]{\ext@arrow 04{24}0\Rightbararrowfill@{#1}{#2}}
\providecommand*\xrightringarrow[2][]{\ext@arrow 01{26}0\rightringarrowfill@{#1}{#2}}
\providecommand*\xrighttriarrow[2][]{\ext@arrow 07{24}0\righttriarrowfill@{#1}{#2}}
\providecommand*\xRighttriarrow[2][]{\ext@arrow 07{24}0\Righttriarrowfill@{#1}{#2}}

\providecommand*\xmapsto[2][]{\ext@arrow 01{20}0\mapstofill@{#1}{#2}}
\providecommand*\xleftrightarrow[2][]{\ext@arrow 10{22}0\leftrightarrowfill@{#1}{#2}}
\providecommand*\xLeftrightarrow[2][]{\ext@arrow 10{27}0{\Leftrightarrowfill@}{#1}{#2}}

\makeatother

\numberwithin{equation}{section}

\theoremstyle{plain}
\newtheorem{Thm}{Theorem}
\newtheorem{Prop}[Thm]{Proposition}

\newtheorem{Lemma}[Thm]{Lemma}
\newtheorem*{Lemma*}{Lemma}

\theoremstyle{definition}
\newtheorem{Defn}[Thm]{Definition}

\newtheorem{Ex}[Thm]{Example}
\newtheorem{Exs}[Thm]{Examples}
\newtheorem{Rk}[Thm]{Remark}

\newcommand{\mnd}[1]{\mathsf{#1}}

\newcommand{\Kl}[1]{\cat{Kl}(\mathsf{#1})}

\begin{document}
\leftmargini=2em \title[Abstract hypernormalisation]{Abstract hypernormalisation, and
  normalisation-by-trace-evaluation for generative systems}
\author{Richard Garner} \address{School of Mathematical and Physical
  Sciences, Macquarie University, NSW 2109, Australia}
\email{richard.garner@mq.edu.au}

\subjclass[2000]{Primary: }
\date{\today}

\thanks{The support of Australian Research Council grants DP160101519
  and FT160100393 is gratefully acknowledged. Thanks also to the
  referees, whose constructive and robust analyses have improved the
  paper immensely.}

\begin{abstract}
  Jacobs' \emph{hypernormalisation} is a construction on finitely
  supported discrete probability distributions, obtained by
  generalising certain patterns occurring in quantitative information
  theory. In this paper, we generalise Jacobs' notion in turn, by
  describing a notion of hypernormalisation in the abstract setting of
  a symmetric monoidal category endowed with a linear exponential
  monad---a structure arising in the categorical semantics of linear
  logic.
  
  We show that Jacobs' hypernormalisation arises in this fashion from
  the finitely supported probability measure monad on the category of
  sets, which can be seen as a linear exponential monad with respect
  to a non-standard monoidal structure on sets which we term the
  \emph{convex monoidal structure}. We give the construction of this
  monoidal structure in terms of a quantum-algebraic notion known as a
  \emph{tricocycloid}.

  Besides the motivating example, and its natural generalisations to
  the continuous context, we give a range of other instances of our
  abstract hypernormalisation, which swap out the side-effect of
  probabilistic choice for other important side-effects such as
  non-deterministic choice, logical choice via tests in a Boolean
  algebra, and input from a stream of values.

  Finally, we exploit our framework to describe a
  normalisation-by-trace-evaluation process for behaviours of various
  kinds of coalgebraic generative systems, including labelled
  transition systems, probabilistic generative
  systems~\cite{Glabbeek1990Reactive}, and stream
  processors~\cite{Hancock2009Representations}.
\end{abstract}
\maketitle

\section{Introduction}
\label{sec:introduction}

\emph{Hypernormalisation} was introduced by Jacobs
in~\cite{Jacobs2017Hyper} in order to provide, among other things, a
smooth category-theoretic formulation of certain
concepts~\cite{McIver2014Abstract} of quantitative information flow.
The main theoretical contribution of this paper is to analyse, in turn,
Jacobs' notion of hypernormalisation, showing how it arises naturally
out of well-studied category-theoretic concepts, and how it
generalises to other settings. The main application of this paper
uses our framework in order to relate bisimilarity and trace
equivalence for a range of coalgebraic generative systems. As is
well-known, for such systems, bisimulation equivalence is a finer
relation than trace equivalence, so that multiple different behaviours
(i.e., states up to bisimilarity) may have the same underlying trace.
We will describe a process of ``normalisation-by-trace-evaluation''
which normalises any given behaviour to a maximally-efficient one with
the same trace.

To motivate our development, it will be useful to take a step back,
and first describe the ideas of quantitative information flow which
motivated~\cite{Jacobs2017Hyper}. These ideas are concerned with the
following question: given a probabilistic process $P$ which, when run
on a \emph{private} input returns a \emph{public} output, how can we
measure the leakage inherent in $P$, that is, the extent to which
knowledge of the output allows an adversary to infer knowledge about
the input?

The natural way of answering this question is information-theoretic:
we model the input and output types as finite sets $A$ and $B$, and
$P$ as a discrete channel, that is, a function assigning to each input
in $A$ a discrete probability distribution over possible outputs in
$B$. Using this, any prior distribution $\pi$ on $A$ yields a joint
distribution on $A \times B$ with associated marginals on $A$
(viz.~$\pi$) and $B$; we can now define the leakage of $P$ as the
``mutual information'', i.e., the difference between the sum
of the marginals' entropies and the entropy of the joint distribution; said
another way, the leakage of $P$ is the difference between the entropy
of the prior distribution on $A$, and the expected entropy (over all
possible observed outputs) of the posterior distribution on $A$.

Now, if $Q$ is another process of input type $A$---but possibly
different output type---then we can compare their drops in entropy to
ascertain whether $P$ or $Q$ leaks more information about $A$.
\emph{Prima facie}, the answer to this depends not only on a choice of
prior $\pi$ for $A$, but also on the flavour of entropy chosen for the
calculations---which has to do with questions like: must a successful
attack guess the input precisely, or would partial knowledge suffice?
A key contribution of~\cite{McIver2014Abstract} is to exhibit a
leakage-ordering which is \emph{robust}, in the sense of being largely
independent of these choices---and moreover, depends not on the
channels $P$ and $Q$ themselves, but only on their so-called
\emph{abstract channel} denotations.

According to~\cite{McIver2014Abstract}, an \emph{abstract channel} of
type $A$ is a function $\D A \rightarrow \D(\D A)$ from the set $\D A$
of probability distributions on $A$ to the set of finitely supported
probability distributions on $\D A$; said another way, it is a
discrete time-homogeneous Markov chain with set of states $\D A$. In
the case of the abstract channel $P^r$ associated to a channel $P$
from $A$ to $B$, this Markov chain encodes the probabilities that a
given prior distribution will update to a given posterior distribution
on account of an observation in $B$. Crucially, however, the identity
of these observations is suppressed, which is what makes the
abstract channel~``abstract''.

Now in~\cite{McIver2014Abstract}, the status of probabilistic channels as
a kind of monadic computation~\cite{Moggi1991Notions} is clearly
acknowledged; indeed, as is well known, the operation $A \mapsto \D A$
underlies the finitely supported discrete distribution monad $\mnd{D}$ on the
category of sets. However, the construction
in~\cite{McIver2014Abstract} of an abstract channel $P^r$ from a
channel $P$ does not exploit this fact. This is
where Jacobs'~\cite{Jacobs2017Hyper} enters the picture; one of its
objectives is to explain the construction $P \mapsto P^r$ via the
calculus of monadic computation, so providing a framework for
generalisation beyond the finite discrete case.

In Jacobs' analysis, the abstract
channel $P^r$ associated to a discrete channel
$P \colon A \rightarrow \D B$ is found as a composite
\begin{equation}\label{eq:36}
  \D A \xrightarrow{\tilde P} \D(A \times B) \xrightarrow{\N} \D(\D A
  \times B) \xrightarrow{\D(\pi_1)} \D(\D A)\rlap{ ,}
\end{equation}
whose three terms we now explain. The last map is the action of the
monad $\D$ on arrows by pushforward---which in this case takes a joint
distribution on $\D A \times B$ to the associated marginal
distribution on $\D A$. The first map, by contrast, takes a prior
distribution $\pi$ on $A$ to the associated joint distribution on
$A \times B$ given by $a,b \mapsto \pi(a) \cdot (Pa)(b)$. In
categorical terms, this $\tilde P$ arises as the composite
\begin{equation}\label{eq:50}
  \D A \xrightarrow{\D(1, P)} \D(A \times \D B)
  \xrightarrow{\mathsf{str}} \D\D(A \times B) \xrightarrow{\mu} \D(A
  \times B)\rlap{ ,}
\end{equation}
where $\mu$ is the monad multiplication, and
$\mathsf{str}$ is its cartesian \emph{strength}.

The remaining part of~\eqref{eq:36} is the map $\N$, which is Jacobs'
hypernormalisation. It can be described as follows. Given a joint
distribution $\omega \in \D(A \times B)$, we have for each $b \in B$
the marginal probability $\omega(b) = \Sigma_{a} \omega(a,b)$; and
when this is non-zero, we have also the conditional distribution
$\omega_{A\mid b}$ on $A$ with
$\omega_{A \mid b}(a) = \omega(a,b) / \omega(b)$. Now
$\N(\omega) \in \D(\D A \times B)$ is the distribution which takes the
value $(\omega_{A \mid b}, b)$ with probability $\omega(b)$, for all
$b \in B$ with $\omega(b) > 0$ (note that it is important for this
that $B$ is a \emph{finite} set).
In particular, $\N$ encodes the process of normalising a
non-zero sub-probability distribution $\omega(\thg, b)$ on $A$ to the
probability distribution $\omega_{A \mid b}$---while avoiding the
impossibility of normalising $\omega(\thg, b)$ when it is
everywhere-zero. This explains the name \emph{hypernormalisation}
chosen for this map.

In~\cite{Jacobs2017Hyper}, Jacobs introduces hypernormalisation by an
element-based definition, and verifies by hand a number of desirable
equational properties---the implication being that, to generalise away
from the finite discrete setting, it would suffice to define a
corresponding map $\N$, and to verify the corresponding properties.
Our objective here is to replace this \emph{axiomatic} approach with a
\emph{synthetic} one: rather than defining hypernormalisation on a
case-by-case basis, we will show how it arises naturally from a
certain well-known categorical framework. This will, in particular,
allow us in a principled way to generalise hypernormalisation (and
so also channel-abstraction) to diverse other settings.

\looseness=-1
The framework in question is that of a symmetric monoidal category
endowed with a \emph{linear exponential monad}. A linear exponential
monad $\mnd T$ is one for which the symmetric monoidal structure of
the base category lifts to the category of $\mnd T$-algebras and there
becomes finite coproduct. Linear exponential monads originate in the
categorical semantics of linear logic~\cite{Benton1993Linear}, but
also have applications in studying abstract differentiation in
mathematics and computer science~\cite{Blute2006Differential}. A key
observation of this paper is that linear exponential monads
\emph{always} have an associated notion of hypernormalisation,
satisfying all the equational axioms one may hope for.

The motivating example fits into this framework
via the discrete distribution monad $\mnd{D}$ on the category
of sets. This turns out to be a linear exponential monad, but with
respect to a non-standard monoidal structure on $\cat{Set}$ which we
term the \emph{convex monoidal structure}. The convex monoidal
structure has the empty set as unit and binary tensor given by
\begin{equation}\label{eq:34}
  A \star B = A + \big((0,1) \times A \times B\big) + B
\end{equation}
where $(0,1)$ denotes the open interval; while its associativity
constraints are controlled by a map
$v \colon (0,1) \times (0,1) \rightarrow (0,1) \times (0,1)$ which
encodes a particular change of coordinates for points of the
topological $2$-simplex.

The unfamiliar aspect here is the monoidal structure~\eqref{eq:34};
but it turns out that this can, in turn, be understood via another
established piece of category theory. A
\emph{tricocycloid}~\cite{Street1998Fusion} in a symmetric monoidal
category is an object $H$ endowed with an invertible map
$v \colon H \otimes H \rightarrow H \otimes H$ satisfying suitable
axioms; and by a general construction of \emph{loc.~cit.}, any
tricocycloid $H$ in $\C$ gives a new monoidal structure on $\C$
defined by $A \star B = A + H \otimes A \otimes B + B$.

\looseness=-1 Typical examples of tricocycloids arise in the
$k$-linear context from Hopf algebras~\cite{Sweedler1969Hopf} and
multiplier Hopf algebras~\cite{Van-Daele1994Multiplier}; but their
relevance here stems from the fact that $(0,1)$ is a tricocycloid in
the cartesian monoidal category of sets, so that we can derive the
convex monoidal structure via the general construction described
above. We refer to $(0,1)$ endowed with the map $v$ as the
\emph{convex tricocycloid}; it is a basic combinatorial object which
lies at the heart of probability theory.

Beyond recapturing our motivating example, we also give a range of
other examples of hypernormalisation arising from other linear
exponential monads. One obvious direction of generalisation we pursue
exhibits various \emph{continuous} probability monads as linear
exponential monads. However, and more interestingly, we also give a
range of \emph{non}-probabilistic examples. From the well-known
perspective of~\cite{Moggi1991Notions, Plotkin2002Notions}, monads
encode computational effects, and the discrete distribution monad
$\mnd D$ in particular encodes (finite) probabilistic choice. We will
see that the monads encoding other computational effects, including
non-deterministic choice, logical choice over tests valued in a
Boolean algebra, and input from a stream of $B$-values, also admit
hypernormalisation.

These non-probabilistic examples will play a role in the main
application of this paper, which uses our abstract hypernormalisation
to relate \emph{bisimilarity} and \emph{trace equivalence} for
automata of a certain kind. Given a monad $\mnd T$ on a category $\C$
with finite coproducts, and a finite set $A$, we define a
\emph{generative $\mnd T$-system with alphabet $A$} to be a set of
states $S$ with a transition map $S \rightarrow T(\sum_{a \in A} S)$.
For suitable choices of $\mnd T$ this yields labelled transition
systems, \emph{probabilistic generative
  systems}~\cite{Glabbeek1990Reactive}, a ``logical'' analogue of
probabilistic generative systems, and the \emph{stream processors}
of~\cite{Hancock2009Representations}.

In these examples, there are various notions of equivalence on states
available~\cite{Glabbeek1990Linear}, the finest being
\emph{bisimilarity}, and the coarsest \emph{trace equivalence}. On the
one hand, it is well-known that bisimilarity can be captured
category-theoretically in a number of ways~\cite{Staton2011Relating},
the cleanest of which is that two states are bisimilar precisely when
they become equal under the unique map to the \emph{final} generative
$\mnd T$-system. This final object exists under very mild conditions,
and may be thought of as an \emph{object of behaviours} $\mathsf{Beh}$
for generative $\mnd T$-systems.

On the other hand, there is a category-theoretic characterisation of
trace equivalence, originally due to~\cite{Power1999Coalgebraic} and
developed further in \emph{inter alia}~\cite{Hasuo2007Generic}; this
involves viewing a generative $\mnd T$-system as a coalgebra
$S \tor \sum_{a \in A} S$ in the \emph{Kleisli category} of $\mnd T$,
and seeking a final such coalgebra in $\cat{Kl}(\mnd T)$. When this
exists, it provides an object of traces, and so a notion of trace
equivalence on states, given as before by equality under the unique
map to the object of traces. However, since $\Kl{\mnd T}$ is a
badly-behaved category, existence of a final coalgebra is not
guaranteed---and in particular, one does not exist in our examples.

Now, the badly-behaved $\Kl{\mnd T}$ can be identified with the full
subcategory of \emph{free} algebras within the much better-behaved
category $\C^{\mnd T}$ of Eilenberg--Moore $\mnd T$-algebras. So if we
define an \emph{object of traces} for generative $\mnd T$-systems
$\mathsf{Tr}$ to be a final object among coalgebras
$S \rightarrow \sum_{a \in A} S$ not in $\Kl{\mnd T}$ but rather in
$\C^{\mnd T}$, then existence will hold under very mild conditions. In
particular, this is the case in each of our examples---for example,
the object of traces for probabilistic generative systems is precisely
the set of probability distributions on $A^\mathbb{N}$---and the
induced notion of trace equivalence is precisely the expected one.

Now, when the objects of behaviours and traces for generative
$\mnd T$-systems exist, there is a canonical map
$\mathsf{reflect} \colon \mathsf{Beh} \rightarrow \mathsf{Tr}$ which
captures the fact that bisimilar states are also trace equivalent. But
as we will see, when the monad $\mnd T$ admits hypernormalisation,
more is true: the reflection map admits a canonical section
$\mathsf{reify} \colon \mathsf{Tr} \rightarrow \mathsf{Beh}$ which
produces a ``minimal'' behaviour realising each trace---where,
informally, the minimality of a behaviour expresses that it carries
out the \emph{least possible} amount of $\mnd T$-computation in order
to determine the next output token from $A$. In particular, the
idempotent
$\mathsf{reify} \circ \mathsf{reflect} \colon \mathsf{Beh} \rightarrow
\mathsf{Beh}$ implements 
``normalisation-by-trace-evaluation'' for behaviours of generative
$\mnd T$-systems.


We conclude this introduction with a brief overview of the contents of
the paper. We begin Section~\ref{sec:hypernormalisation-1} by
recalling Jacobs' notion of hypernormalisation; as in~\eqref{eq:36},
this is a certain map associated to the monad $\mnd{D}$ for finitely
supported probability distributions on the category of sets. The
algebras for this monad are \emph{abstract convex
  spaces}---the variety of algebras generated by the quasivariety of
convex subsets of affine spaces
(cf.~\cite{Neumann1970On-the-quasivariety}), and our first
contribution is to explain how hypernormalisation can be understood in
terms of finite coproducts in the category of abstract convex spaces.
The key is the (well-known) observation that the binary coproduct of
abstract convex spaces $A$ and $B$ is given by $A\star B$ as
in~\eqref{eq:34}, endowed with a suitable convex structure.

In fact, the results just described do not \emph{quite} recapture
hypernormalisation. In Section~\ref{sec:opmonoidal-comonads}, we
rectify this, and in doing so arrive at the key idea of this paper:
that an appropriate general setting for hypernormalisation is a
symmetric monoidal category endowed with a linear exponential monad.
In this setting, we define a notion of hypernormalisation, and show
that it inherits almost all of the good equational properties of
hypernormalisation noted in~\cite{Jacobs2017Hyper}; we also show that
the qualifier ``almost'' can be removed so long as the symmetric
monoidal structure on the base category is \emph{co-affine}---meaning
that the unit object is initial---and the linear exponential monad
$\mnd T$ is \emph{affine}---meaning that $T1 \cong 1$. We also show
that, when $\C$ is a category with finite products and distributive
finite coproducts, and $\mnd T$ has a cartesian strength, we have a perfect
analogue of the channel-to-abstract-channel construction following
Jacobs' pattern~\eqref{eq:36}.

In Section~\ref{sec:conv-mono-struct}, we exhibit the motivating
example of hypernormalisation as an instance of our general
setting by showing that the convex monoidal structure~\eqref{eq:34} is
indeed a symmetric monoidal structure on $\cat{Set}$ for
which the discrete distribution monad $\mnd D$ is a linear exponential
monad. As discussed, we do this by first constructing the convex
tricocycloid, and applying the general construction
of~\cite{Street1998Fusion}. Along the way, we fill out some aspects of
the theory of tricocycloids, in particular relating them to
\emph{operads} in the sense of~\cite{May1972The-geometry}.

In Sections~\ref{sec:examples} and~\ref{sec:tric-effect-algebr}, we
turn to examples of hypernormalisations beyond the motivating one. In
Section~\ref{sec:examples}, we exhibit the structure required to
obtain hypernormalisation maps for three non-discrete probability
monads: the expectation monad on
sets~\cite{Jacobs2016The-expectation}; the monad of Radon probability
measures on compact Hausdorff spaces~\cite{Mislove2011Probabilistic};
and the Kantorovich monad~\cite{Breugel2005The-metric} on $1$-bounded
metric spaces. Then in Section~\ref{sec:tric-effect-algebr}, we turn
to combinatorial examples, including the monad for total finite
non-determinism $\mnd P_f^+$; a monad for ``logical distributions''
valued in a Boolean algebra $B$; and the monad for $B$-ary branching
trees.

Finally, in Section~\ref{sec:appl-stre-proc}, we develop our
application of abstract hypernormalisation to
normalisation-by-trace-evaluation for generative $\mnd T$-systems. We
begin by introducing these systems, and characterise the associated
objects of \emph{behaviours} and \emph{traces}, along with the
\emph{behaviour map} and \emph{trace map} associated to any generative
$\mnd T$-system. We then explain our framework for
normalisation-by-trace-evaluation in the context of a monad $\mnd T$
which admits hypernormalisation. Finally, we describe this process
explicitly for a range of examples of monads admitting
hypernormalisation drawn from elsewhere in the paper.

\section{Hypernormalisation and convex coproducts}
\label{sec:hypernormalisation-1}

\subsection{Hypernormalisation}
\label{sec:hypernormalisation}
In this section, we first recall from~\cite{Jacobs2017Hyper} the
notion of hypernormalisation for finitely supported discrete
probability distributions, and then explain its relation to coproducts
in the category of \emph{abstract convex spaces}---the category of algebras
for the discrete distribution monad.

\begin{Defn}
  \label{def:1}
  A \emph{finitely supported sub-probability distribution} on a set
  $A$ is a function $\omega \colon A \rightarrow [0,1]$ such that
  $\mathrm{supp}(\omega)$ is finite and $\omega(A) \leqslant 1$. We
  call $\omega$ a \emph{probability distribution} if $\omega(A) = 1$.
\end{Defn}
Here, we write $\mathrm{supp}(\omega)$ for the set
$\{a \in A : \omega(a) > 0\}$ and, for any $B \subseteq A$, write
$\omega(B)$ for $\sum_{b \in B} \omega(b)$. It will often be
convenient to write a sub-probability distribution $\omega$ on $A$ as
a formal convex combination
\begin{equation}\label{eq:28}
  \sum_{a \in \mathrm{supp}(\omega)}\!\! \omega(a) \cdot a
\end{equation}
of elements of $A$; so, for example,
$\omega \colon \{a,b,c,d\} \rightarrow [0,1]$ with
$\omega(a) = \omega(c) = \tfrac 1 3$, $\omega(d) = \tfrac 1 6$ and
$\omega(b) = 0$ could also be written as
$\tfrac 1 3 \cdot a + \tfrac 1 3 \cdot c + \tfrac 1 6 \cdot d$.
\begin{Defn}
  \label{def:2}
  If $\omega$ is a sub-probability distribution on $A$ such that
  $\omega(A) > 0$, then its \emph{normalisation} is the probability
  distribution $\overline \omega$ with $\overline{\omega}(a) =
  \omega(a) / \omega(A)$.
\end{Defn}
Of course, if $\omega \colon A \rightarrow [0,1]$ is everywhere-zero,
then we cannot normalise it. One way of understanding Jacobs'
hypernormalisation~\cite{Jacobs2017Hyper} is as a principled way of
avoiding this singularity. In the definition, and henceforth, we write
$\D A$ for the set of probability distributions on a set $A$.

\begin{Defn}
  \label{def:3}
  Let $A$ be a set and $n \in \mathbb N$. The \emph{$n$-ary
    hypernormalisation} function
  \begin{equation*}
    \N \colon \D(\underbrace{A + \cdots + A}_{n}) \longrightarrow
    \D(\underbrace{\D A + \cdots + \D A}_n)
  \end{equation*}
  is given as follows. For $1 \leqslant i \leqslant n$ and $a \in
  A$, we write $(a, i)$ for the image of $a$
  under the $i$th coproduct injection $A \rightarrow A + \cdots + A$. Each
 $\omega \in \D(A + \cdots + A)$ yields $n$ sub-probability
  distributions $\omega_i$ on $A$ with
  $\omega_i(a) = \omega(a, i)$, and we take
  \begin{equation}\label{eq:29}
    \N(\omega) = \sum_{\substack{1 \leqslant i \leqslant
        n\\\omega_i(A) > 0}} \omega_i(A) \cdot (\overline{\omega_i}, i)\rlap{ .}
  \end{equation}
\end{Defn}
In other words $\N(\omega)$ ``normalises the \emph{non-zero}
sub-probability distributions among $\omega_1, \dots, \omega_n$ and
records the total weights''. Note that this agrees with the
description of hypernormalisation given in the introduction: the only
difference is that there, we wrote $B$ for the finite set
$\{1, \dots, n\}$, and wrote $A \times B$ for $A + \dots + A$, so that
hypernormalisation became a map
$\D(A \times B) \rightarrow \D(\D A \times B)$.


As suggested in Section~8 of~\cite{Jacobs2017Hyper}, one may
generalise the hypernormalisation maps by replacing the $n$ copies of
$A$ with $n$ possibly distinct sets, yielding maps
\begin{equation}\label{eq:43}
\N \colon \D(A_1 + \dots + A_n) \rightarrow \D(\D A_1 + \dots + \D
A_n)
\end{equation}
defined in an entirely analogous manner to before. In this paper it
will be this asymmetric version of hypernormalisation that we use. In
fact, the key features of hypernormalisation are fully alive in the
$n = 2$ case, and so in large part we will concentrate on the binary
hypernormalisation maps
\begin{equation}\label{eq:2}
  \N \colon \D(A+B) \rightarrow \D(\D A + \D B)\rlap{ .}
\end{equation}

Of particular note is the case of~\eqref{eq:2} where $B$ is a
singleton set $1 = \{\ast\}$, so that we have a map
$\N \colon \D(A + 1) \rightarrow \D(\D A + \D1) \cong
\D(\D A + 1)$.
An element of $\D(A+1)$ can be identified with a
sub-probability distribution $\omega$ on $A$, with the one additional
point $\ast$ necessarily being given the weight $1 - \omega(A)$;
likewise, an element of $\D(\D A + 1)$ can be identified with a
sub-probability distribution on $\D A$. Under these identifications,
the action of $\N$ can be described as follows:
\begin{itemize}[itemsep=0.25\baselineskip]
\item If $\omega$ is the zero sub-probability distribution on $A$,
  then $\N(\omega)$ is the zero sub-probability distribution on $\D
  A$;
\item Otherwise, $\N(\omega)$ is the sub-probability distribution on
  $\D A$ which assigns the weight $\omega(A)$ to the single point
  $\overline{\omega}$.
\end{itemize}

\subsection{Convex coproducts}
\label{sec:convex-coproducts}
As explained in the introduction, hypernormalisation is closely bound
up with the \emph{discrete distribution monad} $\mnd{D}$ on the
category of sets. We now recall this monad, and explain how
hypernormalisation is related to coproducts in the category of
$\mnd{D}$-algebras.

\begin{Defn}
  \label{def:4}
  The functor $\D \colon \cat{Set} \rightarrow \cat{Set}$ takes
  $A$ to $\D A$ on objects; while on maps, $\D f \colon \D A
  \rightarrow \D B$ sends $\omega \in \D
  A$ to the \emph{pushforward} $f_\ast(\omega) \in
  \D B$ 
  given by
  \begin{equation}\label{eq:52}
    f_\ast(\omega)(b) = \omega(f^{-1}(b))\rlap{ .}
  \end{equation}
  The unit $\eta \colon 1_\cat{Set} \Rightarrow \D$ and multiplication
  $\mu \colon \D \D \Rightarrow \D$ of the
  discrete distribution monad $\mnd D$ have respective components at a set $A$ given by
  \begin{align*}
    \eta_A \colon A & \,\rightarrow\, \D A & \mu_A \colon \D\D A & \,\rightarrow\, \D A \\
    a & \,\mapsto\, 1 \cdot a & \sum_{1 \leqslant i \leqslant n} \!\!\lambda_i
    \cdot{\omega_i} & \,\mapsto\, \big(\,a \mapsto \!\!\sum_{1 \leqslant i \leqslant n} \!\!\lambda_i 
    \omega_i(a)\,\big)\rlap{ .}
  \end{align*}
  (Note that, in giving $\mu_A$, we have to the left a
  \emph{formal} convex combination of elements of $\D A$, and to the
  right, an \emph{actual} convex combination in $[0,1]$.)
\end{Defn}

In~\cite{Jacobs2017Hyper}, the discrete distribution monad is
discussed in terms of its Kleisli category; our interest here is in
the algebras of this monad, which are sometimes known as
\emph{abstract convex spaces}.

\begin{Defn}
  \label{def:5}
  An \emph{abstract convex space} is a set $A$ endowed with an
  operation
  \begin{equation}
    \label{eq:1}
    \begin{aligned}
      (0,1) \times A \times A & \rightarrow A \\
      (r,a,b) & \mapsto r(a,b)
    \end{aligned}
  \end{equation}
  satisfying the following axioms for all $a,b,c \in A$ and $r,s \in (0,1)$:
  \begin{enumerate}[(i)]
  \item $r(a,a) = a$;
  \item $r(a,b) = r^\ast(b,a)$ (recall we write $r^\ast$ for $1-r$);
  \item $r(s(a,b),c) = (rs)(a, \tfrac{r \cdot s^\ast}{(rs)^\ast}(b,c))$.
  \end{enumerate}
  A \emph{map of convex spaces} from $A$ to $B$ is a function
  $f \colon A \rightarrow B$ such that $f(r(a,b)) = r(fa,fb)$ for all
  $a,b \in A$ and $r \in (0,1)$.
\end{Defn}

If we view the operation~\eqref{eq:1} as an ``abstract convex
combination'' $r(a,b) = r \cdot a + r^\ast \cdot b$, then the axioms
are just what is needed to ensure that this behaves as expected. The
two main classes of examples of abstract convex spaces are:
\begin{itemize}
\item Convex subsets of vector spaces (so convex spaces in the usual
  sense) under the usual convex combination operation; and
\item Meet-semilattices under the operation $r(a,b) = a \wedge b$ for all $r \in (0,1)$.
\end{itemize}

 If we extend the operation of an abstract convex space to
one $[0,1] \times A \times A \rightarrow A$ by defining $1(a,b) = a$
and $0(a,b) = b$, then axioms (i)--(iii) are
still validated for the new edge cases in $\{0,1\}$ wherever this
makes sense (i.e., so long as $rs \neq 1$ in (iii)). This yields the
axiomatisation of abstract convex spaces found in~\cite[\sec 2, Axioms
B1--B3]{Neumann1970On-the-quasivariety}. 
In particular, these axioms ensure that each of the valid ways of
interpreting a formal convex combination
\begin{equation}\label{eq:7}
  \sum_{1 \leqslant i \leqslant n} r_i \cdot{a_i} \ \ \in\ \  \D A
\end{equation}
as an element of $A$ via repeated application of the
operation~\eqref{eq:1} will give the same result, so that we have a
well-defined function $\D A \rightarrow A$. This function endows $A$
with $\mnd{D}$-algebra structure, which is the key step in proving:
\begin{Lemma}\cite[Theorem~4]{Jacobs2010Convexity}.
  \label{lem:1}
  The category $\cat{Conv}$ of abstract convex spaces and convex maps
  is isomorphic over $\cat{Set}$ to the category $\cat{Set}^\mnd{D}$
  of $\mnd{D}$-algebras.
\end{Lemma}
This result justifies us in using expressions of the form~\eqref{eq:7}
to denote an element of an abstract convex space $A$, and we do so
without further comment.

The relation betwen abstract convex spaces and hypernormalisation lies
in the construction of finite coproducts in $\cat{Conv}$. While
coproducts in algebraic categories are usually messy and syntactic,
for abstract convex spaces they are quite
intuitive. 
Given $A,B \in \cat{Conv}$, their coproduct must
certainly contain copies of $A$ and $B$; and must also contain a
formal convex combination $r \cdot a + r^\ast \cdot b$ for each
$a \in A$, $b \in B$ and $r \in (0,1)$. For a general algebraic
theory, this process of iteratively adjoining formal interpretations
for operations would continue, but in this case, it stops here:
\begin{Lemma}
  \label{lem:2}(cf.~\cite[\sec 2]{Jacobs2015States}).
  If $A$ and $B$ are abstract convex spaces, then their coproduct $A
  \star B$ in
  $\cat{Conv}$ is
  the set $A + \big((0,1) \times A \times B\big) + B$, endowed with the
  convex combination operator whose most involved case is
  \begin{equation*}
    r_{A \star B}((s,a,b),(t,a',b')) =
    (rs+rt^\ast,(\tfrac{rs}{rs+r^\ast t})_A(a,a'),
    (\tfrac{rs^\ast}{(rs + r^\ast t)^\ast})_B(b,b'))\rlap{ .}
  \end{equation*}
\end{Lemma}
This formula was obtained by expanding out the formal convex
combination
$r \cdot(s \cdot a + s^\ast \cdot b) + r^\ast \cdot (t \cdot{a'} +
t^\ast \cdot {b'})$, rearranging, and partially evaluating the terms
from $A$ and from $B$. The reader should have no difficulty giving the
remaining, simpler, cases (where one or both arguments of
$r_{A \star B}$ come from $A$ or $B$), and in then proving that the
resulting object is an abstract convex space.

If we write elements $a \in A$,
$(r,a,b) \in (0,1) \times A \times B$ and $b \in B$ in the three summands of $A \star
B$ as, respectively,
\begin{equation*}
  \iota_1(a) \ \text{,} \qquad r \cdot a + r^\ast \cdot b \qquad
  \text{and} \qquad 
  \iota_2(b)\rlap{ ,}
\end{equation*}
then the two coproduct injections are given by
$\iota_1 \colon A \rightarrow A \star B \leftarrow B \colon \iota_2$;
and as for the universal property of coproduct, if
$f \colon A \rightarrow C$ and $g \colon A \rightarrow C$ are convex
maps, then the unique induced map
$\spn{f,g} \colon A \star B \rightarrow C$ sends $\iota_1(a)$ or
$\iota_2(b)$ to $f(a)$ or $g(b)$ respectively, and sends
$r \cdot a + r^\ast \cdot b = r_{A \star B}(a,b)$ to
$r \cdot {f(a)} + r^\ast \cdot{g(b)} = r_C(fa,gb)$.

To draw the link with hypernormalisation, consider the
free-forgetful adjunction
\begin{equation}\label{eq:14}
  \cd{
    {\cat{Conv}} \ar@<-4.5pt>[r]_-{U^\mnd{D}} \ar@{<-}@<4.5pt>[r]^-{F^\mnd{D}} \ar@{}[r]|-{\bot} &
    {\cat{Set}}
  }
\end{equation}
associated to the monad $\mnd{D}$. The left adjoint $F^\mnd{D}$ sends
the set $A$ to the set $\D A$, seen as an abstract convex space under
the convex combination operation given pointwise by the usual one on
$[0,1]$. Being a left adjoint, $F^\mnd{D}$ preserves coproducts, and
so we have for any set $A$ and any $n \in \mathbb{N}$ a bijection of
abstract convex spaces
\begin{equation}\label{eq:6}
  \varphi \colon \D(A + B) \rightarrow {\D A \star \D B}\rlap{ ,}
\end{equation}
which, if we spell it out, we see is really just hypernormalisation:
\begin{Prop}
  \label{prop:1}
  The isomorphism~\eqref{eq:6} is given by
  \begin{equation*}
    \varphi(\omega) \ \ = \ \
    \begin{cases}
       \iota_1(\omega_1) & \text{ if $\omega_1(A) = 1$;}\\
      \iota_2(\omega_2) & \text{ if $\omega_2(B) = 1$;}\\
      \omega_1(A) \cdot{\overline{\omega_1}} + \omega_2(B) \cdot{\overline{\omega_2}} & \text{ otherwise,}
    \end{cases}
  \end{equation*}
  where $\omega_1$ and $\omega_2$ are the
  sub-probability distributions obtained by restricting $\omega$ to $A$ and $B$.
\end{Prop}
\begin{proof}
  $\varphi$ is the extension of the composite function
  \begin{equation}\label{eq:8}
    A + B \xrightarrow{\eta + \eta} \D A + \D B \hookrightarrow \D A
    \star \D B
  \end{equation}
  to a convex map
  $\D(A+B) \rightarrow \D A \star \D B$. Precomposing~\eqref{eq:8}
  with $\iota_1 \colon A \rightarrow A+B$ yields
  \begin{equation*}
    A \xrightarrow{\eta} \D A \xrightarrow{\iota_1} \D A \star \D B
  \end{equation*}
  whence $\varphi$ identifies $\D A \hookrightarrow \D(A+B)$ with
  the left coproduct summand of $\D A \star \D B$. This proves the
  first case of the desired formula; the second is similar.

  Finally, consider $\omega \in \D(A+B)$ which does not factor through
  either $\D A$ or $\D B$. Since both sub-probability distributions
  $\omega_1$ and $\omega_2$ are non-zero, we can form both
  $\overline{\omega_1} \in \D A$ and $\overline{\omega_2} \in \D B$,
  and in these terms we now have
  \begin{equation*}
    \omega = \omega_1(A) \cdot{(\iota_1)_\ast(\overline{\omega_1})} +
    \omega_2(B) \cdot{(\iota_2)_\ast(\overline{\omega_1})}
  \end{equation*}
  in $\D(A+B)$. As~$\varphi$ is a convex map, it follows that
  $\varphi(\omega) = \omega_1(A) \cdot \overline{\omega_1} +
  \omega_2(B) \cdot \overline{\omega_2}$ from the two cases already proved.
\end{proof}

\section{Abstract hypernormalisation}
\label{sec:opmonoidal-comonads}
The map~\eqref{eq:6} of Proposition~\ref{prop:1} is related to the
hypernormalisation map~\eqref{eq:2}, but is not quite the
same. 
In this section, we explain how to derive the latter map from the
former one, and isolate the general structure required for this
derivation: that of a \emph{linear exponential monad}. Guided by this,
we define an abstract notion of hypernormalisation with respect to a
linear exponential monad, show that it has the desired equational
properties, and explain how it may give rise to a version of the
channel-to-abstract channel construction from the introduction.

\subsection{Recapturing hypernormalisation}
\label{sec:hypernormalisation-2}
Towards bridging the gap between~\eqref{eq:2} and~\eqref{eq:6}, we observe that~\eqref{eq:2} is \emph{not} a map of
abstract convex spaces, so that to recapture it, we must necessarily leave the
category $\cat{Conv}$. We do so in an apparently simple-minded
fashion, by considering the category $\cat{Conv}_{\mathrm{arb}}$ whose
objects are abstract convex spaces and whose maps are \emph{arbitrary}
functions.


Now, the binary coproduct $\star$ on $\cat{Conv}$ is part of a
symmetric monoidal structure, whose unit is the empty convex space,
and whose coherence isomorphisms are induced from the universal
properties of finite coproducts. This symmetric monoidal structure
\emph{extends} to $\cat{Conv}_{\mathrm{arb}}$; by this we mean simply
that $\cat{Conv}_{\mathrm{arb}}$ has a symmetric monoidal structure
with respect to which the inclusion
$\cat{Conv} \hookrightarrow \cat{Conv}_{\mathrm{arb}}$ becomes
symmetric strict monoidal. This monoidal structure is (necessarily)
given on objects as before, while on maps $f \colon A \rightarrow A'$
and $g \colon B \rightarrow B'$ of $\cat{Conv}_{\mathrm{arb}}$, the
tensor $f \star g \colon A \star B \rightarrow A' \star B'$ is given
by
\begin{equation}\label{eq:12}
  f \star g = f + \big((0,1) \times f \times g\big) + g\rlap{ ,}
\end{equation}
i.e., exactly the same formula as the definition of $\star$ on maps in
$\cat{Conv}$.

Now suppose we are given abstract convex spaces $A$ and $B$. Using the extended
monoidal structure on $\cat{Conv}_{\mathrm{arb}}$, we obtain a function
\begin{equation}\label{eq:13}
  A \star B \xrightarrow{\eta_A \star \eta_B} \D A \star \D B
  \xrightarrow{\varphi^{-1}} \D(A+B)
\end{equation}
whose second part is the inverse of~\eqref{eq:6} and whose first part
is the tensor~\eqref{eq:12} of the (non-convex) functions $\eta_A$ and
$\eta_B$. Working through the definitions, we see that this sends
elements $\iota_1(a)$ and $\iota_2(b)$ of $A \star B$ to the
distributions $1 \cdot a$ and $1 \cdot b$ on $A + B$ concentrated at a
single point; while an element $r \cdot a + r^\ast \cdot b \in A \star B$ is sent to the
two-point distribution $r \cdot a + r^\ast \cdot b$ on $A + B$. Combining this
description of~\eqref{eq:13} with Proposition~\ref{prop:1}, we
immediately obtain:
\begin{Prop}
  \label{prop:3}
  The hypernormalisation map~\eqref{eq:2} is the composite
  \begin{equation}\label{eq:15}
    \D( A +  B) \xrightarrow{\ \varphi\ } \D A \star \D B
    \xrightarrow{\eta_{\D A} \star \eta_{\D B}} \D \D A \star \D \D B
    \xrightarrow{\varphi^{-1}} \D(\D A + \D B)\rlap{ .}
  \end{equation}
\end{Prop}

Thus, the hypernormalisation map~\eqref{eq:2} arises inevitably from
the isomorphism~\eqref{eq:6} together with the fact that the coproduct
monoidal structure on $\cat{Conv}$ extends to
$\cat{Conv}_\mathrm{arb}$. We now give an explanation of why this extension of monoidal
structure should exist.
To motivate this explanation, observe that the formula~\eqref{eq:12}
for the extended tensor product on $\cat{Conv}_\mathrm{arb}$ works
because the underlying set of $A \star B$ depends only on the
underlying sets of $A$ and $B$, and not on their convex
structure. So could the symmetric monoidal structure $\star$ on
$\cat{Conv}$ be a \emph{lifting} of a symmetric monoidal structure on
$\cat{Set}$? In other words, is there a symmetric monoidal structure
$(\star,0)$ on $\cat{Set}$---which as in the introduction we term the
\emph{convex monoidal structure}---such that
$U^\mnd{D} \colon (\cat{Conv}, \star, 0) \rightarrow (\cat{Set},
\star, 0)$ is strict symmetric monoidal?

In Section~\ref{sec:conv-mono-struct} below, we will see that this is
indeed the case; for the moment, let us see how, assuming this fact,
we can recover the symmetric monoidal structure of
$\cat{Conv}_\mathrm{arb}$. To do this, we consider the evident
factorisation
$\cat{Conv} \rightarrow \cat{Conv}_\mathrm{arb} \rightarrow \cat{Set}$
of $U^\mnd{D}$ through $\cat{Conv}_\mathrm{arb}$, and apply the
following result:
\begin{Lemma}
  \label{lem:5}\cite{Power1989A-general}
  Let $F \colon \E \rightarrow \C$ be a strict symmetric monoidal
  functor between symmetric monoidal categories, and let
  \begin{equation*}
    F = \E \xrightarrow{G} \D \xrightarrow{H} \C
  \end{equation*}
  be a factorisation of the underlying functor $F$ wherein $G$ is
  bijective on objects and $H$ is fully faithful. There is a unique
  symmetric monoidal structure on $\D$ making both $G$ and $H$ strict
  symmetric monoidal.
\end{Lemma}
\begin{proof}
  Define the unit and the tensor on objects in $\D$ to be those of
  $\E$, and define the tensor on maps and the coherence morphisms to
  be those of $\C$.
\end{proof}

At this point, if we still take for granted the existence of the
convex monoidal structure $(\star, 0)$ on $\cat{Set}$, then one final
category-theoretic transformation will allow us to derive the
hypernormalisation maps~\eqref{eq:2} purely in terms of the structure
of the discrete distribution monad. We begin by recalling:

\begin{Defn}
  \label{def:6}
  A monad $\mnd T$ on a symmetric monoidal category $(\C, \otimes, I)$
  is \emph{symmetric opmonoidal} if it comes endowed with a map
  $\upsilon_I \colon TI \rightarrow I$ and maps
  $\upsilon_{XY} \colon T(X \otimes Y) \rightarrow TX \otimes TY$
  for $X,Y \in \C$, subject to seven coherence
  axioms; see, for example,~\cite[Section~7]{Moerdijk2002Monads}.
\end{Defn}
The relevance of this definition for us is captured by:

\begin{Lemma}
  \label{lem:3} \cite[Theorem~7.1]{Moerdijk2002Monads}.
  For any monad $\mnd T$ on a symmetric monoidal category
  $(\C, \otimes, I)$, symmetric opmonoidal monad structures on
  $\mnd{T}$ correspond bijectively to liftings of the symmetric monoidal structure
  of $\C$ to $\C^\mnd{T}$.
\end{Lemma}
\begin{proof}
  Given symmetric opmonoidal structure on $\mnd T$, we define the
  lifted tensor product of $\mnd T$-algebras by
  \begin{equation}\label{eq:60}
    (TX \xrightarrow{x} X) \otimes (TY \xrightarrow{y} Y) \ = \ (T(X
    \otimes Y) \xrightarrow{\upsilon_{XY}} TX \otimes TY
    \xrightarrow{x \otimes y} X \otimes Y)
  \end{equation}
  with as unit the $\mnd{T}$-algebra
  $\upsilon_I \colon TI \rightarrow I$. Conversely, given a lifted
  tensor product on $\mnd T$-algebras, we obtain the opmonoidal
  structure map $\upsilon_{XY}$ as the composite
  \begin{equation*}
    T(X \otimes Y) \xrightarrow{T(\eta_X \otimes \eta_Y)} T(TX \otimes
    TY) \xrightarrow{\ \ \ \theta\ \ \ } TX \otimes TY
  \end{equation*}
  where $\theta$ is the $\mnd T$-algebra structure of  
  $(\mu_X \colon TTX \rightarrow TX) \otimes (\mu_Y \colon TTY
  \rightarrow TY)$, and obtain $\upsilon_I \colon TI \rightarrow I$ as
  the $\mnd T$-algebra structure of the lifted unit.
\end{proof}

Thus, the fact that the coproduct monoidal structure on $\cat{Conv}$
lifts the (assumed) convex monoidal structure $(\star, 0)$ on
$\cat{Set}$ can be re-expressed by saying that the discrete distribution monad
$\mnd D$ on $\cat{Set}$ is symmetric opmonoidal with respect to
$(\star, 0)$. Actually, more is true: $\mnd D$ is a \emph{linear
  exponential monad}.

\begin{Defn}
  \label{def:7}
  A \emph{linear exponential monad} on a symmetric monoidal category
  $(\C, \otimes, I)$ is a symmetric opmonoidal monad $\mnd{T}$ on $\C$
  such that the lifted symmetric monoidal structure on the category of
  algebras $\C^\mnd{T}$ is given by finite coproducts. More precisely,
  we means by this that the lifted unit object $(I, \nu_I)$ should be initial in
  $\C^{\mnd{T}}$; and that, for any pair of algebras $(X, x)$ and
  $(Y,y)$, the cospan
  \begin{equation*}
    (X,x) \xrightarrow{\rho_X} (X, x) \otimes (I, \nu_I) \xrightarrow{1 \otimes !} (X,x) \otimes (Y,y) \xleftarrow{! \otimes 1} (I, \nu_I) \otimes (Y,y) \xleftarrow{\lambda_Y} (Y,y)
  \end{equation*}
  should define a binary coproduct in $\C^\mnd{T}$, where we use $!$
  to denote the unique maps out of the initial object $(I, \nu_I)$.
\end{Defn}

Linear exponential \emph{co}monads originate in the categorical
semantics of linear logic~\cite[Definition~3]{Benton1993Linear} where
they interpret the exponential modality which allows a resource to be copied
freely. Importantly, the co-Kleisli category of a linear exponential
comonad on a symmetric monoidal \emph{closed} category is cartesian
closed; this is a categorical formulation of the translation of
intuitionistic logic into linear logic~\cite[\sec
5]{Girard1987Linear}. Linear exponential comonads also arise in
connection with~\cite{Blute2006Differential}'s differential
categories, which are categories endowed with an abstract notion of
differentiation, encoded by a comonad which in many cases is linear
exponential (note that in this context, the term ``monoidal coalgebra
modality'' is often used rather than ``linear exponential comonad'').

The dual notion of linear exponential \emph{monad} appears both in
linear logic, where it models the de Morgan dual connective $?$ of
$!$, and in the study of \emph{co}differential categories, of which
there are many natural examples; see, for
example,~\cite{Blute2016Derivations}. Furthermore, as we will show in
the next section, a linear exponential monad is \emph{exactly} the
structure one needs for an good abstract notion of
hypernormalisation.

The salience of this last observation is not so much that it
establishes a deep connection between probabilistic structures and
linear logic, but rather that it makes available the well-understood
calculus of reasoning for linear exponential (co)monads, as discussed
in, for example,~\cite[Section~7]{Mellies2009Categorical}
or~\cite{Blute2019Differential}. As we will see, this allows to show
that our abstract notion of hypernormalisation verifies all the
equational axioms one could wish for.
%
%
%
%


\subsection{Abstract hypernormalisation}
\label{sec:abstr-hypern}
Given a symmetric monoidal category $(\C, \otimes, I)$ with finite
coproducts and a linear exponential monad $\mnd T$ on $\C$, we
continue to write
$\varphi \colon T(A+B) \rightarrow TA \otimes TB$ for the map
underlying the $\mnd T$-algebra isomorphism
$F^\mnd{T}(A+B) \rightarrow F^\mnd T(A) \otimes F^\mnd{T}(B)$; more
generally, we write
\begin{equation}\label{eq:26}
  \varphi \colon T(A_1 + \dots + A_n) \rightarrow TA_1 \otimes \dots
  \otimes TA_n
\end{equation}
for the corresponding $n$-ary isomorphism. Note that these
isomorphisms are natural in maps of $\C$, which is to say that all
diagrams of the following form commute:
\begin{equation}\label{eq:51}
  \cd{
    {T(A_1 + \dots + A_n)} \ar[r]^-{\varphi} \ar[d]_{T(f_1 + \dots + f_n)} &
    {TA_1 \otimes \dots \otimes TA_n} \ar[d]^{Tf_1 \otimes \dots \otimes Tf_n} \\
    {T(B_1 + \dots + B_n)} \ar[r]_-{\varphi} &
    {TB_1 \otimes \dots \otimes TB_n}\rlap{ .}
  }
\end{equation}

\begin{Defn}
  \label{def:8}
  Let $(\C, \otimes, I)$ be a symmetric monoidal category with finite
  coproducts, and let $\mnd T$ be a linear exponential monad on $\C$.
  The \emph{binary hypernormalisation map}
  $\N \colon T(A+B) \rightarrow T(TA+TB)$ is the composite
  \begin{equation*}
    T(A + B) \xrightarrow{\ \varphi\ } TA \otimes TB \xrightarrow{\eta_{TA}
      \otimes \eta_{TB}} TTA \otimes TTB \xrightarrow{\varphi^{-1}} T(TA
    + TB)\rlap{ .}
  \end{equation*}
  More generally, given objects $A_1, \dots, A_n$, the \emph{$n$-ary
    hypernormalisation map} $\N \colon T(\Sigma_i A_i) \rightarrow T(\Sigma_i TA_i)$ is the composite
  \begin{equation}\label{eq:23}
    T(\sum_i A_i) \xrightarrow{\ \varphi\ } \bigotimes_i TA_i
    \xrightarrow{\otimes_i \eta_{TA_i}} \bigotimes_i TTA_i
    \xrightarrow{\varphi^{-1}} T(\sum_i TA_i)\rlap{ .}
  \end{equation}
\end{Defn}
The leading example, as we confirm in
Section~\ref{sec:conv-mono-struct}, is Jacobs' original
hypernormalisation, which arises on taking $(\C, \otimes, I)$ to be
$(\cat{Set}, \star, 0)$ and $\mnd{T} = \mnd{D}$. However, as we will
see in Section~\ref{sec:examples}, there are many other interesting
examples of abstract hypernormalisation, including continuous
probability monads on suitable categories, and examples related to
continuous functions on streams.

Before turning to these examples, we investigate the degree to which
our abstract hypernormalisation inherits the good equational
properties of Jacobs' original definition. In the statement and proof
of the following result,
we write
$\spn{f_i}_{i \in I} \colon \sum_i A_i \rightarrow B$ to denote a
copairing of maps $f_i \colon A_i \rightarrow B$ out of a coproduct.

\begin{Prop}
  \label{prop:9}
  Let $\mnd T$ be a linear exponential monad on the symmetric monoidal
  category $(\C, \otimes, I)$. The hypernormalisation
  maps~\eqref{eq:23} satisfy the conditions expressed by the
  commutativity of the following diagrams:
  \begin{enumerate}[(1)]
    \item Hypernormalisation has a left inverse:
      \begin{equation*}
        \cd[@C+1em]{
          T(\Sigma_i A_i) \ar[r]^-{\N} \ar@{=}[d]_-{}&
          T(\Sigma_i TA_i) \ar[d]^-{T\spn{T\iota_i}_i} \\
          T(\Sigma_i A_i) & TT(\Sigma_i A_i) \ar[l]^-{\mu_{\Sigma_i A_i}}
        }
      \end{equation*}
    \item Hypernormalisation is idempotent:
      \begin{equation*}
        \cd[@C+1.5em]{
          T(\Sigma_i A_i) \ar[r]^-{\N} \ar[d]_-{\N} &
          T(\Sigma_i TA_i) \ar[d]^-{\N} \\
          T(\Sigma_i TA_i) \ar[r]^-{T(\Sigma_i \eta_{TA_i})} &
          T(\Sigma_i TTA_i)
        }
      \end{equation*}
    \item Hypernormalisation is natural in maps $f_i \colon A_i \rightarrow
      B_i$:
      \begin{equation*}
        \cd{
          T(\Sigma_i A_i) \ar[r]^-{\N} \ar[d]_-{T(\Sigma_i f_i) }& T(\Sigma_i TA_i)
          \ar[d]^-{T(\Sigma_i Tf_i)} \\
          T(\Sigma_i B_i) \ar[r]^-{\N} & T(\Sigma_i TB_i)
        }
      \end{equation*}
    and in Kleisli maps
      $f_i \colon A_i \rightarrow TB_i$:      \begin{equation*}
        \cd{
          T(\Sigma_i A_i) \ar[r]^-{\N} \ar[d]_-{T(\Sigma_i f_i) }& T(\Sigma_i TA_i)
          \ar[r]^-{T(\Sigma_i Tf_i)} & T(\Sigma_i TTB_i)
          \ar[d]_-{T(\Sigma_i \mu_{B_i})} \\
          T(\Sigma_i TB_i) \ar[r]^-{\N} & T(\Sigma_i T^2B_i)
          \ar[r]^-{T(\Sigma_i \mu_{B_i})} & T(\Sigma_i TB_i)\rlap{ .}
        }
      \end{equation*}
    \end{enumerate}
\end{Prop}
\begin{proof}
  To prove (1), we first claim that each diagram as to the left below
  commutes. This is a general fact about linear exponential
  monads---see, for
  example~\cite[Section~7]{Mellies2009Categorical}---but we
  include the proof for the sake of self-containedness.
  \begin{equation}\label{eq:25}
    \cd{
      T(\Sigma_i TA_i) \ar[d]_-{T\spn{T\iota_i}_i} \ar[r]^-{\varphi} & \otimes_i TTA_i
      \ar[dd]^-{\otimes_i \mu_{A_i}} \\
      TT(\Sigma_i A_i) \ar[d]_-{\mu_{\Sigma_i A_i}} \\
      T(\Sigma_i A_i) \ar[r]^-{\varphi} & \otimes_i TA_i
    } \quad
   \cd{
     TTA_i \ar[d]_-{TT\iota_i} \ar[rr]^-{\jmath_i}
     \ar[dr]^-{\mu_{TA_i}} & & \otimes_i TTA_i \ar[dd]^-{\otimes_i
        \mu_{A_i}} \\
      TT(\Sigma_i A_i) \ar[d]_-{\mu_{\Sigma_i A_i}} & TA_i
      \ar[dl]_-{T\iota_i}\ar[dr]^-{\jmath_i}\\
      T(\Sigma_i A_i) \ar[rr]^-{\varphi} & & \otimes_i TA_i\rlap{ .}
    }
  \end{equation}
  
  Note that both paths are $\mnd T$-algebra maps
  $F^\mnd{T}(\Sigma_i TA_i) \rightarrow \otimes_i F^\mnd{T}A_i$ with
  as domain a coproduct of the $\mnd T$-algebras $F^\mnd{T} A_i$. So
  it suffices to show commutativity on precomposing by a
  coproduct coprojection
  $T\iota_i \colon TTA_i \rightarrow T(\Sigma_i TA_i)$. This means
  showing the outside of the diagram to the right above commutes, wherein we
  write $\jmath_i$ for a coproduct coprojection
  $X_i \rightarrow \otimes_i X_i$ in the category of
  $\mnd T$-algebras. But the bottom triangle commutes
  by definition of $\varphi$, the left region by naturality of $\mu$
  and the right region by naturality of the coproduct coprojections
  $\jmath$.

  Now commutativity in~\eqref{eq:25} yields commutativity in the right
  part of:
  \begin{equation*}
    \cd{
      T(\Sigma_i A_i) \ar[d]_-{\varphi} \ar[r]^-{\N} & T(\Sigma_i TA_i) \ar[r]^-{T\spn{T\iota_i}_i} \ar@{<-}[d]^-{\varphi^{-1}} & 
      TT(\Sigma_i A_i) \ar[r]^-{\mu_{\Sigma_i A_i}} &
      T(\Sigma_i A_i) \ar@{<-}[d]^-{\varphi^{-1}} \\
      \otimes_i TA_i \ar[r]^-{\otimes_i \eta_{TA_i}} & \otimes_i TTA_i 
      \ar[rr]^-{\otimes_i \mu_{A_i}}  & & \otimes_i TA_i
    }
  \end{equation*}
  whose left part commutes by definition of $\N$. So the outside
  commutes; now by the monad axioms for $\mnd T$, the lower composite is
  the identity, whence also the upper
  one as required for (1).

  Turning to (2),
  we observe that pre-composing by $\varphi^{-1}$ and post-composing
  by $\varphi$ yields the square
  \begin{equation*}
    \cd[@C+1.5em]{
      {\otimes_i TA_i} \ar[r]^-{\otimes_i \eta_{TA_i}}
      \ar[d]_{\otimes_i \eta_{TA_i}} &
      {\otimes_i T^2A_i} \ar[d]^{\otimes_i \eta_{T^2A_i}} \\
      {\otimes_i T^2A_i} \ar[r]^-{\otimes_i T\eta_{TA_i}} &
      {\otimes_i T^3A_i}
    }
  \end{equation*}
  which commutes by functoriality of $\otimes$ and naturality of
  $\eta$. Finally, for (3), commutativity of the first diagram is
  clear from the naturality~\eqref{eq:51} of the maps
  $\varphi \colon T(\Sigma_i A_i) \rightarrow \otimes_i TA_i$ in the
  $A_i$, the functoriality of $\otimes_i$, and the naturality of the
  unit $\eta \colon 1 \Rightarrow T$. Commutativity of the second
  diagram follows trivially from the first after postcomposing by
  $T(\Sigma_i \mu_{B_i})$.
\end{proof}

The preceding conditions generalise ones appearing
in~\cite[Lemma~5]{Jacobs2017Hyper}. Our (2) and (3) correspond exactly
to its (3) and (5), while our (1) corresponds either to the right
diagram of its (2) or to its (4). We have no correlate of the right
diagram of part (1) of \cite[Lemma~5]{Jacobs2017Hyper}, since it uses
the canonical \emph{strength} of the discrete distribution monad
$\mnd{D}$ with respect to the cartesian monoidal structure of
$\cat{Set}$, and it is not clear what this should be replaced with in
general. This leaves only the left diagrams appearing in (1) and (2)
of \cite[Lemma~5]{Jacobs2017Hyper}. Interestingly, while these make
sense in our setting, they do not hold without additional assumptions.
For the left diagram of (2), this condition is:

\begin{Defn}
  \label{def:18}
  A monad $\mnd T$ on a category $\C$ with a terminal object $1$ is
  \emph{affine} if the unique map $T1 \rightarrow 1$ is
  invertible (necessarily with inverse $\eta_1 \colon 1 \rightarrow T1$).
\end{Defn}

\begin{Prop}
  \label{prop:12}
  Let $\mnd T$ be an affine linear exponential monad on the symmetric
  monoidal category $(\C, \otimes, I)$ with terminal object $1$. The
  hypernormalisation maps~\eqref{eq:23} satisfy the additional
  condition that:
  \begin{enumerate}\addtocounter{enumi}{3}
  \item Destroying the output structure destroys hypernormalisation:
    \begin{equation*}
      \cd[@C-1em]{
        {T(\Sigma_i A_i)} \ar[rr]^-{\N} \ar[dr]_-{T(\Sigma_i !)} &&
        {T(\Sigma_i TA_i)} \ar[dl]^-{T(\Sigma_i !)}\rlap{ .}\\ &
        {T(\Sigma_i 1)}
      }
    \end{equation*}
  \end{enumerate}
\end{Prop}
\begin{proof}
  We may precompose (4) by
  $\varphi^{-1} \colon \otimes_i TA_i \rightarrow T(\Sigma_i A_i)$,
  postcompose by
  $\varphi \colon T(\Sigma_i 1) \rightarrow \otimes_i T1$, and rewrite
  using the definition of $\N$ and the naturality~\eqref{eq:51} to
  obtain the triangle to the left in:
  \begin{equation*}
    \cd[@!C@C-2em]{
      {\otimes_i TA_i} \ar[rr]^-{\otimes_i \eta_{TA_i}}
      \ar[dr]_-{\otimes_i T!} &&
      {\otimes_i TTA_i} \ar[dl]^-{\otimes_i T!} & &
      {\otimes_i TA_i} \ar[rr]^-{\otimes_i !}
      \ar[dr]_-{\otimes_i T!} &&
      {\otimes_i 1} \ar[dl]^-{\otimes_i \eta_1}\\ & 
      {\otimes_i T1} & & & &
      {\otimes_i T1}
    }
  \end{equation*}
  whose commutativity is equivalent to that of (4). But by naturality
  of $\eta$, this triangle is equally the triangle on the right above,
  which commutes since postcomposing by the invertible map
  $\otimes_i ! \colon \otimes_i T1 \rightarrow \otimes_i 1$ yields
  along both sides the map
  $\otimes_i ! \colon \otimes_i TA_i \rightarrow \otimes_i 1$.
\end{proof}
  
Finally, we consider what is necessary for the left diagram of part
(1) of \cite[Lemma~5]{Jacobs2017Hyper} to commute in our setting.
\begin{Defn}
  \label{def:12}
  A symmetric monoidal category $(\C, \otimes, I)$ is said to be
  \emph{co-affine} if its unit object is initial.
\end{Defn}
So, for example, the convex monoidal structure and the cocartesian
monoidal structure on $\cat{Set}$ are co-affine, while the cartesian
monoidal structure is not so. The point of this extra condition is
that it allows us to prove:
\begin{Lemma}
  \label{lem:7}
  Let $\mnd T$ be a linear exponential monad on the symmetric monoidal
  co-affine category $(\C, \otimes, I)$. Finite coproduct
  coprojections $\jmath_i \colon X_i \rightarrow \otimes_i X_i$ in the
  category of $\mnd T$-algebras are natural with respect to
  \emph{arbitrary} maps of $\C$.
\end{Lemma}
\begin{proof}
  Since non-empty finite coproducts can be constructed from binary
  ones, it suffices to prove the binary case. Given $\mnd T$-algebras
  $(X,x)$ and $(Y,y)$, we know from Definition~\ref{def:7} that the
  coproduct coprojection $(X,x) \rightarrow (X,x) \otimes (Y,y)$
  is given by the composite
  \begin{equation*}
    (X,x) \xrightarrow{\rho_X} (X,x) \otimes (I,\nu_I) \xrightarrow{1 \otimes !} (X, x) \otimes (Y,y)\rlap{ ,}
  \end{equation*}
  where $! \colon (I,\nu_I) \rightarrow (Y,y)$ is the unique map of
  $\mnd T$-algebras induced by the initiality of $(I, \nu_I)$ in
  $\C^{\mnd{T}}$. By co-affineness, the underlying map in $\C$ of this composite is
  \begin{equation*}
    X \xrightarrow{\rho_X} X \otimes I \xrightarrow{1 \otimes !} X \otimes Y\rlap{ ,}
  \end{equation*}
  where $!$ is the unique map out of the initial object $I \in \C$.
  Given this description, the desired naturality with respect to
  arbitrary maps of $\C$ is now immediate.
\end{proof}
\begin{Prop}
  \label{prop:10}
  Let $\mnd T$ be a linear exponential monad on the symmetric monoidal
  co-affine category $(\C, \otimes, I)$. The hypernormalisation
  maps~\eqref{eq:23} satisfy the additional condition that:
  \begin{enumerate}\addtocounter{enumi}{4}
  \item Normalising trivial input gives trivial output:
    \begin{equation*}
      \cd{
        {TA_i} \ar[r]^-{T\iota_i} \ar[d]_{\iota_i}
        & 
        {T(\Sigma_i A_i)} \ar[d]^{\N} \\ 
        {\Sigma_i TA_i} \ar[r]_-{\eta_{\Sigma_i TA_i}} &
        {T(\Sigma_i TA_i)}\rlap{ .}
      }
    \end{equation*}
  \end{enumerate}
\end{Prop}
\begin{proof}
  By definition of $\N$ and naturality of $\eta$, this is equally to
  show that the diagram below left commutes. Since $\varphi$ is the
  underlying map of the unique comparison between the coproducts
  $F^\mnd{T}(\Sigma_i A_i)$ and $\otimes_i F^\mnd{T}A_i$ in
  $\C^\mnd{T}$, it in particular commutes with the coproduct
  coprojections, so that this diagram is equally the one below right,
  which commutes by Lemma~\ref{lem:7}.
  \begin{equation*}
    \cd{
      {TA_i} \ar[r]^-{T\iota_i} \ar[d]_{\eta_{TA_i}} & 
      {T(\Sigma_i A_i)} \ar[r]^{\varphi} &
      \otimes_i TA_i \ar[d]^-{\otimes_i \eta_{TA_i}} & 
      {TA_i} \ar[r]^-{\jmath_i} \ar[d]_{\eta_{TA_i}} & 
      \otimes_i TA_i \ar[d]^-{\otimes_i \eta_{TA_i}}\\ 
      {TTA_i} \ar[r]^-{T\iota_i} &
      {T(\Sigma_i TA_i)} \ar[r]^-{\varphi} & \otimes_i TTA_i &
            {TTA_i} \ar[r]^-{\jmath_i} & \otimes_i TTA_i
    }\qedhere
  \end{equation*}
\end{proof}

\subsection{Channel abstraction}
\label{sec:channel-abstraction}

In this short section, we explain how, in our abstract setting,
hypernormalisation can be used to build an analogue of the
channel-to-abstract-channel construction of~\cite{McIver2014Abstract}
described in the introduction. We are motivated in doing this by the
examples of hypernormalisation for \emph{continuous} probability
monads described in
Sections~\ref{sec:expectation-monad}--\ref{sec:kantorovich-monad}
below; thus, in what follows, the reader should keep in mind the
interpretation that $\C$ is some category of ``spaces'', and that
$\mnd{T}$ is a monad of ``distributions'' on $\C$.

Let us recap what the construction should do. The input data, a
channel, is simply a map $P \colon A \rightarrow TB$, which we think
of as giving probabilities that a private input in $A$ will give rise
to a public output in $B$. The output data, the associated abstract
channel, is a map $P^r \colon TA \rightarrow TTA$, thought of as
giving the probabilities that a given prior distribution on $A$ should
update to a given posterior distribution on $A$ via conditioning on an
observed output in $B$.

Following Jacobs' lead, we will build $P^r$ from $P$ via a
composite~\eqref{eq:36}. The main difficulty comes in finding an
analogue of the first map $\tilde P$ therein. This map, we recall, was
itself a composite~\eqref{eq:50}, one of whose terms is the canonical
cartesian \emph{strength} of the finite discrete distribution monad
$\mnd{D}$. We already remarked in the previous section that it was not
clear what to replace this with in general, so here we adopt the most
simple-minded approach that is compatible with our examples: we simply
assume that, again, our monad $\mnd T$ has a cartesian strength.

This is enough to generalise the first map in~\eqref{eq:36}; however,
there is still a small problem with the second map, which should be a
map $T(A \times B) \rightarrow T(TA \times B)$ given by
hypernormalisation. In the motivating example, this was unproblematic:
we could use the fact that $B$ was a finite set $\{1, \dots, n\}$ to
express $A \times B$ as an $n$-fold coproduct $A + \dots + A$, and
then apply $n$-ary hypernormalisation. In our general context, we can
do something similar so long as finite coproducts distribute over
finite products in $\C$, and we assume that $B$ is an $n$-ary
coproduct $1 + \dots + 1$; then by distributivity we have
$A \times B \cong A + \dots + A$ and can proceed as before.
The above discussion thus justifies giving:

\begin{Defn}
  \label{def:28}
  Let $(\C, \otimes, I)$ be a symmetric monoidal category with finite
  products and distributive finite coproducts. Let $\mnd T$ be a
  linear exponential monad on $\C$ endowed with a cartesian strength.
  Given a map $P \colon A \rightarrow TB$, where $B = 1 + \dots + 1$
  is an $n$-fold coproduct of the terminal object, we define the
  \emph{associated abstract channel} to be the map
  $P^r \colon TA \rightarrow TTA$ given by
  \begin{equation*}
    TA \xrightarrow{\tilde P} T(A \times B) \xrightarrow{\N} T(TA
    \times B) \xrightarrow{T(\pi_1)} TTA
  \end{equation*}
  wherein the first map is the composite
  \begin{equation*}
    T A \xrightarrow{T(1, P)} T(A \times TB)
    \xrightarrow{\mnd{T}(\mathsf{str})} TT(A \times B) \xrightarrow{\mu} T(A
    \times B)\rlap{ ,}
  \end{equation*}
  and the second map is the composite of the $n$-ary
  hypernormalisation map $T(A + \dots
  + A) \rightarrow T(TA + \dots + TA)$ with  isomorphisms $T(A
  \times B) \cong T(A + \dots + A)$ and $T(TA + \dots + TA) \cong T(TA \times B)$.
\end{Defn}

The restriction we impose on the form of $B$ above is a real one; in
our examples, it means that our channel $P \colon A \rightarrow TB$
involves a \emph{continuous} space of hidden inputs but only a
\emph{finite discrete} space of observable outputs. While this is
already progress, allowing an arbitrary observation space $B$ would
require something, more general than hypernormalisation, which gave a
smooth categorical treatment of disintegration for probability
measures on a product space.


\section{Tricocycloids and the convex monoidal structure}
\label{sec:conv-mono-struct}

In this section, we complete our description of the convex monoidal
structure $(\star, 0)$ on $\cat{Set}$ with respect to which the
discrete distribution monad is linear exponential, and by doing so
exhibit Jacobs' hypernormalisation as a particular instance of our
abstract hypernormalisation.

The aspects of the convex monoidal structure we have not yet
discussed are its unit, associativity and symmetry constraints. Given
that it should lift to the coproduct monoidal structure on
$\cat{Conv}$, we can read off these constraints from the corresponding
ones for coproducts in $\cat{Conv}$. However, there is still work to
do: we must show the maps involved can be defined in a way that does not
depend on any convex structure, but only on the underlying sets.

We clarify the combinatorics involved in this by using a notion from
quantum algebra known as a \emph{tricocycloid}. We begin this section
by explaining how tricocycloids give rise to monoidal structures, and
how they relate to operads in the sense of~\cite{May1972The-geometry};
we then exhibit a tricocycloid in $\cat{Set}$ which will allow us to
construct the desired convex monoidal
structure. 

\subsection{Tricocycloids}
\label{sec:tricocycloids}
Although our applications will primarily be in the category of sets,
the construction we are about to give naturally exists in a more
general setting. Rather than just the cartesian monoidal category of
sets, it starts from a symmetric monoidal category $(\C, \otimes, I)$
with finite \emph{distributive} coproducts---i.e., finite coproducts
that are preserved by tensor in each variable. For simplicity, we
write $\otimes$ as if it were \emph{strictly} associative, and for
brevity, we may denote tensor by mere juxtaposition. We can now ask:
given an object $H \in \C$---which in the motivating case will be the
set $(0,1$)---under what circumstances is there a symmetric monoidal
structure $(\star, 0)$ on $\C$ with unit the initial object, and
tensor
\begin{equation}\label{eq:19}
  A \star B \defeq A \,+\, H \otimes A \otimes B  \,+\, B \qquad \text?
\end{equation}

First let us consider what is necessary to get a \emph{monoidal}
structure. The unit constraints
$A \star 0 \rightarrow A$ and
$0 \star A \rightarrow A$ are easy; we have canonical
isomorphisms
\begin{equation}\label{eq:17}
  A + H A 0 + 0 \xrightarrow{\cong} A + 0 + 0
  \xrightarrow{\cong} A \quad
  0 + H  0 B + B \xrightarrow{\cong} 0 + 0 + B
  \xrightarrow{\cong} B
\end{equation}
using the preservation of the initial object by
tensor on each side. The
associativity constraint
$(A \star B) \star C \rightarrow A \star (B \star C)$ is
more interesting; it involves a map
\begin{equation*}
  (A + HAB + B) + H (A + H A
  B + B) C + C \rightarrow A + H  A (B + H
  B  C + C) + B + H  B C + C
\end{equation*}
which, 
since tensor preserves binary coproducts in each variable, is equally
a map
\begingroup\makeatletter\def\f@size{9.5}\check@mathfonts
\begin{equation*}
  A + HAB + B + HAC + HHA
    BC + HBC + C \ \xrightarrow{\ }\  A + H  A B + HAH
  B  C + HAC + B + H  B C + C\ \text.
\end{equation*}
\endgroup

Now the coherence axioms relating the associativity and the unit
constraints force this map to take the summands
$A, B, C, HAC, HAC, HBC$ of the domain to the corresponding summands
of the codomain via identity maps, so leaving only the $HHABC$-summand
of the domain unaccounted for. Though we are not forced to, it would
be most natural to map this summand to the $HAHBC$-summand of the
codomain via a composite
\begin{equation}\label{eq:16}
  H H A B C \xrightarrow{v  1
     1 1} H H A B C
  \xrightarrow{1 \sigma 1 1} H A H B C\rlap{ ,}
\end{equation}
where here $v \colon HH \rightarrow HH$ is some fixed invertible map,
and $\sigma$ is the symmetry.

At this point, we have all the data of a monoidal structure,
satisfying all the axioms except perhaps for the Mac Lane pentagon
axiom, which equates the two arrows
$((A \star B) \star C) \star D \rightrightarrows A \star (B \star (C
\star D))$ constructible from the associativity constraint cells. If
we expand out the definitions, we find that this equality is automatic
on most summands of the domain; the only non-trivial case to be
checked is the equality of the two morphisms
$HHHABCD \rightrightarrows HAHBHCD$ given by the respective string
diagrams (read from top-to-bottom):
\begin{equation}\label{eq:10}
\vcenter{\hbox{\begin{tikzpicture}[y=0.70pt, x=0.70pt, yscale=-1.000000, xscale=1.000000, inner sep=0pt, outer sep=0pt]
\path[draw=black,line join=miter,line cap=butt,even odd rule,line width=0.800pt]
  (80.0000,922.3622) .. controls (80.3536,930.1404) and (84.9497,927.8414) ..
  (86.2500,939.8622)  node[above=0.15cm,at start] {$\scriptstyle H$};
\path[draw=black,line join=miter,line cap=butt,even odd rule,line width=0.800pt]
  (180.0000,922.3622) -- (180.0000,1032.3622)  node[above=0.15cm,at start] {$\scriptstyle C$};
\path[draw=black,line join=miter,line cap=butt,even odd rule,line width=0.800pt]
  (200.0000,922.3622) -- (200.0000,1032.3622)  node[above=0.15cm,at start] {$\scriptstyle D$};
\path[draw=black,line join=miter,line cap=butt,even odd rule,line width=0.800pt]
  (99.5000,922.3622) .. controls (99.8536,930.8475) and (93.1360,929.6091) ..
  (93.2500,939.8622)  node[above=0.15cm,at start] {$\scriptstyle H$};
\path[draw=black,line join=miter,line cap=butt,even odd rule,line width=0.800pt]
  (87.1036,989.2795) .. controls (86.0104,974.0294) and (86.3640,959.5338) ..
  (86.7500,944.8622);
\path[draw=black,line join=miter,line cap=butt,even odd rule,line width=0.800pt]
  (120.0000,1032.3622) .. controls (120.0000,1015.0381) and (93.2825,1009.8535)
  .. (93.0429,995.0043) (120.0000,922.3622) .. controls (119.9180,931.7781) and
  (115.2634,941.9923) .. (109.8332,951.8389)   node[above=0.15cm,at start] {$\scriptstyle H$}(107.3183,956.2965) .. controls
  (100.0698,968.9112) and (92.4856,980.6059) ..
  (92.8964,988.8224)(160.0000,1032.3622) .. controls (160.0000,1013.0884) and
  (162.8870,993.5084) .. (152.1649,973.3740) .. controls (142.5218,955.2658) and
  (101.9291,959.1115) .. (93.2500,944.8622)(160.0000,922.3622) .. controls
  (160.5458,936.3702) and (154.7665,948.7987) ..
  (149.2677,965.9454)   node[above=0.15cm,at start] {$\scriptstyle B$}(147.6286,971.2526) .. controls (143.0957,986.5611) and
  (139.2331,1005.6727) .. (140.0000,1032.3622) (140.0000,922.3622) .. controls
  (140.2460,929.6188) and (135.4801,942.3575) ..
  (129.1596,956.8534)   node[above=0.15cm,at start] {$\scriptstyle A$}(127.1187,961.4791) .. controls (120.0558,977.3210) and
  (111.6326,994.7224) .. (106.0683,1009.1376)(104.2396,1014.0670) .. controls
  (101.6332,1021.4237) and (100.0000,1027.7591) ..
  (100.0000,1032.3622);
\path[draw=black,line join=miter,line cap=butt,even odd rule,line width=0.800pt]
  (80.0000,1032.3622) .. controls (79.6464,1012.9168) and (86.5784,1011.1006) ..
  (87.5251,994.8371);
\path[fill=black] (89.0762,941.7743) node[circle, draw, line width=0.65pt,
  minimum width=4.5mm, fill=white, inner sep=0.25mm] (text4238) {$\scriptstyle v$};
\path[fill=black] (90.7404,991.5806) node[circle, draw, line width=0.65pt,
  minimum width=4.5mm, fill=white, inner sep=0.25mm] (text4238-3) {$\scriptstyle v$};
\end{tikzpicture}}} \qquad \quad \text{and} \quad \qquad \vcenter{\hbox{
\begin{tikzpicture}[y=0.70pt, x=0.70pt, yscale=-1.000000, xscale=1.000000, inner sep=0pt, outer sep=0pt]
\path[draw=black,line join=miter,line cap=butt,even odd rule,line width=0.800pt]
  (369.4454,922.5112) -- (369.4454,1032.5112)   node[above=0.15cm,at start] {$\scriptstyle D$};
\path[draw=black,line join=miter,line cap=butt,even odd rule,line width=0.800pt]
  (249.4454,922.5112) .. controls (249.7990,947.9670) and (258.9914,950.9445) ..
  (256.8701,969.3293)   node[above=0.15cm,at start] {$\scriptstyle H$};
\path[draw=black,line join=miter,line cap=butt,even odd rule,line width=0.800pt]
  (262.7279,969.6827) .. controls (262.7279,956.6013) and (276.8701,955.1534) ..
  (276.4026,943.8396);
\path[draw=black,line join=miter,line cap=butt,even odd rule,line width=0.800pt]
  (256.2635,973.9254) -- (249.4454,1032.5112);
\path[draw=black,line join=miter,line cap=butt,even odd rule,line width=0.800pt]
  (297.2236,999.3292) .. controls (298.2843,987.2078) and (274.6482,991.1490) ..
  (262.6274,974.1784)(303.6881,999.6827) .. controls (304.0416,982.3587) and
  (302.1734,960.8101) .. (281.4277,944.1931)(309.4454,922.5112) .. controls
  (310.0516,939.5803) and (303.1707,946.7184) ..
  (294.8843,954.7341)   node[above=0.15cm,at start] {$\scriptstyle A$}(291.2525,958.2668) .. controls (288.7652,960.7323) and
  (286.2329,963.4026) .. (283.8046,966.5423) .. controls (279.9777,971.4903) and
  (277.2233,976.5767) .. (275.2302,981.8016)(273.6151,986.6038) .. controls
  (269.4163,1000.9012) and (270.0525,1016.2030) .. (269.4454,1032.5112);
\path[draw=black,line join=miter,line cap=butt,even odd rule,line width=0.800pt]
  (297.6920,1004.9860) .. controls (298.2843,1017.2078) and (289.4454,1015.8942)
  .. (289.4454,1032.5112);
\path[draw=black,line join=miter,line cap=butt,even odd rule,line width=0.800pt]
  (303.3488,1006.0467) .. controls (303.3488,1018.7746) and (329.7990,1018.7226)
  .. (329.4454,1032.5112)(329.4454,922.5111) .. controls (330.0044,938.1629) and
  (327.1235,989.4160) .. (318.0320,1016.8382)   node[above=0.15cm,at start] {$\scriptstyle B$}(316.3025,1021.5665) .. controls
  (314.2904,1026.5267) and (312.0126,1030.3434) .. (309.4454,1032.5112);
\path[draw=black,line join=miter,line cap=butt,even odd rule,line width=0.800pt]
  (349.4454,922.5112) -- (349.4454,1032.5112)   node[above=0.15cm,at start] {$\scriptstyle C$};
\path[draw=black,line join=miter,line cap=butt,even odd rule,line width=0.800pt]
  (269.4454,922.5112) .. controls (270.1526,932.7642) and (276.1630,929.7581) ..
  (275.6954,940.0112)   node[above=0.15cm,at start] {$\scriptstyle H$};
\path[draw=black,line join=miter,line cap=butt,even odd rule,line width=0.800pt]
  (288.9454,922.5112) .. controls (288.9454,929.2287) and (282.2279,928.6974) ..
  (282.6954,940.0112)   node[above=0.15cm,at start] {$\scriptstyle H$};
\path[fill=black] (278.7716,941.1231) node[circle, draw, line width=0.65pt,
  minimum width=4.5mm, fill=white, inner sep=0.25mm] (text4238-3-9) {$\scriptstyle v$};
\path[fill=black] (259.5928,971.5634) node[circle, draw, line width=0.65pt,
  minimum width=4.5mm, fill=white, inner sep=0.25mm] (text4238-3-8) {$\scriptstyle v$};
\path[fill=black] (299.5443,1001.6154) node[circle, draw, line width=0.65pt,
  minimum width=4.5mm, fill=white, inner sep=0.25mm] (text4238-3-3) {$\scriptstyle v$};
\end{tikzpicture}}} \rlap{\ \ \  .}
\end{equation}
By examining the strings, a \emph{sufficient} condition for this
equality to hold is the equality of the two maps $HHH \rightrightarrows HHH$
represented by the string diagrams
\begin{equation}\label{eq:18}
\vcenter{\hbox{  \begin{tikzpicture}[y=0.70pt, x=0.70pt, yscale=-1.000000, xscale=1.000000, inner sep=0pt, outer sep=0pt]
\path[draw=black,line join=miter,line cap=butt,even odd rule,line width=0.800pt]
  (270.0000,852.3622) .. controls (270.2500,860.1122) and (262.2500,861.4872) ..
  (262.2500,868.9872);
\path[draw=black,line join=miter,line cap=butt,even odd rule,line width=0.800pt]
  (290.0000,852.3622) .. controls (290.0000,870.7640) and (285.7407,884.6419) ..
  (280.4352,895.2374)(277.6012,900.4535) .. controls (272.0268,909.9543) and
  (265.9621,916.3844) .. (262.6250,920.9872)(262.2500,873.6122) .. controls
  (283.3182,899.5898) and (290.6435,918.7024) .. (290.0000,942.3622);
\path[draw=black,line join=miter,line cap=butt,even odd rule,line width=0.800pt]
  (250.0000,852.3622) .. controls (250.7500,861.6122) and (258.6250,862.3622) ..
  (258.8750,869.3622);
\path[draw=black,line join=miter,line cap=butt,even odd rule,line width=0.800pt]
  (258.2500,921.1122) .. controls (259.0000,904.6122) and (251.6903,910.1948) ..
  (251.6099,896.6234) .. controls (251.5295,883.0520) and (258.9772,884.6592) ..
  (258.6250,873.2372);
\path[draw=black,line join=miter,line cap=butt,even odd rule,line width=0.800pt]
  (250.0000,942.3622) .. controls (250.2500,933.8622) and (258.5000,932.8622) ..
  (258.5000,925.6122);
\path[draw=black,line join=miter,line cap=butt,even odd rule,line width=0.800pt]
  (262.0000,925.3622) .. controls (262.5000,933.3622) and (270.0000,933.1122) ..
  (270.0000,942.3622);
\path[fill=black] (260.1172,923.1932) node[circle, draw, line width=0.65pt,
  minimum width=4.5mm, fill=white, inner sep=0.25mm] (text4238-3-3-4) {$\scriptstyle v$};
\path[fill=black] (259.7422,871.4432) node[circle, draw, line width=0.65pt,
  minimum width=4.5mm, fill=white, inner sep=0.25mm] (text4238-3-9-0) {$\scriptstyle v$};
\end{tikzpicture}}} \qquad \quad \text{and} \qquad \quad
\vcenter{\hbox{\begin{tikzpicture}[y=0.70pt, x=0.70pt, yscale=-1.000000, xscale=1.000000, inner sep=0pt, outer sep=0pt]
\path[draw=black,line join=miter,line cap=butt,even odd rule,line width=0.800pt]
  (170.0000,942.3622) .. controls (170.2500,933.8622) and (178.5000,932.8622) ..
  (178.5000,925.6122);
\path[draw=black,line join=miter,line cap=butt,even odd rule,line width=0.800pt]
  (182.0000,925.3622) .. controls (182.5000,933.3622) and (190.0000,933.1122) ..
  (190.0000,942.3622);
\path[draw=black,line join=miter,line cap=butt,even odd rule,line width=0.800pt]
  (150.0000,852.3622) .. controls (150.3536,877.8180) and (159.7959,876.7955) ..
  (157.6746,895.1803);
\path[draw=black,line join=miter,line cap=butt,even odd rule,line width=0.800pt]
  (162.6575,895.2837) .. controls (162.6575,882.2023) and (178.1746,883.2544) ..
  (177.7071,871.9406);
\path[draw=black,line join=miter,line cap=butt,even odd rule,line width=0.800pt]
  (157.6930,899.5264) .. controls (158.1930,915.2764) and (150.2500,923.1122) ..
  (150.0000,942.3622);
\path[draw=black,line join=miter,line cap=butt,even odd rule,line width=0.800pt]
  (178.5282,921.0552) .. controls (179.5888,908.9338) and (174.5778,916.7500) ..
  (162.5570,899.7794);
\path[draw=black,line join=miter,line cap=butt,even odd rule,line width=0.800pt]
  (190.0000,852.3622) .. controls (190.2500,860.1122) and (182.5000,861.3622) ..
  (182.5000,868.8622);
\path[draw=black,line join=miter,line cap=butt,even odd rule,line width=0.800pt]
  (170.0000,852.3622) .. controls (170.7500,861.6122) and (177.7500,861.6122) ..
  (178.0000,868.6122);
\path[draw=black,line join=miter,line cap=butt,even odd rule,line width=0.800pt]
  (182.7349,920.9760) .. controls (181.9849,904.4760) and (189.6696,908.1836) ..
  (189.7500,894.6122) .. controls (189.8304,881.0408) and (182.5077,883.2729) ..
  (182.8599,871.8510);
\path[fill=black] (180.0762,870.5991) node[circle, draw, line width=0.65pt,
  minimum width=4.5mm, fill=white, inner sep=0.25mm] (text4238-3-9) {$\scriptstyle v$};
\path[fill=black] (160.2723,897.5394) node[circle, draw, line width=0.65pt,
  minimum width=4.5mm, fill=white, inner sep=0.25mm] (text4238-3-8-2) {$\scriptstyle v$};
\path[fill=black] (180.5989,923.4664) node[circle, draw, line width=0.65pt,
  minimum width=4.5mm, fill=white, inner sep=0.25mm] (text4238-3-3) {$\scriptstyle v$};
\end{tikzpicture}}} \rlap{\ \ \ ;}
\end{equation}
and taking $A = B = C = D = I$ in~\eqref{eq:10} shows
that this sufficient condition is also \emph{necessary}. In fact, the
structure of a map $v \colon HH \rightarrow HH$ rendering equal the
strings in~\eqref{eq:18} has been studied before:
\begin{Defn}
  \label{def:9}\cite{Street1998Fusion}
  Let $(\C, \otimes, I)$ be a symmetric monoidal category. A
  \emph{tricocycloid} in $\C$ comprises an object $H \in \C$ and an
  invertible map $v \colon H \otimes H \rightarrow H \otimes H$
  satisfying the equality
  \begin{equation}\label{eq:49}
    (v \otimes 1)(1 \otimes \sigma)(v \otimes 1) = (1 \otimes v)(v
    \otimes 1)(1 \otimes v) \colon H \otimes H \otimes H \rightarrow H
    \otimes H \otimes H\rlap{ .}
  \end{equation}
\end{Defn}
The preceding argument shows:
\begin{Prop}
  \label{prop:4}
  Let $(\C, \otimes, I)$ be a symmetric monoidal category with finite
  distributive coproducts, let $H \in \C$ and
  let $v \colon H \otimes H \rightarrow H \otimes H$. The pair $(H,v)$ is a
  tricocycloid if and only if there is a monoidal structure
  $(\star, 0)$ with $\star$ as in~\eqref{eq:19}, and with unit and
  associativity constraints as in~\eqref{eq:17} and~\eqref{eq:16}.
\end{Prop}

We can think of the object $H$ underlying a tricocycloid as
parametrising ``ways of non-trivially combining two things''; the map
$v$ then compares two ways in which $H \otimes H$ could
parametrise ``ways of non-trivially combining three things''. This
intuition may be clarified in terms of the notion of \emph{operad}.
Operads were introduced by May in~\cite{May1972The-geometry} as a tool
for describing certain kinds of topological-algebraic theory arising
in homotopy theory, and involve objects of ``$n$-ary operations'' for
each $n$, with suitable composition laws. The following notion of
\emph{pseudo-operad}, due to Markl, is concerned with the case where
the objects of nullary and unary operations are trivial, and so can be omitted.

\begin{Defn}
  \label{def:26}
  A \emph{pseudo-operad}~\cite{Markl1996Models} in a
  symmetric monoidal category $\C$ is a sequence
  $(H_n)_{n \geqslant 2}$ of objects endowed with maps
  \begin{equation}\label{eq:11}
    \circ_i \colon H_n \otimes H_m \rightarrow H_{n+m-1} \qquad \text{for $n,m\geqslant1$ and $1 \leqslant i \leqslant n$}
  \end{equation}
  rendering commutative the following diagrams for $n,m,k \geqslant
  1$ and $1 \leqslant i < j \leqslant n$:
  \begin{equation*}
    \cd{
      H_n \otimes H_m \otimes H_k \ar[d]_-{\circ_i \otimes 1} \ar[r]^-{1 \otimes \sigma} & H_n
      \otimes H_k \otimes H_m \ar[r]^-{\mathord{\circ_j} \otimes 1}  &
      H_{n+k-1} \otimes H_m \ar[d]^-{\circ_i} \\
      H_{n+m-1} \otimes H_k \ar[rr]^-{\circ_{j+m-1}} & & H_{n+m+k-2}
    }
  \end{equation*}
  and the following diagrams for $n,m,k \geqslant 1$ and $1
  \leqslant i \leqslant n$ and $1 \leqslant j \leqslant m$:
  \begin{equation*}
    \cd{
      H_n \otimes H_m \otimes H_k \ar[d]_-{\circ_i \otimes 1}
      \ar[r]^-{1 \otimes \mathord{\circ_j}}  &
      H_n \otimes H_{m+k-1} \ar[d]^-{\circ_i} \\
      H_{n+m-1} \otimes H_k \ar[r]^-{\circ_{j+i-1}} & H_{n+m+k-2}\rlap{ .}
    }
  \end{equation*}
\end{Defn}
We think of the objects $H_n$ involved in a pseudo-operad as
parametrising ``ways of non-trivially combining $n$ things''; the maps
$\circ_i$ then describe the way of combining $n+m-1$ things induced by
a way of combining $n$ things and a way of combining $m$ things,
according to the following schema:
\begin{equation*}
  \vcenter{\hbox{
    \begin{tikzpicture}[scale=0.75, line width=0.5pt]
    \draw (0,1) -- ++ (0:1) -- ++ (240:1) -- ++ (120:1);
    \draw (0,1) ++ (120:-1) -- ++ (270:0.5);
    \draw (0.2,1) -- ++ (90:0.5);
    \draw (0.8,1) -- ++ (90:0.5);
    \node (d) at (0.5,1.25) {$\scriptstyle\dots$};
    \node (a) at (0.5,0.7) {$\scriptstyle\alpha$};
  \end{tikzpicture}
  }} \ \otimes \
  \vcenter{\hbox{
  \begin{tikzpicture}[scale=0.75, line width=0.5pt]
    \draw (0,1) -- ++ (0:1) -- ++ (240:1) -- ++ (120:1);
    \draw (0,1) ++ (120:-1) -- ++ (270:0.5);
    \draw (0.2,1) -- ++ (90:0.5);
    \draw (0.8,1) -- ++ (90:0.5);
    \node (d) at (0.5,1.25) {$\scriptstyle\dots$};
    \node (a) at (0.5,0.7) {$\scriptstyle\beta$};
  \end{tikzpicture}
  }}
  \qquad \mapsto \qquad 
  \vcenter{\hbox{
  \begin{tikzpicture}[scale=0.75,line width=0.5pt]
    \draw (0,1) -- ++ (0:1) -- ++ (240:1) -- ++ (120:1);
    \draw (0,1) ++ (120:-1) -- ++ (270:0.5);
    \draw (0.2,1) .. controls ++(-0.2,0.3) and ++(0,-0.5) .. ++ (-0.4,1.86);
    \draw (0.5,1) node[above right=-2.5pt] {$\scriptscriptstyle i$} -- ++
    (90:0.5) -- ++ (120:1) -- ++ (0:1) -- ++ (240:1);
    \draw (0.8,1) .. controls ++(0.2,0.3) and ++(0,-0.5) .. ++ (0.4,1.86);
    \node (d) at (0.33,1.3) {$\scriptscriptstyle\dots$};
    \node (e) at (0.72,1.3) {$\scriptscriptstyle\dots$};
    \draw (0.5,1.5) ++ (120:1) ++ (0.2,0) -- ++ (90:0.5);
    \draw (0.5,1.5) ++ (120:1) ++ (0.8,0) -- ++ (90:0.5);
    \node (f) at (0.5,2.6) {$\scriptstyle\dots$};
    \node (g) at (0.02,2.6) {$\scriptstyle\dots$};
    \node (h) at (1.02,2.6) {$\scriptstyle\dots$};
    \node (a) at (0.5,0.65) {$\scriptstyle\alpha$};
    \node (b) at (0.5,2.05) {$\scriptstyle\beta$};
  \end{tikzpicture}
  }} \rlap{ .}
\end{equation*}

In general, there is no reason to expect the objects $H_n$
parametrising $n$-ary combinations for $n \geqslant 3$ to be
determined by the object $H_2$ of binary combinations; but when this
\emph{is} the case, we get a tricocycloid. This idea dates back to
Day~\cite{Day1970On-closed}, is explained in detail in the
introduction of~\cite{Day2003Lax-monoids}, and is made precise by:
\begin{Lemma}
  \label{lem:10}
  To give a tricocycloid in a symmetric monoidal category $\C$ is
  equally to give a pseudo-operad for which the maps~\eqref{eq:11} are
  all invertible.
\end{Lemma}
\begin{proof}[Proof (sketch)]
  From a pseudo-operad $H$ we obtain a
  tricocycloid with underlying object $H_2$ and with
  \begin{equation}\label{eq:48}
    v = H_2 \otimes H_2 \xrightarrow{\circ_1} H_3
    \xrightarrow{(\circ_2)^{-1}} H_2 \otimes H_2\rlap{ .}
  \end{equation}
  The tricocycloid axiom follows by constructing a suitable
  commutative diagram relating the various composition operations $H_2
  \otimes H_2 \otimes H_2 \rightarrow H_4$, and using invertibility of
  the maps $\circ_i$. Conversely, from a tricocycloid $(H,v)$ we
  construct a pseudo-operad with $H_n = H^{\otimes(n-1)}$, and with
  the maps $\circ_i$ given by suitable composites of $v$ which we will
  not spell out in general; but let us at least say that, in the lowest dimension, we have
  $\circ_1, \circ_2 \colon H_2 \otimes H_2 \rightarrow H_3$ given by
  $v, \id \colon H \otimes H \rightarrow H \otimes H$
\end{proof}


We now describe, following~\cite[Section~4]{Street1998Fusion}, the
additional structure on a tricocycloid needed to induce a symmetry on
the associated monoidal structure. 
Such a symmetry is given by coherent isomorphisms
$\sigma_{AB} \colon A \star B \rightarrow B \star A$, i.e., maps
$A + HAB + B \rightarrow B + HBA + A$, and the coherence axiom
relating $\sigma$ with the unit constraints force the $A$- and
$B$-summands of the domain to be mapped to the corresponding summands
of the codomain. Like before, it is now natural to map the remaining
$HAB$-summand to the $HBA$-summand via a composite
\begin{equation}\label{eq:20}
  HAB \xrightarrow{\gamma 11} HAB \xrightarrow{1\sigma} HBA
\end{equation}
for some fixed map $\gamma \colon H \rightarrow H$. Since a symmetry
must satisfy $\sigma_{BA} \circ \sigma_{AB} = 1$, it follows that
$\gamma$ must be an involution (i.e., $\gamma^2 = 1$). As for the
hexagon axiom relating the symmetry to the associativity, its only
non-trivial case expresses the equality of the maps
$HHABC \rightrightarrows HAHBC$ given by the respective diagrams:
\begin{equation}\label{eq:21}
\vcenter{\hbox{\begin{tikzpicture}[y=0.60pt, x=0.70pt, yscale=-1.000000, xscale=1.000000, inner sep=0pt, outer sep=0pt]
\path[draw=black,line join=miter,line cap=butt,even odd rule,line width=0.800pt]
  (40.0000,902.3622) -- (40.0000,916.8370) node[above=0.15cm,at start] {$\scriptstyle H$};
\path[draw=black,line join=miter,line cap=butt,even odd rule,line width=0.800pt]
  (60.0000,902.3622) .. controls (59.7392,922.4255) and (80.2733,942.5041) ..
  (80.0000,962.3622) node[above=0.15cm,at start] {$\scriptstyle A$} .. controls (79.7218,982.5797) and (80.0000,980.1709) ..
  (80.0000,1007.3916) .. controls (80.0000,1020.8243) and (100.3535,1037.2807)
  .. (100.0000,1050.3622);
\path[draw=black,line join=miter,line cap=butt,even odd rule,line width=0.800pt]
  (80.0000,902.3622) .. controls (80.7071,922.2706) and (60.2376,942.9284) ..
  (60.0000,962.3622) node[above=0.15cm,at start] {$\scriptstyle B$} .. controls (59.6439,991.4858) and (39.1186,966.8757) ..
  (40.0000,1050.3622);
\path[draw=black,line join=miter,line cap=butt,even odd rule,line width=0.800pt]
  (100.0000,902.3622) -- (100.0000,1008.0987) node[above=0.15cm,at
  start] {$\scriptstyle C$} .. controls (100.0000,1024.3598)
  and (79.6464,1034.4523) .. (80.0000,1050.3622);
\path[draw=black,line join=miter,line cap=butt,even odd rule,line width=0.800pt]
  (20.0000,902.3622) .. controls (20.0000,935.4837) and (27.4246,924.8762) ..
  (27.4246,956.7782) node[above=0.15cm,at start] {$\scriptstyle H$};
\path[draw=black,line join=miter,line cap=butt,even odd rule,line width=0.800pt]
  (31.7678,958.5944) .. controls (31.7678,941.7529) and (40.0000,951.1460) ..
  (40.0000,921.5338);
\path[draw=black,line join=miter,line cap=butt,even odd rule,line width=0.800pt]
  (32.1734,960.9340) .. controls (32.1734,990.0388) and (60.0000,976.9292) ..
  (60.0000,1008.3622);
\path[draw=black,line join=miter,line cap=butt,even odd rule,line width=0.800pt]
  (26.5165,960.2269) .. controls (26.5165,985.2222) and (20.0000,1008.8591) ..
  (20.0000,1050.3622);
\path[draw=black,line join=miter,line cap=butt,even odd rule,line width=0.800pt]
  (60.0000,1012.3622) -- (60.0000,1050.3622);
\path[fill=black] (29.3941,959.0660) node[circle, draw, line width=0.65pt,
  minimum width=4.5mm, fill=white, inner sep=0.25mm] (text4238-5) {$\scriptstyle v$};
\path[fill=black] (40.4390,919.3124) node[circle, draw, line width=0.65pt,
  minimum width=4.5mm, fill=white, inner sep=0.25mm] (text4238-1) {$\scriptstyle \gamma$};
\path[fill=black] (60.2380,1010.2971) node[circle, draw, line width=0.65pt,
  minimum width=4.5mm, fill=white, inner sep=0.25mm] (text4238-0) {$\scriptstyle \gamma$};
\end{tikzpicture}}} \quad \qquad \text{and} \qquad \quad
\vcenter{\hbox{\begin{tikzpicture}[y=0.60pt, x=0.70pt, yscale=-1.000000, xscale=1.000000, inner sep=0pt, outer sep=0pt]
\path[draw=black,line join=miter,line cap=butt,even odd rule,line width=0.800pt]
  (140.0000,902.3622) .. controls (140.0000,912.6396) and (147.4318,909.8802) ..
  (147.4318,919.4261) node[above=0.15cm,at start] {$\scriptstyle H$};
\path[draw=black,line join=miter,line cap=butt,even odd rule,line width=0.800pt]
  (152.3815,919.4261) .. controls (152.3815,910.9408) and (160.0000,910.5016) ..
  (160.0000,902.3622) node[above=0.15cm,at end] {$\scriptstyle H$};
\path[draw=black,line join=miter,line cap=butt,even odd rule,line width=0.800pt]
  (147.0782,925.0830) .. controls (143.1891,943.1843) and (140.0000,944.4880) ..
  (140.0000,966.3622);
\path[draw=black,line join=miter,line cap=butt,even odd rule,line width=0.800pt]
  (152.2218,1017.7972) .. controls (152.2218,986.5607) and (179.7231,1003.7102)
  .. (180.0000,974.3622) .. controls (180.3536,936.8855) and (152.7351,957.0282)
  .. (152.7351,925.0830);
\path[draw=black,line join=miter,line cap=butt,even odd rule,line width=0.800pt]
  (179.9587,902.8091) .. controls (179.9587,916.6700) and (179.6257,912.0628) ..
  (180.0000,922.3622) node[above=0.15cm,at start] {$\scriptstyle A$} .. controls (180.9878,949.5430) and (159.6464,937.0944) ..
  (159.6464,974.2409) .. controls (159.6464,1011.3874) and (220.0000,987.2251)
  .. (220.0000,1016.3622) .. controls (220.0000,1045.4993) and
  (220.0000,1027.6906) .. (220.0000,1050.3622);
\path[draw=black,line join=miter,line cap=butt,even odd rule,line width=0.800pt]
  (140.0000,970.3622) .. controls (138.5303,993.5646) and (146.4645,1012.8469)
  .. (147.1716,1017.8977);
\path[draw=black,line join=miter,line cap=butt,even odd rule,line width=0.800pt]
  (200.0000,902.3622) -- (200.0000,984.3622) node[above=0.15cm,at
  start] {$\scriptstyle B$} .. controls (200.0000,1005.4795)
  and (160.3535,993.9982) .. (160.0000,1050.3622);
\path[draw=black,line join=miter,line cap=butt,even odd rule,line width=0.800pt]
  (220.0000,902.3622) -- (220.0000,984.3622) node[above=0.15cm,at
  start] {$\scriptstyle C$} .. controls (220.0000,1002.0398)
  and (200.0000,1004.0884) .. (200.0000,1026.3622) -- (200.0000,1050.3622);
\path[draw=black,line join=miter,line cap=butt,even odd rule,line width=0.800pt]
  (147.4318,1021.8873) .. controls (147.4318,1041.1332) and (140.0000,1031.8326)
  .. (140.0000,1050.3622);
\path[draw=black,line join=miter,line cap=butt,even odd rule,line width=0.800pt]
  (152.3815,1021.5338) .. controls (154.7678,1044.4492) and (180.7071,1030.4627)
  .. (180.0000,1050.3622);
\path[fill=black] (149.4818,1020.0101) node[circle, draw, line width=0.65pt,
  minimum width=4.5mm, fill=white, inner sep=0.25mm] (text4238-4) {$\scriptstyle v$};
\path[fill=black] (139.7875,967.8705) node[circle, draw, line width=0.65pt,
  minimum width=4.5mm, fill=white, inner sep=0.25mm] (text4238-7) {$\scriptstyle \gamma$};
\path[fill=black] (149.9558,921.9428) node[circle, draw, line width=0.65pt,
  minimum width=4.5mm, fill=white, inner sep=0.25mm] (text4238-4-7)
  {$\scriptstyle v$};
\end{tikzpicture}}}\rlap{\ \ \ \ .}
\end{equation}
Like before, it is necessary and sufficient for this that we should
have equality of the diagrams obtained from~\eqref{eq:21} by deleting
the $A$-, $B$- and $C$-strings; we encapsulate this requirement in:
\begin{Defn}
  \label{def:10}
  Let $(H,v)$ be a tricocycloid in the symmetric monoidal category $(\C,
  \otimes, I)$. A \emph{symmetry} for $H$ is an involution $\gamma
  \colon H \rightarrow H$ satisfying the equality
  \begin{equation}\label{eq:24}
    (1 \otimes \gamma)v(1 \otimes \gamma) = v(\gamma \otimes 1) v
    \colon H \otimes H \rightarrow H \otimes H\rlap{ .}
  \end{equation}
\end{Defn}
The preceding argument thus shows:
\begin{Prop}
  \label{prop:7}
  Let $(\C, \otimes, I)$ be a symmetric monoidal category with finite
  distributive coproducts, and let $(H,v)$ be a tricocycloid in $\C$.
  An involution $\gamma \colon H \rightarrow H$ is a symmetry for
  $(H,v)$ just when the maps
  $\sigma_{AB} \colon A \star B \rightarrow B \star A$ determined
  by~\eqref{eq:20} endow the associated monoidal structure
  $(\star, 0)$ on $\C$ with a symmetry.
\end{Prop}
A symmetry on a tricocycloid can also be described via the
corresponding pseudo-operad. We call a pseudo-operad $H$ in $\C$
\emph{symmetric} if each $H_n$ carries a symmetric group action
$\alpha \colon S_n \rightarrow \C(H_n, H_n)$, with respect to which
composition is equivariant. It is now straightforward to show that
giving a symmetric tricocycloid is the same as giving a symmetric
pseudo-operad with all maps~\eqref{eq:11} invertible.

\subsection{The convex monoidal structure}
\label{sec:convex-tricycloid}

Using the preceding theory, we can obtain the associativity and
symmetry constraints of the desired convex monoidal structure $A,B
\mapsto A + (0,1) \times A \times B + B$ on
$\cat{Set}$ by endowing the set $(0,1)$ with the structure of a
symmetric tricocycloid.

This tricycloid is most easily understood by deriving it from a
symmetric pseudo-operad. Indeed, for each $n \geqslant 2$ we may
consider the set
\begin{equation*}
  H_n = \{(r_1, \dots, r_n) \in (0,1)^n : r_1 + \dots + r_n = 1\}\rlap{ .}
\end{equation*}
We now have maps $\circ_i \colon H_n \times H_m \rightarrow H_{n+m-1}
$ defined by
\begin{equation}\label{eq:44}
  \bigl((r_1, \dots, r_n), (s_1, \dots, s_m)\bigr) \mapsto (r_1, \dots, r_{i-1},
  r_is_1,
  \dots, r_i s_m, r_{i+1}, \dots, r_n)
\end{equation}
and maps $\sigma \colon S_n \times H_n \rightarrow H_n$ defined by
\begin{equation}\label{eq:45}
 \bigl(g, (r_1, \dots, r_n)\bigr) \mapsto (r_{g(1)}, \dots,
r_{g(n)})
\end{equation}
which easily satisfy the axioms for a symmetric
pseudo-operad.
Moreover, each of the maps $\circ_i$ is invertible,
with inverse
\begin{equation}\label{eq:47}
  (t_1, \dots, t_{n+m-1}) \mapsto \bigl((t_1, \dots, t_{i-1},
  u, t_{i+m}, \dots, t_{n+m-1}), (\tfrac{t_i}{u}, \dots, \tfrac{t_{i+m-1}}{u})\bigr)
\end{equation}
where here $u \defeq \sum_{j=i}^{i+m-1} t_j$. So by
Lemma~\ref{lem:10}, $H_2$ is a symmetric tricocycloid; transporting
this structure along the isomorphism $H_2 \cong (0,1)$ given by
$(r,s) \mapsto r$, we conclude that $(0,1)$ is a symmetric
tricocycloid. The following result spells the structure out, and gives
a direct proof of the symmetric tricocycloid axioms.

\begin{Prop}
  \label{prop:5}
  In the cartesian monoidal category of sets, $(0,1)$ is a
  symmetric tricocycloid, which we term the \emph{convex tricocycloid}, under the
  operations
  \begin{align*}
    v \colon (0,1)^2 & \mapsto (0,1)^2  & \gamma \colon (0,1) &
    \mapsto (0,1) \\
    (r,s) & \mapsto (rs, \tfrac{r \cdot s^\ast}{(rs)^\ast}) & r &
    \mapsto r^\ast\rlap{ ,}
  \end{align*}
  where, as before, we write $r^\ast = 1-r$ for any $r \in (0,1)$.
\end{Prop}
\begin{proof}
  We begin by checking that $\bigl((0,1), v\bigr)$ is a tricocycloid.
  It is easy arithmetic to see that $rs$ and
  $r \cdot s^\ast / (rs)^\ast$ are in $(0,1)$ whenever $r$ and $s$
  are, so that $v$ is well-defined. For the tricocycloid axiom, the
  function
  $(v \times 1)(1 \times \sigma)(v \times 1) \colon (0,1)^3
  \rightarrow (0,1)^3$ is given by
  \begin{equation*}
    (r,s,t) \mapsto (rs, \tfrac{r\cdot s^\ast}{(rs)^\ast}, t)
    \mapsto (rs, t, \tfrac{r\cdot s^\ast}{(rs)^\ast})
    \mapsto (rst, \tfrac{rs\cdot t^\ast}{(rst)^\ast}, \tfrac{r\cdot s^\ast}{(rs)^\ast})
    \rlap{ ,}
  \end{equation*}
  while the map $(1 \times v)(v \times 1)(1 \times v) \colon (0,1)^3
  \rightarrow (0,1)^3$ is given by
  \begin{equation*}
    (r,s,t) \mapsto (r, st, \tfrac{s \cdot t^\ast}{(st)^\ast})
    \mapsto (rst, \tfrac{r(st)^\ast}{(rst)^\ast},\tfrac{s \cdot t^\ast}{(st)^\ast})
    \mapsto \Big(\,rst, \tfrac{rs \cdot
      t^\ast}{(rst)^\ast},\left(\!\tfrac{r(st)^\ast}{(rst)^\ast}\!\right)\!
      \left(\!\tfrac{s \cdot t^\ast}{(st)^\ast} \!\right)^\ast\!\! \Big/\! \left(\!\tfrac{rs \cdot t^\ast}{(rst)^\ast}\!\right)^\ast\Big)\text.
  \end{equation*}
  To see that the final terms agree, note first that for any
  $a,b \in (0,1)$ we have
  \begin{equation}\label{eq:22}
    \left(\!\tfrac{a \cdot b^\ast}{(ab)^\ast}\!\right)^\ast = 1 -
    \tfrac{a - ab}{1 - ab} = \tfrac{1-a}{1 - ab} = \tfrac{a^\ast}{(ab)^\ast}
  \end{equation}
  so that the desired equality follows from the calculation
  \begin{equation*}
\left(\!\tfrac{r(st)^\ast}{(rst)^\ast}\!\right)\!
      \left(\!\tfrac{s \cdot t^\ast}{(st)^\ast}
        \!\right)^\ast \! \Big/\! \left(\!\tfrac{rs \cdot
          t^\ast}{(rst)^\ast}\!\right)^\ast \ = \    \left(\!\tfrac{r(st)^\ast}{(rst)^\ast}\!\right)\!
      \left(\!\tfrac{s^\ast}{(st)^\ast}\!\right) \! \Big/\! \left(\!\tfrac{(rs)^\ast}{(rst)^\ast}\!\right)
    \ = \ \tfrac{r \cdot s^\ast}{(rs)^\ast}\rlap{ .}
  \end{equation*}
  We note also that $v$ is invertible, with inverse
  $v^{-1}(t,u) = \bigl((t^\ast \cdot u^\ast)^\ast, \tfrac{(t^\ast \cdot
    u)^\ast}{t^\ast \cdot u^\ast}\bigr)$; that $v \circ v^{-1}$ and $v^{-1}
  \circ v$ are identities follows by a short calculation
  using~\eqref{eq:22}.

  We now show that $\gamma$ provides a symmetry for the tricocycloid
  $\bigl((0,1), v\bigr)$. Clearly $\gamma$ is an involution, so it
  remains to check the coherence axiom. The map
  $(1 \times \gamma)v(1 \times \gamma) \colon (0,1)^2 \rightarrow
  (0,1)^2$ is given by
  \begin{equation*}
    (r,s) \mapsto (r, s^\ast) \mapsto (r\cdot s^\ast, \tfrac{r \cdot s^{\ast\ast}}{(r
      \cdot s^\ast)^\ast}) \mapsto \Big(r\cdot s^\ast, \left(\!\tfrac{r \cdot s^{\ast\ast}}{(r
      \cdot s^\ast)^\ast}\!\right)^\ast\Big)
  \end{equation*}
  while $v(\gamma \times 1)v \colon (0,1)^2 \rightarrow (0,1)^2$ is
  given by
  \begin{equation*}
    (r,s) \mapsto (rs, \tfrac{r \cdot s^\ast}{(rs)^\ast}) \mapsto
((rs)^\ast, \tfrac{r \cdot s^\ast}{(rs)^\ast}) \mapsto \Big(r \cdot
s^\ast, (rs)^\ast \left(\!\tfrac{r \cdot
    s^\ast}{(rs)^\ast}\!\right)^\ast \!\!\Big/ (r \cdot s^\ast)^\ast \Big)\rlap{ .}
  \end{equation*}
  To check the equality of the second terms, we calculate
  using~\eqref{eq:22} twice that:
  \begin{equation*}
    (rs)^\ast \left(\!\tfrac{r \cdot
        s^\ast}{(rs)^\ast}\!\right)^\ast \!\!\Big/ (r \cdot s^\ast)^\ast =
    (rs)^\ast \tfrac{r^\ast}{(rs)^\ast} \!\Big/\! (r \cdot
    s^\ast)^\ast = \tfrac{r^\ast}{(r \cdot
    s^\ast)^\ast} = 
    \left(\!\tfrac{r \cdot s^{\ast\ast}}{(r
      \cdot s^\ast)^\ast}\!\right)^\ast\rlap{ .} \qedhere
  \end{equation*}
\end{proof}

Since the definition of $v$ derives from the coproduct of abstract
convex spaces, it is not unreasonable that the same coefficients
should appear here as in the convex space axioms. We will see the
deeper reason for this in Section~\ref{sec:tric-effect-algebr}.

\begin{Defn}
  \label{def:11}
  The \emph{convex monoidal structure} is the symmetric monoidal
  structure $(\star, 0)$ on the category of sets associated to the
  convex tricocycloid $\big((0,1), v, \gamma\big)$.
\end{Defn}
Working through the details of Lemma~\ref{lem:2}, the reader should
have no difficulty in verifying that the forgetful functor
$\cat{Conv} \rightarrow \cat{Set}$ is \emph{strict} symmetric monoidal
with respect to the coproduct monoidal structure on $\cat{Conv}$ and
the convex monoidal structure on $\cat{Set}$. In light of
Lemma~\ref{lem:1} and Lemma~\ref{lem:3}, we have thus verified:
\begin{Prop}
  \label{prop:8}
  With respect to the convex monoidal structure on $\cat{Set}$, the
  discrete distribution monad $\mnd D$ is a linear exponential monad.
\end{Prop}
It is now easy to check that the abstract hypernormalisation
maps~\eqref{eq:23} reduce in this case to Jacobs' hypernormalisation
maps~\eqref{eq:43}, as desired.

\section{Probabilistic examples}
\label{sec:examples}

In the following two sections, we describe instances of abstract
hypernormalisation which go beyond the motivating case. In this
section, we explore examples involving \emph{continuous} probability
monads; the hypernormalisation arising here is suitable for the
channel-to-abstract-channel construction of
Section~\ref{sec:channel-abstraction} above, and so for non-discrete
generalisations of the theory of~\cite{McIver2014Abstract}.

\subsection{The expectation monad}
\label{sec:expectation-monad}
The next simplest probabilistic monad beyond the finite discrete case
is the so-called \emph{expectation monad} $\mnd{E}$ on the category of
sets. This was named and investigated
in~\cite{Jacobs2016The-expectation}, but dates back
to~\cite{Swirszcz1974Monadic}, where it was described, implicitly, as
the monad generated by the chain of adjunctions:
\begin{equation}\label{eq:37}
  \cd{
    {\cat{K\C onv}} \ar@<-4.5pt>[r]_-{} \ar@{<-}@<4.5pt>[r]^-{} \ar@{}[r]|-{\bot} &
    {{\cat{K\H}} } \ar@<-4.5pt>[r]_-{} \ar@{<-}@<4.5pt>[r]^-{} \ar@{}[r]|-{\bot} &
    {\cat{Set}} \rlap{ .}
  }
\end{equation}

Here, $\cat{K\H}$ is the category of compact Hausdorff spaces; while
$\cat{K\C onv}$ is the category whose objects are compact convex
subsets $A$ of locally convex vector spaces, and whose morphisms are
continuous affine maps (where affineness is the condition
$f(ra + r^\ast a') = rf(a) + r^\ast f(a')$.) The two right adjoints
in~\eqref{eq:37} are the obvious forgetful functors, while the two
left adjoints send, respectively, a set $X$ to its space of
ultrafilters $\beta X$ with the Stone topology, and a compact
Hausdorff space $Y$ to its space of Radon probability measures,
identified via the Riesz representation theorem with the positive
elements of norm 1 in the ordered Banach space of continuous linear
functionals $C(Y, \mathbb{R}) \rightarrow \mathbb{R}$.

As explained in~\cite{Jacobs2016The-expectation}, the monad $\mnd{E}$
induced by the composite adjunction~\eqref{eq:37} can be described in
various ways; the most direct is as follows. We write $\E X$ for the
set of normalised, finitely additive functions $\omega \colon \P X
\rightarrow [0,1]$, i.e., functions such that $\omega(X) = 1$ and
$\omega(A \cup B) = \omega(A) + \omega(B)$ whenever $A,B \subseteq X$ are disjoint.
The action of $\E$ on morphisms is given by pushforward,
$(\E f)(\omega)(B) = \omega(f^{-1}(B))$; the monad unit
$\eta_X \colon X \rightarrow \E X$ takes $x \in X$ to the Dirac
distribution with $\eta_X(x)(A) = 1$ if $x \in A$ and
$\eta_X(x)(A) = 0$ otherwise; while the monad multiplication is given
by a suitable notion of integration against a valuation:
\begin{align*}
  \mu_X(\omega)(A) = \int_{\tau \in \E X} \tau(A) \, \mathrm d \mathbf{\omega}\rlap{ .}
\end{align*}

The details may be found in~\cite{Jacobs2016The-expectation}; however,
we will not need them to describe hypernormalisation for
$\mnd{E}$. Rather, we need only make $\mnd{E}$ into a
linear exponential monad, which we can do from the perspective of the
category of algebras using:

\begin{Prop}
  \label{prop:14}(\cite[Theorem~4]{Swirszcz1974Monadic})
  The composite adjunction in~\eqref{eq:37} is monadic.
\end{Prop}

Thus $\mnd{E}$-algebras can be identified with compact
convex subsets of locally convex vector spaces; and so if we can
understand finite coproducts of these, we can obtain the desired
linear exponential structure on $\mnd{E}$. 

\begin{Prop}
  \label{prop:15}
  If $A \subseteq V$ and $B \subseteq W$ are objects of $\cat{K\C
    onv}$, then their coproduct may be given as
  \begin{equation}\label{eq:38}
    \{(ra, r^\ast b, r) : r \in [0,1], a \in A, b \in B\}
    \subseteq V \oplus W \oplus \mathbb{R}\rlap{ .}
  \end{equation}
\end{Prop}
\begin{proof}
  This is~\cite[Proposition~7]{Roumen2017Duality}.
\end{proof}

At the level of underlying sets, the coproduct of $A$ and $B$
in~\eqref{eq:38} is given by
\begin{equation*}
  \{(a,0,1) : a \in A\} + \{(ra,r^\ast b, r) : r \in (0,1), a \in A,
  b\in B\} + \{(0,b,0) : b \in B\}
\end{equation*}
which is clearly isomorphic to $A + (0,1) \times A \times B
+ B = A \star B$. In a similar way, the coherence constraints for the
coproduct in $\cat{K\C onv}$ lift the coherence constraints for the
convex monoidal structure on $\cat{Set}$, and so we have:

\begin{Prop}
  \label{prop:11}
  The expectation monad on $\cat{Set}$ is a linear
  exponential monad with respect to the convex monoidal structure.
\end{Prop}
It follows that the expectation monad admits a notion of
hypernormalisation. To calculate this, we first define, like before,
the \emph{normalisation} of a valuation
$\omega \colon \P A \rightarrow [0,1]$ with $\omega(A) > 0$ to be the
normalised valuation $\overline \omega$ with
$\overline \omega(U) = \omega(U) / \omega(A)$. Noting that each set
$\E A$ admits a structure of abstract convex space where
$r(\omega_1, \omega_2)(U) = r \omega_1(U) + r^\ast \omega_2(U)$, we
may now describe the hypernormalisation map
$\N \colon \E(\Sigma_i A_i) \rightarrow \E(\Sigma_i \E A_i)$ as
in~\eqref{eq:29} by
\begin{equation*}
  \N(\omega) = \sum_{\substack{1 \leqslant i \leqslant
      n\\\omega_i(A_i) > 0}} \omega_i(A_i) \cdot \iota_i(\overline{\omega_i})\rlap{ .}
\end{equation*}

\subsection{The Radon monad}
\label{sec:radon-monad}
The expectation monad $\mnd{E}$ on $\cat{Set}$ arose from the
composite adjunction in~\eqref{eq:37}; on the other hand, the
left-hand adjunction in~\eqref{eq:37} induces the \emph{Radon monad}
$\mnd{R}$ on $\K\H$. Again, this was introduced implicitly
in~\cite{Swirszcz1974Monadic}, with the details now being provided
by~\cite{Mislove2011Probabilistic}.


\begin{Defn}
  \label{def:16}
  Let $X$ be a compact Hausdorff space. A \emph{Radon probability
    measure} on $X$ is a probability measure
  $\omega \colon \Sigma_X \rightarrow [0,1]$ on the Borel
  $\sigma$-algebra of $X$ such that
  $\omega(M) = \mathrm{sup}\{\omega(K) : K \subseteq M, K \text{
    compact}\}$ for all $M \in \Sigma_X$. We write
  $\R(X)$ for the space of Radon probability measures on $X$ with the
  weak topology: the coarsest topology such that, for each continuous
  $f \colon X \rightarrow \mathbb{R}$, the integration map
  $\omega \mapsto \int f \mathrm{d}\omega$ is continuous as a function
  $\R(X) \rightarrow \mathbb{R}$.
\end{Defn}

The remaining aspects of the Radon monad $\mnd{R}$ on $\K\H$ are
much as before: the action on morphisms is by pushforward, the monad
unit selects the Dirac valuations, and the multiplication is given by
integration against a measure. To obtain our notion of
hypernormalisation, we will again exploit monadicity, using:
\begin{Prop}
  (\cite[Theorem~4]{Swirszcz1974Monadic})
  The left-hand adjunction in~\eqref{eq:37} is monadic.
\end{Prop}

So, identifying the category of $\mnd{R}$-algebras with the category
$\cat{K\C onv}$, it only remains to relate the coproduct~\eqref{eq:38}
of compact convex spaces with a suitable monoidal structure on $\K\H$.
This will be the well-known \emph{topological join}:

\begin{Defn}
  \label{def:21}
  The \emph{join} of two topological spaces $X$ and $Y$ is the
  quotient space of the product space $[0,1] \times X \times Y$ under
  the smallest equivalence relation $\sim$ for which $(0,x,y) \sim
  (0,x',y)$ and $(1,x,y) \sim (1,x,y')$ for all $x,x' \in X$ and $y,y'
  \in Y$.
\end{Defn}
We can realise the topological join $X \star Y$ as the set $X + (0,1)
\times X \times Y + Y$, with a basis for the topology generated by
sets of three forms
\begin{equation*}
  U + (0,a) \times U \times Y + \emptyset \qquad 
  \emptyset + (a,b) \times U \times V + \emptyset \qquad 
  \emptyset + (b,1) \times X \times V + V
\end{equation*}
for all rationals $0<a<b<1$ and all $U \subseteq X$, $V \subseteq Y$
open. Presented in this way, it is easy to see that topological join
is part of a monoidal structure $(\star, 0)$ on $\K\H$ which lifts the
convex monoidal structure on $\cat{Set}$. On the other hand, comparing
with the formula~\eqref{eq:38}, we conclude that the forgetful functor
$\cat{K\C onv} \rightarrow \K\H$ sends coproduct to topological join,
and so we obtain:

\begin{Prop}
\label{prop:13}
The Radon monad on $\cat{K\H}$ is a linear exponential monad with
respect to the join monoidal structure $(\star, 0)$ on $\K\H$.
\end{Prop}

We thus obtain hypernormalisation maps
$\N \colon \R(\Sigma_i A_i) \rightarrow \R(\Sigma_i \R A_i)$ whose action on a
distribution $\omega$ is obtained in much the same way as previously.
The key point is that we get continuity of $\N$ for free: something
which otherwise would have required some fairly messy calculation.

\subsection{The Kantorovich monad}
\label{sec:kantorovich-monad}

For our final example of probabilistic hypernormalisation, we consider
the \emph{Kantorovich monad}~\cite{Breugel2005The-metric} on the category $\cat{CMet}_1$ whose objects are
complete metric spaces which are $1$-bounded (i.e.,
$d(x,y) \leqslant 1$ for all $x,y$) and whose morphisms are
$1$-Lipschitz mappings (i.e., $d(fx,fy) \leqslant d(x,y)$ for all
$x,y$). 

\begin{Defn}
  \label{def:23}
  Let $X$ be a complete $1$-bounded metric space. $\K(X)$ is the
  complete metric space whose elements are Radon probability measures
  on $X$, under the Kantorovich (or ``earth-mover's'') metric:
  \begin{equation*}
    d(\omega, \pi) = \mathrm{inf}\{\,\textstyle\int d_X(x,y)
    \,\mathrm{d}\mu(x,y) : \mu \in \K(X \times X), (\pi_1)_\ast(\mu) =
    \omega, (\pi_2)_\ast(\mu) = \pi\,\}
  \end{equation*}
  where the infimum is over joint distributions on $X \times X$ with marginals $\omega$ and~$\pi$.
\end{Defn}
This operation underlies a monad on $\cat{CMet}_1$ following the
established pattern; and to exhibit a notion of hypernormalisation, we
also follow the established pattern, by investigating
coproducts of $\mnd{K}$-algebras. We begin with the characterisation
of these algebras.

\begin{Defn}
  \label{def:24}
  A \emph{convex metric space} is a metric space $X$ which is also a
  convex space in the sense of Definition~\ref{def:5}, subject to the
  compatibility condition
  \begin{equation}
    d\bigl( r(x,z), r(y,z)\bigr) \leqslant r d(x,y) \text{\ \ \ 
      (or equally, } d\bigl( r(x,y), r(x,z)\bigr) \leqslant r^\ast d(y,z)
    \text{ )} \rlap{ .}\label{def:25}
  \end{equation}
  We write $\cat{C\C onv\M et}_1$ for the category of $1$-bounded
  complete convex metric spaces, and convex $1$-Lipschitz maps.
\end{Defn}

\begin{Prop}
  \label{prop:17}
  The category of $\mnd{K}$-algebras is isomorphic to $\cat{C\C onv\M
    et}_1$ over $\cat{C\M
    et}_1$.
\end{Prop}
\begin{proof}
  This is~\cite[Theorem~5.2.1]{Fritz2019A-Probability} (though see
  also~\cite[Theorem~10.9]{Mardare2016Quantitative}).
\end{proof}
And now we characterise finite coproducts in this category.
\begin{Prop}
  \label{prop:18}
  The coproduct of $X,Y \in \cat{C\C onv\M et}_1$ is the coproduct $X
  \star Y$ of underlying convex spaces, endowed with the metric:
  \begin{equation}\label{eq:39}
    d_{X \star Y}(r \cdot x + r^\ast \cdot y, s \cdot w + s^\ast \cdot
    z) = rd(x,w) + (s-r) + (1-s)d(y,z)
  \end{equation}
  for any $0 \leqslant r \leqslant s \leqslant 1$; here, by
  convention, we allow $1 \cdot x + 0 \cdot y$ to denote $x \in X
  \subseteq X \star Y$, and correspondingly for $0 \cdot x + 1 \cdot y$.
\end{Prop}
\begin{proof}
  It is a straightfoward calculation to show that this is indeed a
  complete, $1$-bounded metric satisfying~\eqref{def:25}; indeed, it
  is easy to see that $d_{X \star Y}$ is really just the Kantorovich
  metric restricted to distributions concentrated at two points.

  To exhibit the
  universal property of coproduct, let $f \colon X \rightarrow Z$ and
  $g \colon Y \rightarrow Z$ be maps in $\cat{C\C onv\M et}_1$, and
  let $\spn{f,g} \colon X \star Y \rightarrow Z$ be the induced unique map of
  convex spaces as in Section~\ref{sec:convex-coproducts}. It suffices
  to show that $\spn{f,g}$ is $1$-Lipschitz. For the most involved
  case, consider elements $r \cdot x + r^\ast \cdot y$ and $s \cdot w + s^\ast \cdot
    z$ in $X \star Y$ with $0 < r \leqslant s < 1$. We must show that
  \begin{equation*}
    d\bigl(r(fx,gy), s(fw,gz)\bigr) \leqslant rd(x,w) + (s-r) + (1-s)d(y,z)\rlap{ .}
  \end{equation*}
  Now, the left-hand side is by the triangle inequality smaller than
  \begin{equation*}
    d\bigl(r(fx,gy), r(fw,gy)\bigr) +
    d\bigl(r(fw,gy), s(fw,gy)\bigr) + d\bigl(s(fw,gy), s(fw,gz)\bigr)
  \end{equation*}
  and we calculate that $d(r(fx,gy), r(fw,gy)) \leqslant r d(fx,fw) \leqslant r
    d(x,w)$ and $d\bigl(s(fw,gy), s(fw,gz)\bigr) \leqslant s^\ast d(gy,gz) \leqslant s^\ast
    d(y,z)$ 
  using~\eqref{def:25} and contractivity of $f$ and $g$.
  So it suffices to show that $d\bigl(r(fw,gy), s(fw,gy)\bigr)
  \leqslant s-r$ in $Z$. Writing $u = fw$ and $v = gy$, this follows
  by the calculation
  \begin{align*}
    d\bigl( r(u,v), s(u,v)\bigr) &= d\bigl( s(\tfrac r s(u,v),v),
    s(u,v)\bigr) \leqslant s d\bigl(\tfrac r s(u,v), u\bigr)\\
    & = s d\bigl(\tfrac {s-r}s (v,u), \tfrac {s-r}s
    (u,u)\bigr) \leqslant (s-r) d(v,u) \leqslant s-r\rlap{ .}\qedhere
  \end{align*}
\end{proof}

We are thus in the familiar situation that the 
$\cat{C\M et}_1$-object underlying the coproduct of $X,Y \in \cat{C\C onv\M et}_1$
depends only on the underlying $\cat{C\M et}_1$-objects of $X$ and
$Y$. Accordingly, we make:

\begin{Defn}
  \label{def:22}
  The \emph{join} of two $1$-bounded complete metric spaces $X$ and
  $Y$ is the set $X \star Y = X + (0,1) \times X \times Y + Y$ endowed
  with the metric~\eqref{eq:39}. This provides the binary tensor of
  the \emph{convex monoidal structure} $(\star, 0)$ on $\cat{C\M et}_1$,
  whose remaining data is all lifted from the convex monoidal
  structure on $\cat{Set}$.
\end{Defn}
And so obtain:
\begin{Prop}
  \label{prop:16}
  The Kantorovich monad on $\cat{C\M et}_1$ is a linear exponential
  monad with respect to the join monoidal structure $(\star, 0)$ on
  $\cat{C\M et}_1$.
\end{Prop}
As such, we have a good notion of hypernormalisation for the
Kantorovich monad. Once again, the maps
$\N \colon \K(\Sigma_i A_i) \rightarrow \K(\Sigma_i \K A_i)$ are
defined in the expected way---but we obtain with no extra work the
fact that this is a contractive mapping, as required.

\subsection{Other probability monads}
\label{sec:other-prob-monads}

There are many other probability monads in existence; for example, the
Giry monads~\cite{Giry1982A-categorical} on the category of measurable
spaces, and on the category of Polish spaces; the probabilistic
powerdomain on the category of dcpos~\cite{Jones1989A-Probabilistic},
and on the category of continuous dcpos; and so on.

In each case it would be reasonable to try and derive a notion of
hypernormalisation following the pattern set out above. However, in
many of these other cases, we are hampered by a lack of a concrete
description of the monad algebras: while they are always some kind of
barycentric algebra, the precise structure involved is hard to pin
down. Without a concrete description of the algebras, we cannot give a
concrete description of their finite coproducts; and so cannot in this
way obtain the required linear exponential structure.

However, another approach is possible. As explained
in~\cite{Bierman1995What}, to obtain a linear exponential structure on
a monad $\mnd{T}$ with respect to a monoidal structure $(\otimes, I)$,
it suffices to exhibit suitably coherent isomorphisms
$T(A + B) \cong TA \otimes TB$ and $T(0) \cong I$. For other
probability monads, we can perfectly well do just this---but 
careful analysis will be necessary. For example, in the case of the
Giry monad on measurable spaces, we would need to exhibit the space
$\G(X+Y)$ of measures on a coproduct of measurable spaces as
obtained by combining in some manner the measurable spaces $\G X$ and
$\G Y$. At the level of underlying sets, this is easy: we have as usual
$\G(X+Y) \cong \G X + (0,1) \times \G X \times \G Y + \G Y$. But,
describing the $\sigma$-algebra of $\G(X+Y)$ in terms of those for $\G
X$ and $\G Y$ seems much harder---and so we leave this to future work.

\section{Combinatorial examples}
\label{sec:tric-effect-algebr}

In this section, we first consider further examples of our framework
that arise by replacing $(0,1)$ by a different symmetric tricocycloid
in the category of sets. 
One important situation that is not quite captured by this is that
where finite probability distributions on a set are replaced by
``logical distributions''---convex combinations whose coefficients are
drawn not from $[0,1]$ but from a given Boolean algebra $B$; however,
we will see that we \emph{can} capture this example by instantiating
our framework in a category other than the category of sets. Finally
in this section, we consider an example in which hypernormalisation
implements an extensional collapse for programs of type
$\mathsf{Stream}(A) \rightarrow B$.

\subsection{Other tricocycloids}
\label{sec:other-tricocycloids}
As noted above, our next examples of abstract hypernormalisation will
arise from other tricocycloids in the category of sets. To motivate
the manner in which this will happen, we first explain how we can
derive the finite discrete distribution monad $\mnd{D}$ from the
convex tricocycloid.

First of all, the convex tricocycloid induces the convex monoidal
structure $(\star, 0)$, and so, as with any symmetric monoidal
structure, we can consider the commutative $\star$-monoids. It is easy
to see that these are \emph{almost} abstract convex spaces: they are sets $A$
endowed with an operation $(0,1) \times A \rightarrow A \rightarrow A$
satisfying axioms (ii) and (iii), but not necessarily (i) from
Definition~\ref{def:5}.
The missing axiom (i) is the idempotency condition $r(a,a) = a$, which
we can capture abstractly as follows.

%

\begin{Defn}
  \label{def:13}
  A \emph{monoidal diagonal} for a symmetric monoidal category
  $(\C, \otimes, I)$ is a monoidal natural transformation
  $\delta \colon \id \Rightarrow \otimes \circ \Delta \colon \C
  \rightarrow \C$; this comprises a natural family of maps
  $\delta_A \colon A \rightarrow A \otimes A$ rendering
  commutative each diagram:
  \begin{equation}\label{eq:31}
    \cd[@C-1.5em]{
      & {A \otimes B} \ar[dl]_-{\delta_A \otimes \delta_B}
      \ar[dr]^-{\delta_{A \otimes B}} \\
      {A \otimes A \otimes B \otimes B} \ar[rr]^-{1 \otimes \sigma
        \otimes 1} & &
      {A \otimes B \otimes A \otimes B}\rlap{ .}
    }
  \end{equation}
  A commutative $\otimes$-monoid $(A, m, e)$ in $\C$ is
  \emph{idempotent} if, whenever $\delta$ is a monoidal diagonal for
  $(\C, \otimes, I)$, we have
  $m \circ \delta_A = \id_A \colon A \rightarrow A$.
\end{Defn}

%
%

When $(\C, \otimes, I)$ is $(\cat{Set}, \times, 1)$, the only monoidal
diagonal is the usual diagonal
$(\id, \id) \colon A \rightarrow A \times A$, and so idempotency in
the above sense coincides with idempotency in the usual sense.
On the other hand:
\begin{Lemma}
  \label{lem:13}
  A commutative monoid in $(\cat{Set}, \star, 0)$ is idempotent just
  when it satisfies axiom (i) as well as axioms (ii) and (iii) in
  Definition~\ref{def:5}---in other words, just when it is an abstract
  convex space.
\end{Lemma}
\begin{proof}
  By considering naturality with respect to maps $1 \rightarrow A$, we
  see that the possible monoidal diagonals
  $\delta_r \colon \id \Rightarrow \mathord{\star} \circ \mathord{\Delta}$ are indexed by
  $r \in [0,1]$; we have that $\delta_0, \delta_1$ are the left and
  right coproduct coprojections
  $A \rightarrow A + \big((0,1) \times A \times A\big) + A$, and that
  $\delta_r(a) = (r, a, a)$ for $0 < r < 1$. It follows that a
  commutative $\star$-monoid is idempotent just when it satisfies the
  additional axiom $r(a,a) = a$ for all $0 < r < 1$---that is, just
  when it is an abstract convex space.
\end{proof}

So from the convex tricocycloid, we can obtain abstract convex spaces as the
the idempotent commutative monoids for the associated monoidal
structure $(\star, 0)$, and obtain the monad $\mnd{D}$ as the free
idempotent commutative $\star$-monoid monad. More generally, for any
symmetric tricocycloid $H$ with structure maps
\begin{align*}
  v \colon H^2 & \mapsto H^2  & \gamma \colon H &
  \mapsto H \\
  (r,s) & \mapsto (rs, r \diamond s) & r &
  \mapsto r^\ast\rlap{ ,}
\end{align*}
we can consider idempotent commutative monoids for the associated
monoidal structure $\star_H$; these will be sets $A$ equipped with an
operation $H \times A \times A \rightarrow A$, written like before as
$r,a,b \mapsto r(a,b)$, that satisfies
\begin{equation*}
  r(a,a) = a \qquad r(a,b) = r^\ast(b,a) \qquad r(s(a,b),c) = (rs)\bigl(a, (r
  \diamond s)(b,c)\bigr)\rlap{ .}
\end{equation*}

Since this is clearly algebraic structure, we obtain a monad
$\mnd D_H$ on $\cat{Set}$ whose algebras are these ``$H$-convex
sets''. In fact, this monad is \emph{automatically} linear exponential
for the $\star_H$-monoidal structure, and so admits a notion of
hypernormalisation. To justify this last claim, we first use the well
known result of~\cite{Fox1976Coalgebras} that the tensor product on any
symmetric monoidal category $(\C, \otimes, I)$ lifts to the category
of commutative monoids, and there becomes coproduct; so free
commutative monoid monads are always linear exponential. Because we
consider \emph{monoidal} diagonals, the idempotent commutative monoids
are always closed under this lifted tensor product, so that free
\emph{idempotent} commutative monoid monads are also linear
exponential; in particular, any monad $\mnd{D}_H$ obtained as above is
linear exponential, as claimed.

We now illustrate this with some other examples of symmetric
tricocycloids in $\cat{Set}$. In giving these examples, we can exploit
the fact that $\mnd{D}_H$ is linear exponential to calculate the
action of the monads $\mnd{D}_H$. Indeed, it is clear that
$\D_H(1) \cong 1$, whence for any \emph{finite} set
$n \cong 1 + \dots + 1$ we have
$\D_H(n) \cong 1 \star_H \dots \star_H 1$; and to extend to
\emph{infinite} sets, we note that the theory of $H$-convex sets is
finitary, so that $\D_H(A)$ is the directed colimit of the sets
$\D_H(n)$ for all finite $n \subseteq A$.

\begin{Ex}
  \label{ex:14}
  Consider the one-element symmetric tricocycloid $1$. In this case,
  the induced monoidal structure is given by
  $A \star_1 B = A + A \times B + B$, which it is also fruitful to
  think of as
  \begin{equation}\label{eq:32}
    A \star_1 B = \Bigl((A + \{\bot\}) \times (B + \{\bot\})\Bigr) \setminus (\bot, \bot)\rlap{ .}
  \end{equation}
  In this case, the idempotent commutative
  $\star_1$-monoids are \emph{join-semilattices} (possibly
  without bottom element), the monad $\mnd D_1$ is the non-empty
  finite powerset monad, and hypernormalisation 
  $\N \colon \D_1(A+B) \rightarrow \D_1(\D_1 A + \D_1 B)$ is given by
  \begin{equation}
    \label{eq:30}
    \begin{aligned}
      \{a_1, \dots, a_n\} & \mapsto \{\{a_1, \dots, a_n\}\} \\
      \{b_1, \dots, b_m\} & \mapsto \{\{b_1, \dots, b_m\}\} \\
      \{a_1, \dots, a_n, b_1, \dots, b_m\} & \mapsto \{\{a_1,
          \dots, a_n\}, \{b_1, \dots, b_m\}\}\rlap{ .}
    \end{aligned}
  \end{equation}
  Computationally, going from $\mnd{D}$ to $\mnd{D}_1$ amounts to
  stepping back from probabilistic to non-deterministic (terminating)
  computation; we are interested in \emph{possibility} rather than
  \emph{probability}.
\end{Ex}

\begin{Ex}
  \label{ex:15}
  Let $X$ be a topological space, and let $H$ be the set of continuous
  functions $X \rightarrow (0,1)$ with the tricocycloid structure
  given pointwise as in $(0,1)$. In this case, a typical example of a
  $H$-convex set is the set of global sections of a sheaf of vector
  spaces on $X$; while the action of the monad $\mnd{D}_H$ is given by
  \begin{equation*}
    \D_H(A) = \{ S \subseteq A \text{ finite}, \omega \colon X
    \rightarrow \Delta_S \text{ continuous}\}
  \end{equation*}
  where $\Delta_S$ is the interior of the standard topological
  $\abs{S}$-simplex, given by a singleton when $\abs{S} = 1$, and by
  $\{r \in (0,1)^S : \sum_{x \in S} r(s) = 1\}$ otherwise. In other
  words, the elements of $\D_H(A)$ are those of $A$, together with all
  non-trivial ``finite convex combinations''
  \begin{equation*}
    \sum_{1 \leqslant i \leqslant n} f_i \cdot a_i
  \end{equation*}
  where each $a_i \in A$, and where the $f_i$'s are continuous maps
  $X \rightarrow (0,1)$ satisfying $\sum_i f_i = 1$ (i.e.,
  constituting a partition of unity). In this case, the notion of
  hypernormalisation carries over \emph{mutatis mutandis} from the
  motivating case, where arithmetic on the coefficients $f_i$ is done
  pointwise in $(0,1)$.
\end{Ex}


Before continuing, we take a slight detour in order to deepen our
understanding of tricocycloids. As we have just seen, we can obtain
the discrete distribution monad from the tricocycloid $(0,1)$.
However, as shown in~\cite{Jacobs2011Probabilities}, we may also
obtain it in an apparently different way: from the \emph{effect
  monoid} structure on~$[0,1]$.

\begin{Defn}
  \label{def:27}(cf.~\cite{Foulis1994Effect})
  A \emph{partial commutative monoid} is a set $M$
  with a constant $0$ and partial binary operation $\ovee
  \colon M \times M \rightharpoonup M$ satisfying the 
  axioms:
  \begin{gather*}
    r \ovee 0 \simeq r \simeq 0 \ovee r \qquad (r \ovee s) \ovee t
    \simeq r \ovee (s \ovee t) \qquad r \ovee s \simeq s \ovee r
  \end{gather*}
  where $\simeq$ denotes \emph{Kleene equality} of partially defined
  functions, i.e., one side is defined just when the other is, and
  they are then equal. $M$ is an \emph{effect
    algebra} if it is equipped with a constant
  $1$ and (total) unary operation $(\thg)^\bot$ such that:
  \begin{enumerate}[(i)]
  \item For all $r \in M$, the element $r^\bot$ is unique such that $r \ovee r^\bot \simeq 1$;
  \item If $r \ovee 1$ is defined, then $r = 0$.
  \end{enumerate}
  An effect algebra is an \emph{effect monoid} if it comes equipped
  with a (total) binary operation $r, s \mapsto r \cdot s$
  which is associative and has unit $1$, and which distributes over
  $\ovee$; i.e., we have equalities
  \begin{equation*}
    r \cdot 0 = 0 = 0 \cdot r \quad\ \  r \cdot (s \ovee t) \simeq' (r
    \cdot s) \ovee (r \cdot t) \ \ \quad 
    (r \ovee s) \cdot t \simeq' (r \cdot s) \ovee (r \cdot t) 
  \end{equation*}
  where $\simeq'$ means ``if the left-hand side is defined, then so is
  the right-hand side, and they are equal''.
\end{Defn}

\begin{Exs}
  \label{ex:4}
  \begin{enumerate}[(i),itemsep=0.25\baselineskip]
  \item $[0,1]$ is an effect monoid where $0,1$ have their
    usual meanings; $\cdot$ is ordinary multiplication; $r \ovee s$ is
    given by $r+s$ if this sum lies in $[0,1]$, and is undefined
    otherwise; and where $r^\bot = 1 - r$.
  \item Any Boolean algebra is an effect monoid, where $0,1$ are
    $\bot, \top$, where $\cdot$ is intersection $\wedge$, where $r
    \ovee s$ is given by $r \vee s$ if $r \wedge s = \bot$, and is
    undefined otherwise; and where $r^\bot$ is the complement of $r$.
  \end{enumerate}
\end{Exs}

As shown in~\cite{Jacobs2011Probabilities}, any effect monoid $M$
induces a monad $\mnd{D}_M$ on $\cat{Set}$ whose action on objects is
given by
\begin{equation}\label{eq:53}
  \D_M(A) = \{\, \omega \colon A \rightarrow M \mid
  \mathrm{supp}(\omega) \text{ finite and } \ov\displaylimits_{a \in A} \omega(a) \simeq 1\,\}
\end{equation}
and whose remaining structure is defined by analogy with the discrete
distribution monad; in particular, if $M = [0,1]$, we get the discrete
distribution monad itself.

\begin{Rk}
  \label{rk:4}
It is natural to ask how these two constructions of $\mnd{D}$---from
the tricocycloid $(0,1)$ and the effect monoid $[0,1]$---relate to
each other. More generally, we may ask whether the monad $\mnd{D}_M$
associated to an effect monoid $M$ also arises as the monad $\mnd{D}_H$
associated to some tricocycloid, or vice versa. This question was
investigated in detail by Kaddar~\cite{Kaddar2019Tricocycloids}. One
of his main results is that the assignments
\begin{equation*}
  M \mapsto M \setminus \{0,1\} \qquad \text{and} \qquad H \mapsto H
  \amalg \{0,1\}
\end{equation*}
give a bijection between tricocycloids satisfying certain
cancellativity properties, and effect monoids with
\emph{normalisation}, i.e., effect monoids such that for all
$a, b$ with $a \ovee b$ defined and $a \neq 1$, there is a
unique $c$ with $b = a^\bot c$. In particular, this construction
relates the tricocycloid $(0,1)$ and the effect monoid $[0,1]$.

An intuitive way of understanding this correspondence is by way of the
notion of pseudo-operad from Definition~\ref{def:26}. Any effect
monoid $M$ gives a symmetric pseudo-operad with underlying sets
\begin{equation}\label{eq:46}
  H_n = \{ (m_1, \dots, m_n) \in (M \setminus \{0,1\})^n : m_1 \ovee \dots \ovee m_n
  \simeq 1\}\rlap{ ,}
\end{equation}
and structure maps $\circ_i$ and $\sigma$ defined as
in~\eqref{eq:44} and~\eqref{eq:45}. By Lemma~\ref{lem:10}, this pseudo-operad
yields a tricocycloid just when each map
$\circ_i \colon H_n \times H_m \rightarrow H_{n+m-1}$ is invertible:
which is \emph{exactly} the condition that  $M$ admit normalisation.

Although this is not proven in detail
in~\cite{Kaddar2019Tricocycloids}, it flows from the above
understanding that the monad $\mnd{D}_M$ of an effect monoid with
normalisation coincides with the monad $\mnd{D}_H$ of the
corresponding tricocycloid; the point is that the $n$-ary ``$M$-convex
combination operations'' for $\mnd{D}_M$ can via normalisation be
decomposed into composites of \emph{binary} convex combinations.
\end{Rk}

\subsection{Logical hypernormalisation}
\label{sec:logic-hypern}

So far we have said nothing about the second example of an effect
monoid from Examples~\ref{ex:4}, that of a Boolean algebra. One might
hope this example to be associated to some kind of ``logical
hypernormalisation''. This does turn out to be the case, but there are
some subtleties, as we now explain.

\begin{Ex}
  \label{ex:13}
  Let $B$ be a non-trivial Boolean algebra; we may view $B$ as an
  effect monoid as in Examples~\ref{ex:4}, and so may form the
  associated monad $\mnd{D}_B$ as in~\eqref{eq:53}, with action on
  objects given by:
  \begin{equation*}
    \D_B(X) = \{\omega \colon X \rightarrow B \mid \omega \text{ has finite support and }\mathrm{im}(\omega) \text{ is a partition of $B$}\}\rlap{ .}
  \end{equation*}
  
  The Eilenberg--Moore algebras for $\mnd{D}_B$ admit a
  characterisation via binary operations due to
  Bergman~\cite[Theorem~14]{Bergman1991Actions}: they are sets $A$
  endowed with an operation $B \times A \times A \rightarrow A$,
  written $r,a,b \mapsto r(a,b)$ as usual, satisfying the following
  axioms:
\begin{enumerate}[(i)]
\item $r(a,a) = a$;
\item $r(a,b) = r^\bot(b,a)$;
\item $0(a,b) = b$;
\item $r(r(a,b),c) = r(a,c)$;
\item $r(s(a,b),b) = (rs)(a,b)$.
\end{enumerate}
But in fact, in the presence of (i)--(iii), axioms (iv)--(v) may be replaced by:
\begin{enumerate}[(i)]
\item[(iv)$'$] $r(s(a,b),c) = (sr)(a, r(b,c))$.
\end{enumerate}

The easier direction is that taking $b=c$ in (iv)$'$ yields (v), while
taking $s = r^\bot$ and using (iii) yields (iv). Conversely, if we
assume (iv) and (v), then we obtain (iv)$'$ by the calculation
\begin{align*}
  (sr)(a, r(b,c)) &= s(r(a,r(b,c)), r(b,c)) = s(r(a,c), r(b,c))\\
  &= r(s(a,b), s(c,c)) = r(s(a,b), c)
\end{align*}
using, in turn: (v); the equality $r(a,r(b,c)) = r(a,c)$ obtained from
(iv) and (ii); the non-trivial equality
$r(s(w,x),s(y,z)) = s(r(w,y),r(x,z))$
of~\cite[Proposition~11]{Bergman1991Actions}; and (i). Thus, if we
write $r(a,b)$ as $a +_r b$, then this alternate axiomisation becomes:
\begin{equation*}
  a +_r a = a \quad a +_r b = b +_{r^\bot} a
  \quad (a +_s b) +_r c = a +_{s r} (b +_r c) \quad a +_{0} b = b
\end{equation*}
which are the axioms U1--U3, U7 for the operation of \emph{guarded
  union} in the theory of guarded Kleene algebra with
tests~\cite{Smolka2019Guarded}.
\end{Ex}

\begin{Rk}
  \label{rk:3}
  The nice description of the $\mnd{D}_B$-algebras in this example
  does \emph{not} follow from the arguments of Remark~\ref{rk:4}.
  For indeed, seen as an effect monoid, $B$ does not admit normalisation,
  since in any effect monoid with normalisation, $r \neq 0$ and
  $rs = rt$ implies $s = t$, which is clearly not true in a Boolean
  algebra.

  However, we can give a more careful explanation as to why
  $\mnd{D}_B$ is generated by binary operations, by considering,
  again, the associated
  pseudo-operad of $B$ with underlying sets~\eqref{eq:46}. Since $B$
  does not admit normalisation, the $\circ_i$-maps of this
  pseudo-operad as in~\eqref{eq:44} are not invertible. However, they
  \emph{do} admit well-behaved sections $(\circ_i)^\ast$ given by
\begin{equation*}
  (t_1, \dots, t_{n+m-1}) \mapsto \bigl((t_1, \dots, t_{i-1},
  u, t_{i+m}, \dots, t_{n+m-1}), (u \Rightarrow t_i, t_{i+1}, \dots, t_{i+m-1})\bigr)
\end{equation*}
where we write $u$ for $\bigvee_{j=i}^{i+m-1} t_i$ and
$u \Rightarrow t_i$ for $\neg u \vee t_i$. Using thes these sections, we can obtain a map
$v \colon H_2 \times H_2 \rightarrow H_2 \times H_2$ much like
in~\eqref{eq:48} as the composite
\begin{equation*}
      \smash{H_2 \otimes H_2 \xrightarrow{\circ_1} H_3
    \xrightarrow{(\circ_2)^{\ast}} H_2 \otimes H_2\rlap{ ;}}
\end{equation*}
and while the $v$ so obtained is not invertible, it \emph{does}
satisfy the tricocycloid axiom~\eqref{eq:49}. This means that applying
the construction of Section~\ref{sec:tricocycloids} \emph{mutatis
  mutandis} yields not a monoidal structure, but a \emph{skew
  monoidal} structure in the sense
of~\cite{Szlachanyi2012Skew-monoidal}, with a
symmetry induced by the (invertible)
$(\thg)^\bot \colon H_2 \rightarrow H_2$. In this circumstance, there
is no problem in still considering idempotent commutative
monoids---and on doing, we recover Bergman's description of the
$\mnd{D}_B$-algebras via binary operations given above.
\end{Rk}

The fact that a Boolean algebra $B$ \emph{qua} effect monoid lacks
normalisation prevents us from carrying out ``logical
hypernormalisation'' for the monad $\mnd{D}_B$. Indeed, consider an
element $\omega \in \mathsf{D}_B(X+Y)$. We have the (complementary)
elements $\omega(X) = \bigvee_{x \in X} \omega(x)$ and
$\omega(Y) = \bigvee_{y \in Y} \omega(y)$ of $B$, and from this may
attempt to produce an element
$\N(\omega) \in \mathsf{D}_B(\mathsf{D}_BX + \mathsf{D}_BY)$. We focus
on the interesting case where neither $\omega(X)$ nor $\omega(Y)$ are
$\bot$; here, writing $\omega_X$ and $\omega_Y$ for the restriction of
$\omega$ to $X$ and $Y$, we would like to take
\begin{equation*}
  \N(\omega) = \omega(X) \cdot \overline{\omega_X}  + \omega(Y) \cdot \overline{\omega_Y}\rlap{ .}
\end{equation*}
The issue is that there is no obvious meaning to be assigned to
$\overline{\omega_X}$ or $\overline{\omega_Y}$. Indeed, $\omega_X$
represents a sum $\sum_{i} r_i \cdot x_i$ of elements of $X$ weighted
by disjoint elements $r_i \in B$ of total weight $\omega(X)$, and
there is no sensible way of distributing the missing weight
$\omega(Y)$ among the $r_i$'s to obtain a normalised distribution
$\overline{\omega_X}$. (More precisely: there is no way of doing so
which would make $\N$ into a natural transformation).

However, it turns out we \emph{can} describe a kind of logical
hypernormalisation by changing our perspective. Observe that elements
of $\mathsf{D}_B(X)$ can be seen as total computations which switch on
an element of the Boolean algebra $B$ before returning an element in
$X$. However, it also make sense to consider \emph{partial}
computations (like $\omega_X$ and $\omega_Y$ above) where one branch
of the switch is left undefined. In this situation, we can still
obtain a monad, but to do so we must allow not only computations but
also \emph{values} to be defined only on some part of $B$. Thus, we
must change both the monad \emph{and} the base category, as follows.

\begin{Defn}
  \label{def:14}
  Let $B$ be a non-trivial Boolean algebra. The category $\cat{Set}_B$
  of \emph{$B$-labelled sets} has:
  \begin{itemize}
  \item \textbf{Objects} being pairs
  $\mathbf{X} = (X, \abs{\thg}_{\mathbf{X}})$ where $X$ is a set and
  $\abs{\thg}_{\mathbf{X}} \colon X \rightarrow B \setminus \{\bot\}$;
\item \textbf{Morphisms} $\mathbf{X} \rightarrow \mathbf{Y}$ being
  functions $f \colon X \rightarrow Y$ such that
  $\abs{f(x)}_{\mathbf{Y}} = \abs{x}_{\mathbf{X}}$ for all $x \in X$.
  \end{itemize}
\end{Defn}
As suggested above, we think of a $B$-labelled set $\mathbf{X}$ as a
set of partially defined values, where the value $x$ is defined only
when the test $\abs{x}_\mathbf{X} \in B$---which we term the
\emph{domain} of $x$---evaluates to true. We choose to disallow
elements of domain $\bot$; the real reason for doing this is that, as
we shall see shortly, it allows the formulae we will write to resemble
more closely those for $\D$. However, for the moment, we may justify
the choice on the grounds that there should always be a \emph{unique}
totally undefined element---which, as such, need not be recorded
explicitly. \newcommand{\sdb}{\mathrm{s}\D_B}
\newcommand{\sDb}{\mathrm{s}\mnd{D}_B}
\begin{Defn}
  \label{def:29}
  Let $\mathbf{X}$ be a $B$-labelled set. A \emph{finitely supported
    logical subdistribution} on $\mathbf{X}$ is a function
  $\omega \colon X \rightarrow B$ such that $\mathsf{supp}(\omega)$ is
  non-empty and finite, the image of $\omega$ is pairwise-disjoint,
  and $\omega(x) \leqslant \abs{x}_\mathbf{X}$ for all $x \in X$. We
  may also write $\omega$ as a formal convex sum as in~\eqref{eq:28}.
  
  The \emph{domain} of a subdistribution $\omega$ is the value
  $\omega(X)$, where like before we write
  $\omega(A) = \bigvee_{a \in A} \omega(a)$ for any $A \subseteq X$.
  Of course, $\omega$ is a \emph{logical distribution} if it has
  domain $\top$. We write $\sdb(\mathbf{X})$ for the $B$-set of
  logical subdistributions on $\mathbf{X}$.
\end{Defn}

Note that we exclude the totally undefined logical subdistribution in
this definition. On the one hand, this is forced by the fact that our
labelled $B$-sets may not have elements of domain $\bot$; on the
other, we can justify this as capturing the fact that the totally
undefined subdistribution is ``unnormalizable'', in the sense that it
can't be expressed as a distribution on a non-trivial Boolean algebra.

The assignment $\mathbf{X} \mapsto \mathrm{s}\D_B(\mathbf{X})$
underlies a monad on $\cat{Set}_B$, whose action on objects is given
by pushforward as in~\eqref{eq:52}, and whose unit and multiplication
are given by
\begin{align*}
  \eta_\mathbf{X} \colon \mathbf{X} & \,\rightarrow\, \sdb(\mathbf{X}) & \mu_\mathbf{X} \colon \sdb\sdb(\mathbf{X}) & \,\rightarrow\, \sdb(\mathbf{X}) \\
  x & \,\mapsto\, \abs{x} \cdot x & \sum_{1 \leqslant i \leqslant n} \!\!r_i
  \cdot{\omega_i} & \,\mapsto\, \big(\,x \mapsto \!\!\bigvee_{1 \leqslant i \leqslant n} \!\!r_i \wedge
  \omega_i(x)\,\big)\rlap{ .}
\end{align*}

As with the finite distribution monad on $\cat{Set}$, the algebras for
the logical distribution monad on $\cat{Set}_B$ are well known. In
what follows, if $\mathbf{X}$ is a $B$-labelled set, then we write
$X(b) \subseteq X$ for the set of elements of domain $b$.

\begin{Defn}
  \label{def:30}
  Let $B$ be a Boolean algebra. A \emph{presheaf} on $B$ is a
  $B$-labelled set $\mathbf X$ endowed with restriction maps
  $\res {(\thg)} c \colon X(b) \rightarrow X(c)$ for all
  $c \leqslant b$ in $B \setminus \{\bot\}$, satisfying the evident
  functoriality axioms. A presheaf is a \emph{sheaf} if for any
  disjoint $b, c \in B \setminus \{\bot\}$, the following diagram is a
  product:
    \begin{equation*}
      X(b) \xleftarrow{\res{(\thg)} b} X(b \vee c) \xrightarrow{\res{(\thg)} c} X(c)\rlap{ .}
    \end{equation*}
    We write $\cat{Psh}(B)$ for the category of presheaves on $B$,
    whose maps are $B$-labelled set maps commuting with the
    restriction operations, and $\cat{Sh}(B) \leqslant \cat{Psh}(B)$
    for the full subcategory of sheaves.
\end{Defn}

\begin{Prop}
  \label{prop:6}
  The category of Eilenberg--Moore algebras of the logical
  subdistribution monad $\sDb$ is isomorphic over $\cat{Set}_B$ to the
  category of sheaves $\cat{Sh}(B)$.
\end{Prop}
This result is a simple exercise in the theory of sheaves; we include
a sketch proof for completeness.

\begin{proof}[Proof (sketch)]
  Given a sheaf $\mathbf X$, we endow its underlying $B$-labelled set
  with $\sDb$-algebra structure
  $\theta \colon \sdb(\mathbf{X}) \rightarrow \mathbf{X}$ by taking
  $\theta( \sum_i r_i \cdot x_i )$ to be the unique element
  $y \in X(\bigvee_i r_i)$ such that $\res {y} {r_i} = \res{x_i}{r_i}$
  for each $i \in I$. Conversely, if
  $\theta \colon \sdb(\mathbf{X}) \rightarrow \mathbf{X}$ is a
  $\D$-algebra, then we make its underlying $B$-set into a presheaf on
  $B$ by taking
  \begin{align*}
    \res{(\thg)} c \colon X(b) &\rightarrow X(c) \\
    x & \mapsto \theta(c \cdot x)\rlap{ .}
  \end{align*}
  
  For the sheaf axiom, if $b,c \in B \setminus \{\bot\}$ are disjoint,
  $x \in X(b)$ and $y \in X(c)$, then the unique $z \in X(b \vee c)$
  with $\res z b = x$ and $\res z c = y$ is given by
  $z = \theta(b \cdot x + c \cdot y)$.
\end{proof}

Given this result, we can exploit the well-known characterisation of
coproducts in categories of sheaves to give a concrete description of
finite coproducts in the category of
$\sDb$-algebras; alternatively, we can give a direct proof
paralleling Lemma~\ref{lem:2}. Either way, we have:
\begin{Lemma}
  \label{lem:12}
  The initial sheaf on $B$ is the empty $B$-labelled set $\mathbf{0}$;
  while if $\mathbf X$
  and $\mathbf Y$ are sheaves on $B$, then their coproduct
  $\mathbf X \star \mathbf Y$ in $\cat{Sh}(B)$ is the $B$-labelled set
  with elements of domain $b$ determined by:
  \begin{equation*}
    (\mathbf X \star \mathbf Y)(b) = X(b) \, + \Bigl(\,\sum_{\substack{b_1, b_2 \in B \setminus \bot \\ b = b_1 \ovee b_2}} X(b_1) \times Y(b_2)\Bigr) \ + \ Y(b)\rlap{ ,}
  \end{equation*}
  and whose sheaf restriction operation $(\mathbf X \star \mathbf
  Y)(b) \rightarrow (\mathbf X \star \mathbf Y)(c)$ is inherited from $X$ and
  $Y$ on the outer summands, and on the middle summand is
  given by
  \begin{align*}
    \sum_{\substack{b_1, b_2 \in B \setminus \bot\\b = b_1 \ovee b_2}} X(b_1) \times Y(b_2) & \rightarrow
    X(c) + \Bigl(\,\sum_{\substack{c_1, c_2 \in B \setminus \bot\\c = c_1 \ovee c_2}} X(r) \times Y(s)\Bigr) + Y(c)\\
    (x,y) & \mapsto
    \begin{cases}
      \mathsf{in}_1(\res x c) & \text{ if $c \leqslant b_1$;}\\
      \mathsf{in}_3(\res y c) & \text{ if $c \leqslant b_2$;}\\
    \mathsf{in}_2(\res{x}{b_1 \wedge c}, \res{y}{b_2 \wedge c}) & \text{ otherwise.}
    \end{cases}
  \end{align*}
\end{Lemma}
Since the underlying $B$-labelled set of the sheaf coproduct
$\mathbf{X} \star \mathbf{Y}$ relies only on the underlying
$B$-labelled sets of $\mathbf{X}$ and $\mathbf{Y}$, and not on their
sheaf structure, we may expect that $(\star, \mathbf{0})$ extends to a
monoidal structure on $\cat{Set}_B$ with respect to which $\sDb$ is
linear exponential. This is in fact the case. The only non-trivial
point is obtaining the unitality, associativity and symmetry
constraints for $\star$; and like before, these are determined by the
corresponding constraints for coproducts of sheaves---so long as these
descend to $\cat{Set}_B$.

We concentrate here on associativity. Note first that, if we define
$X(\bot)$ to be a singleton set for any $B$-labelled set $\mathbf X$,
then the binary tensor can be written as
\begin{equation*}
  (\mathbf{X} \ast \mathbf{Y})(b) = \sum_{b = b_1 \ovee b_2} X(b_1) \times X(b_2)\rlap{ .}
\end{equation*}

With this convention, if $\mathbf{X}$, $\mathbf{Y}$ and $\mathbf{Z}$ are $B$-labelled sets,
then the three-fold tensors
$(\mathbf{X} \star \mathbf{Y}) \star \mathbf{Z}$ and
$\mathbf{X} \star (\mathbf{Y} \star \mathbf{Z})$ satisfy
\begin{align*}
  ((\mathbf{X} \star \mathbf{Y}) \star \mathbf{Z})(b) & = 
  \sum_{b = q \ovee t} (X \star Y)(q) \times Z(t) = 
  \sum_{b= q \ovee t} \Bigl(\sum_{q = r \ovee s} X(r) \times Y(s)\Bigr) \times Z(t)\\
  (\mathbf{X} \star (\mathbf{Y} \star \mathbf{Z}))(b) & = 
  \sum_{b= r \ovee u} X(r) \times (Y \star Z)(u) = 
  \sum_{b =r \ovee u} X(r) \times \Bigl(\sum_{u = s \ovee t} Y(s) \times Z(t)\Bigr)
\end{align*}
which are isomorphic to each other by way of the set
\begin{equation*}
  \sum_{b = r \ovee s \ovee t} X(r) \times Y(s) \times Z(t)\rlap{ .}
\end{equation*}
Assembling these isomorphisms as $b \in B$ varies gives the desired
natural isomorphisms
$(\mathbf{X} \star \mathbf{Y}) \star \mathbf{Z} \rightarrow \mathbf{X}
\star (\mathbf{Y} \star \mathbf{Z})$. It is now straightforward to
verify the Mac~Lane coherence axioms making this into a a symmetric
monoidal structure on $\cat{Set}_B$, and to see that this structure
lifts to the coproduct monoidal structure on the category of
$\sDb$-algebras (i.e., sheaves on $B$). Thus we have shown:

\begin{Prop}
  \label{prop:19}
  The logical subdistribution monad $\sDb$ is linear exponential with
  respect to the monoidal structure $(\star, \mathbf{0})$ on $\cat{Set}_B$.
\end{Prop}

The induced hypernormalisation maps
$\N \colon \sdb(\Sigma_i X_i) \rightarrow \sdb(\Sigma_i \sdb (X_i))$
can be described as folows. An element of $\sdb(\Sigma_i X_i)$ of
domain $b$ is a function $\omega \colon \Sigma_i X_i \rightarrow B$ of
finite support whose image is a partition of $b \in B$, such that
$\omega(x) \leqslant \abs{x}_{\mathbf{X}_i}$ for all $x \in X_i$. For
each $i$, we have the elements $\omega(X_i) \in B$, which themselves
constitute an $I$-fold partition of $b$; and we also have for each $i$
the restricted function
$\omega_i = \res{\omega}{X_i} \colon X_i \rightarrow B$, which, so
long as $\omega(X_i) \neq \bot$, is an element of $\D X_i$ of domain
$\omega(X_i)$. As such, we can define $\N(\omega)$ to be the element
\begin{equation*}
  \sum_{\substack{i \in I \\ \omega(X_i) \neq \bot}} \omega(X_i) \cdot \iota_i(\omega_i) \in \D(\Sigma_i \D X_i)\rlap{ .}
\end{equation*}

In this way, we have sidestepped the problem of normalising the
subdistributions $\omega_i$ to total ones, as we would have needed to
do for the monad $\mnd{D}_B$ on $\cat{Set}$; instead, we can allow
them to remain as subdistributions, and so obtain a good notion of
hypernormalisation.

One way of seeing this is that we have decoupled the two aspects of
hypernormalisation of probability distributions from each other: on
the one hand, collecting like terms to obtain an outer distribution;
and on the other, normalising the inner sub-distributions to genuine
distributions. Our logical hypernormalisation does the former, but not
the latter; and while this may seem to limit its value, we will see in
Section~\ref{sec:appl-stre-proc} that this is not the case.

\subsection{Normalisation of continuous functions on streams}
\label{sec:norm-cont-funct}

For our final example, we describe an instance of hypernormalisation
of a different flavour; in it, the hypernormalisation map will allow
us to normalise programs encoding continuous functions defined on a
type of streams.

By a \emph{stream} over a finite alphabet $B$, we mean simply an
$\mathbb{N}$-indexed family
$\vec b = (b_0, b_1, \dots) \in B^\mathbb{N}$. The set
$\mathsf{Stream}(B)$ of streams can be topologised, as usual, as a
product of discrete spaces, and then a function
$f \colon \mathsf{Stream}(B) \rightarrow A$ to a finite discrete space
$A$ is continuous precisely when it is \emph{locally constant}:
\begin{equation*}
  \forall \vec b \in \mathsf{Stream}(B).\, \exists n \in \mathbb{N}.\,
  \forall \vec c \in \mathsf{Stream}(B).\, b_{\leqslant n} = c_{\leqslant n}
  \ \Longrightarrow\  f(\vec b)= f(\vec c)\rlap{ ;}
\end{equation*}
where we define $b_{\leqslant n} = (b_0, \dots, b_n)$. In other words,
a continuous $f$ uses only finitely many tokens of its input to
compute its output.

It is an idea which goes back to Brouwer (though
see~\cite{Hancock2009Representations} for a modern treatment) that
such functions can be represented, non-uniquely, by well-founded
$B$-ary branching trees with leaves labelled in $A$, i.e., elements of
the initial algebra $T_B(A) = \mu X.\, X^B + A$. For example, the
binary tree to the left in:
\begin{equation*}
  \cd[@-1.3em]{
    & b & & c& & b & & b\\
    a & & {\ast \ar@{-}[ul]^-{0} \ar@{-}[ur]_-{1}} & & b & & {\ast \ar@{-}[ul]^-{0} \ar@{-}[ur]_-{1}}\\
    & {\ast \ar@{-}[ul]^-{0} \ar@{-}[ur]_-{1}} & & & & {\ast \ar@{-}[ul]^-{0} \ar@{-}[ur]_-{1}}\\
    & & & \ast \ar@{-}[ull]^-{0} \ar@{-}[urr]_-{1} 
  } \qquad \qquad 
  \cd[@-1.3em]{
    & b & & c\\
    a & & \ast \ar@{-}[ul]^-{0} \ar@{-}[ur]_-{1}\\
    & \ast \ar@{-}[ul]^-{0} \ar@{-}[ur]_-{1} & & b\\
    & & \ast \ar@{-}[ul]^-{0} \ar@{-}[ur]_-{1} 
  }
\end{equation*}
is a decision tree encoding the function
$f \colon \mathsf{Stream}(\{0,1\}) \rightarrow \{a,b,c\}$ with
\begin{equation*}
  f(00 \dots) = a\text{ ,} \qquad f(010 \dots) = f(1\dots) =
  b\text{ ,} \quad \text{and} \quad f(011\dots) = c\rlap{ .}
\end{equation*}

However, this tree is an inefficient encoding, since, on receiving an
first input token of $1$, we request needless additional input
tokens before returning the output value $b$. There is an obvious
algorithm which normalises such decision trees to maximally efficient
ones, in this case the tree right above: starting from the leaves, we
recursively contract all subtrees as left below to leaves as to the right:
\begin{equation}\label{eq:42}
\cd[@-1.5em]{x \ar@{-}[dr] & & \ar@{-}[dl] x \\ & \ast} \qquad
\rightsquigarrow \qquad x\rlap{ .}
\end{equation}

While this normalisation algorithm is clearly motivated from a
computational perspective, it is less clear how to understand it
structurally. We will show that, in fact, it arises as an instance of
hypernormalisation.

To see it in this way, we consider the monad $\mnd{T}_B$ whose value
at $A$ is the set $T_B(A)$ of $B$-ary branching trees with leaves in
$A$. This is the \emph{free monad} on the endofunctor $X \mapsto X^B$
of $\cat{Set}$, so that $\mnd{T}_B$-algebras are exactly
\emph{$B$-ary magmas}: sets $X$
endowed with a $B$-ary operation $X^B \rightarrow X$. We will show
that $\mnd{T}_B$ is a linear exponential monad on $\cat{Set}$ for a
suitable monoidal structure, and so admits hypernormalisation. From
this, we obtain for any finite set $A$ a function
\begin{equation}\label{eq:41}
T_B(A) \cong  T_B(\Sigma_{a \in A} 1) \xrightarrow{\N}
  T_B\bigl(\Sigma_{a \in A} T_B 1\bigr) \xrightarrow{\! T_B(\Sigma_{a \in A}
    !)\!}
  T_B(\Sigma_{a \in A} 1) \cong T_B(A)
\end{equation}
which we will see implements the normalisation algorithm
described above.

To find the monoidal structure on $\cat{Set}$ for which $\mnd T_B$ is
linear exponential, the first step, as usual, is to characterise
finite coproducts of $\mnd T_B$-algebras---so equally, finite
coproducts of $B$-ary magmas. Clearly, the empty set underlies the
initial $B$-ary magma. The following result characterises binary
coproducts of $B$-ary magmas; in its statement, and henceforth, if
$\tau \in T_B(A)$ and $A' \subseteq A$, then we say that $\tau$ is
\emph{$A'$-labelled} if it lies in $T_B(A') \subseteq T_B(A)$.

\begin{Lemma}
  \label{lem:9}
  Let $(X, \xi \colon X^B \rightarrow X)$ and
  $(Y, \gamma \colon Y^B \rightarrow Y)$ be $B$-ary magmas. Their
  coproduct in the category of $B$-ary magmas is given by
  \begin{equation*}
    X \odot Y \defeq \{\,\tau \in T_B(X + Y) : \tau \text{ has no
      non-trivial $X$- or $Y$-labelled subtree}\,\}\rlap{ .}
  \end{equation*}
  The magma operation $\theta \colon (X \odot Y)^B \rightarrow X \odot Y$ takes
  a family of trees $(\tau_b : b \in B)$ to the disjoint union of the
  $\tau_b$'s, joined together at a fresh root vertex,
  \begin{equation*}
    \theta(\tau_b : b \in B) = 
    \cd{
      \tau_b \ar@{-}[dr]^-{}&  \dots & \tau_{b'}\rlap{, } \ar@{-}[dl]_-{} \\ & \bullet
    }
  \end{equation*}
  except in the cases where this would create a non-trivial $X$- or
  $Y$-labelled subtree. These exceptional cases are where:
  \begin{itemize}
  \item For each $b \in B$, the tree $\tau_b$ is a one-vertex tree
    $\bullet_{\textstyle x_b}$ labelled by some $x_b \in X$; in this case, we take $\theta(\tau_b :
    b \in B) = \xi(x_b : b \in B)$.
  \item For each $b \in B$, the tree $\tau_a$ is a one-vertex tree
    $\bullet_{\textstyle y_b}$ labelled by some $y_b \in Y$; in this case, we take $\theta(\tau_b :
    b \in B) = \gamma(y_b : b \in B)$.
  \end{itemize}
\end{Lemma}
\begin{proof}
  By its construction, the magma operation on $X \odot Y$ is
  well-defined, and the two maps
  $\iota_1 \colon X \rightarrow X \odot Y \leftarrow Y \colon \iota_2$
  sending an element of $X$ or $Y$ to the corresponding one-vertex
  labelled tree are magma homomorphisms. Moreover, given $B$-ary magma
  morphisms $f \colon X \rightarrow Z$ and $g \colon Y \rightarrow Z$,
  the unique magma morphism $\spn{f,g} \colon X \odot Y \rightarrow Z$
  with $\spn{f,g} \iota_1 = f$ and $\spn{f,g}\iota_2 = g$ is given as
  the composite of the inclusion $X \odot Y \hookrightarrow T_B(X+Y)$
  with the unique $B$-ary magma morphism $T_B(X+Y) \rightarrow Z$
  extending the function $\spn{f,g} \colon X+Y \rightarrow Z$.
\end{proof}

Note that the set $X \odot Y$ underlying the coproduct of $(X, \xi)$
and $(Y, \gamma)$ does not rely on $\xi$ and $\gamma$, but only on the
underlying sets $X$ and $Y$. So, like before, we may posit the
existence of a symmetric monoidal structure $(\odot, 0)$ on
$\cat{Set}$ for which $\mnd T_B$ is linear exponential. Once again,
the only point requiring work is defining the associativity, unitality
and symmetry isomorphisms, and, like before, we concentrate on the case of associativity.
For this, we follow the idea of Lemma~\ref{lem:10} by interposing a
ternary tensor product $X \odot Y \odot Z\subseteq T_B(X+Y+Z)$
composed of those trees without any non-trivial $X$-, $Y$- or
$Z$-labelled subtree.
\begin{Lemma}
  \label{lem:11}
  For any sets $X,Y,Z$ we have isomorphisms
\begin{equation*}
  (X \odot Y) \odot Z \xrightarrow{\ \ell\ } X \odot Y \odot Z
  \xleftarrow{\ r\ } X \odot
  (Y \odot Z)
\end{equation*}
\end{Lemma}
\begin{proof}[Proof (sketch)]
  An element $\tau \in (X \odot Y) \odot Z$ is an
  $(X \odot Y) + Z$-labelled tree, and each $X \odot Y$-leaf is itself
  an $X+Y$-labelled tree. These data are equally encapsulated by an
  $X+Y+Z$-labelled tree with a collection of vertices marked: namely,
  the roots of the $X \odot Y$-trees at the leaves of
  the original $\tau$. Forgetting this vertex marking yields the
  element $\ell(\tau) \in X \odot Y \odot Z$; and to show invertibility of
  the $\ell$ so defined, we need to reconstruct $\tau$'s marking
  uniquely from $\ell(\tau)$. But this is easy;
  it is uniquely characterised by the following properties:
  \begin{enumerate}[(i)]
  \item The subtree above every marked vertex is $X+Y$-labelled;
  \item Every leaf vertex is above some marked vertex;
  \item No non-leaf vertex has all of its $B$ children marked,
  \end{enumerate}
  and we can obtain it via the following algorithm: first mark each
  $X$- or $Y$-leaf; then recursively move markings towards the root
  until (ii) is satisfied. This must terminate by well-foundedness.
  This defines the invertible $\ell$; now $r$ is dual.
\end{proof}
We can thus take
$r^{-1} \ell \colon (X \odot Y) \odot Z \rightarrow X \odot (Y \odot
Z)$ as the desired associativity constraint. The pentagon axiom
equating the two maps
$((X \odot Y) \odot Z) \odot W \rightarrow X \odot (Y \odot (Z \odot
W))$ now follows in the spirit of Lemma~\ref{lem:10} from the
observation that each edge in this pentagon is simply a way of
redistributing markings on a particular element of the quaternary tensor
$X \odot Y \odot Z \odot W$.
Proceeding similarly for the unit and symmetry constraints, we may
complete the construction of the symmetric monoidal structure
$(\odot, 0)$ on $\cat{Set}$ and see that it lifts to the coproduct
monoidal structure on the category of $\mnd{T}$-algebras. In other
words, we have:

\begin{Prop}
  \label{prop:2}
  The $B$-ary magma monad $\mnd{T}_B$ is linear exponential with
  respect to the monoidal structure $(\odot, 0)$ on $\cat{Set}$.
\end{Prop}

The associated hypernormalisation maps
$\N \colon T_B(\Sigma_i X_i) \rightarrow T_B\bigl(\Sigma_i T_B(X_i)\bigr)$ may
be described as follows. An element $\tau \in T_B(\Sigma_i X_i)$ is a
$\Sigma_i X_i$-labelled $B$-ary tree. There is a unique way of marking
vertices in $\tau$ such that:
\begin{enumerate}[(i)]
\item The subtree above any marked vertex is $X_i$-labelled for some
  $i$;
\item No vertex has all $B$ of its children marked.
\end{enumerate}
On constructing this marking, the subtree above each marked vertex is
an element of $\Sigma_i T_B(X_i)$; so viewing each such subtree as a
leaf labelled in $\Sigma_i T_B(X_i)$, we have obtained the element
$\N(\tau) \in T_B(\Sigma_i T_B(X_i))$.

More intuitively, if we think of 
$\tau \in T_B(\Sigma_i X_i)$ as a decision tree computing a continuous
function $f \colon \mathsf{Stream}(B) \rightarrow \Sigma_i X_i$, then
$\N(\tau) \in T_B\bigl(\Sigma_i T_B(X_i)\bigr)$ computes 
a function $f' \colon \mathsf{Stream}(B) \rightarrow \Sigma_i T_B(X_i)$ as
follows: given $S \in \mathsf{Stream}(B)$, we run the computation of
$f(S)$ using $\tau$, and halt at the precise moment that the summand
$X_i \subseteq \sum_i X_i$ in which $f(S)$ lies has been determined.
We then return as $f'(S)$ the $X_i$-labelled subtree lying above the
halting vertex, i.e., the continuation of the computation of $f(S)$ as
an element of the set $X_i$.

We now use our understanding of hypernormalisation to describe the map
$T_B(A) \rightarrow T_B(A)$ of~\eqref{eq:41}. This first applies
$\N \colon T_B(\Sigma_{a \in A}1) \rightarrow T_B(\Sigma_{a \in
  A}T_B1)$, whose effect on $\tau \in T_B(A)$ is to mark the roots of
the largest subtrees whose leaves are all labelled with a single element
$a \in A$. It then applies the function
$T_B(\Sigma_a !) \colon T_B(\Sigma_{a \in A}T_B1) \rightarrow
T_B(\Sigma_{a \in A}1)$, which has the effect of collapsing the marked
vertex at the root of each $\{a\}$-labelled subtree to the bare leaf $a$.
The endofunction of $T_B(A)$ so resulting acts on a tree $\tau$
precisely by carrying out the contractions in~\eqref{eq:42}---in other
words, it normalises $\tau$ to its most efficient representative, as desired.

\section{Application: relating behavioural and trace equivalence}
\label{sec:appl-stre-proc}

In this final section, we use our framework for abstract
hypernormalisation to relate behavioural equivalence and trace
equivalence for certain kinds of automata; more precisely, we will use
it to describe a normalisation-by-evaluation process
which normalises \emph{behaviours}---states modulo bisimilarity---to
maximally efficient representatives of the corresponding
\emph{traces}---states modulo trace equivalence.


\subsection{Generative systems}
\label{sec:generative-systems}

We begin by introducing the kinds of automata that we will be
concerned with.

\begin{Defn}
  \label{def:31}
  Let $\mnd T$ be a monad on 
  a category $\C$ with finite coproducts, and $A$ a finite set. A
  \emph{generative $\mnd T$-system} with alphabet $A$ is an object
  $S \in \C$ endowed with a map
  $\sigma \colon S \rightarrow T(\Sigma_{a \in A} S)$. We write
  $\cat{Gen}_A(\mnd T)$ for the category of generative
  $\mnd T$-systems over $A$, with as maps the obvious homomorphisms.
\end{Defn}

\begin{Ex}
  \label{ex:1}
When $\mnd T$ is the identity monad on $\cat{Set}$, this
definition yields \emph{deterministic} generative systems
$\sigma \colon S \rightarrow A \times S$. We see $S$ as a set of
states, and $\sigma$ as associating to each state in $S$ an output
token from $A$, and a next state in~$S$; so we have a degenerate
kind of Mealy machine.
\end{Ex}
\begin{Ex}
  \label{ex:2}
On the other hand, for the non-empty finite powerset monad on
$\cat{Set}$, a generative $\mnd P_f^+$-system
$\sigma \colon S \rightarrow \P_f^+(A \times S)$ is a
non-terminating, finitely branching \emph{labelled transition system}.
\end{Ex}
We will use $\mnd T = \mathsf{id}_\cat{Set}$ and $\mnd T = \mnd P_f^+$
as our running examples in the next few sections; in
Sections~\ref{sec:stream-processors} and~\ref{sec:prob-logic-gener}
below, we will consider other monads capturing \emph{stream
  processors}~\cite{Hancock2009Representations, Garner2021Stream},
\emph{probabilistic generative systems}~\cite{Glabbeek1990Reactive},
and a ``logical'' version of generative systems in the spirit of
Section~\ref{sec:logic-hypern} above.

\subsection{Behaviours and traces}
\label{sec:behaviours-traces}

For our general notion of generative system, there are two natural
notions of equivalence between states. On the one hand, we have
\emph{bisimilarity}; for example, in the case of a labelled transition
$\sigma \colon S \rightarrow \P_f^+(A \times S)$, states $s,t \in S$
are bisimilar if they are related by a bisimulation $R$ on $S$, i.e.,
an equivalence relation with
\begin{equation}\label{eq:54}
  x \mathrel{R} y \text{ and } (a,x') \in \sigma(x) \implies \exists y'.\, x' \mathrel{R} y' \text{ and } (a,y') \in \sigma(y)\rlap{ .}
\end{equation}
On the other hand, we have \emph{trace equivalence}; for labelled
transition systems, the \emph{trace} of a state $s_0 \in S$ is the
subset $\mathsf{tr}(s_0) \subseteq \mathsf{Stream}(A)$ satisfying
\begin{equation}\label{eq:55}
  \mathsf{tr}(s_0) = \{(a_0, a_1, \dots) \mid \exists (s_1, s_2, \dots) \text{ s.t. } (a_i, s_{i+1}) \in \sigma(s_i) \,\forall i\}
\end{equation}
and two states are deemed trace equivalent if they have the same
trace.

We now provide the general definitions of bisimilarity and trace
equivalence for generative $\mnd T$-systems. We begin with the
better-known case of bisimilarity:

\begin{Defn}
  \label{def:34}
  Let $\mnd T$ be a monad on a category $\C$ with finite coproducts
  and $A$ a finite set. An \emph{object of behaviours} for generative
  $\mnd{T}$-systems over $A$ is a final object
  $\beta \colon \mathsf{Beh} \rightarrow T(\Sigma_{a \in A}
  \mathsf{Beh})$ in the category $\cat{Gen}_A(\mnd T)$. The
  \emph{behaviour map} for a generative $\mnd T$-system
  $\sigma \colon S \rightarrow T(\Sigma_{a \in A} S)$ is the unique
  homomorphism
  $\mathsf{beh} \colon (S, \sigma) \rightarrow (\mathsf{Beh}, \beta)$.
  Two states are \emph{bisimilar} if they have the same image under
  the behaviour map.
\end{Defn}

\begin{Ex}
  \label{ex:3}
  When $\mnd T = \mnd{id}_\cat{Set}$, the object of behaviours is
  $\mathsf{Stream}(A)$ with the structure map $(\mathsf{hd},
  \mathsf{tl}) \colon \mathsf{Stream}(A)
  \rightarrow A \times \mathsf{Stream}(A)$ given by 
  \begin{equation*}
     (a_0, a_1, \dots)  \mapsto \bigl(a_0, (a_1, a_2, \dots)\bigr)\rlap{ .}
  \end{equation*}
  Given a generative system
  $\sigma = (\sigma_0, \sigma_1) \colon S \rightarrow A \times S$ and
  a state $s \in S$, we have
  \begin{equation*}
    \mathsf{beh}(s) = (\sigma_0(s), \sigma_0(\sigma_1(s)), \sigma_0(\sigma_1(\sigma_1(s))), \dots)\rlap{ .}
  \end{equation*}
\end{Ex}
\begin{Ex}
  \label{ex:5}
  When $\mnd T = \mnd{P}_f^+$, the object of behaviours can be
  described via the techniques of~\cite{Aczel1989Final}. Consider the
  set $\mathsf{Beh}'$ of rooted trees which are finitely branching,
  purely infinite (i.e., with no leaf vertices), and have edges
  labelled by elements of $a$. This becomes a labelled transition system
  $\mathsf{Beh}' \rightarrow \P_f^+(A \times
  \mathsf{Beh}')$~via
  \begin{equation*}
    \cd[@-1em]{ \tau_1 \ar@{-}[dr]_-{a_1}& \dots &
        \tau_n\rlap{, } \ar@{-}[dl]^-{a_n} \\ & \bullet }  \qquad \mapsto\qquad  \{(a_1,
    \tau_1), \dots, (a_n, \tau_n)\}
  \end{equation*}
  and the object of behaviours is the quotient of $\mathsf{Beh}_w$ by
  the largest bisimulation as in~\eqref{eq:54}. Alternatively,
  following~\cite{Worrell2005Final, Adamek2015Final}, we can describe
  it as the set of \emph{strongly extensional} trees of
  $\mathsf{Beh}_w$ (i.e., those admitting no tree bisimulation).
  
  Given a labelled transition system $\sigma \colon S \rightarrow
  \P_f^+(A \times S)$ and $s \in S$, the behaviour $\mathsf{beh}(s)$
  is given coinductively by
  \begin{equation*}
    \mathsf{beh}(s) = \cd[@-1em]{ \mathsf{beh}(t_1) \ar@{-}[dr]_-{a_1}& \dots &
        \mathsf{beh}(t_1)\rlap{, } \ar@{-}[dl]^-{a_n} \\ & \bullet } \qquad \text{where } \sigma(s) = \{(a_1, t_1), \dots, (a_n, t_n)\}\rlap{ .}
  \end{equation*}
\end{Ex}

We now turn to trace equivalence. As for bisimulation, this will be
characterised in terms of equality under a map, this time to a
suitably-defined object of traces.

\begin{Defn}
  \label{def:32}
  Let $\C$ be a category with finite coproducts, and $A$ a finite set.
  An \emph{$A$-ary comagma} in $\C$ is an object $X$ endowed with a
  map $X \rightarrow \sum_{a \in A} X$; we write $\cat{Comag}_A(\C)$
  for the category of $A$-ary comagmas in $\C$ and their
  homomorphisms. Given a monad $\mnd T$ on $\C$, an \emph{object of
    traces} for generative $\mnd T$-systems over $A$ is defined to be
  a final object
  $\tau \colon \mathsf{Tr} \rightarrow \Sigma_{a \in A} \mathsf{Tr}$
  in the category $\cat{Comag}_A(\C^\mnd{T})$.
\end{Defn}

To motivate this definition, note that a generative $\mnd T$-system
$\sigma \colon S \rightarrow T(\sum_{a \in A} S)$ is equally well an
$A$-ary comagma in the Kleisli category $\cat{Kl}(\mnd T)$; and it is
an idea going back to~\cite{Power1999Coalgebraic} that a suitable
notion of object of traces in this context should be given by a final
object in $\cat{Comag}_A(\cat{Kl}(\mnd T))$. Note this is
\emph{different} from a final object in $\cat{Gen}_A(\mnd T)$, since
morphisms $(S, \sigma) \rightarrow (S, \sigma')$ in the latter
category involve maps $S \rightarrow S'$ in $\C$, while those in the
former category involve maps $S \rightarrow T(S')$.

However, the category $\Kl{\mnd T}$ is not well-behaved, and a final
$A$-ary comagma is \emph{not} guaranteed to exist. Often, this problem
can be resolved by expanding the category $\C$ and the monad
$\mnd T$---see, for example,~\cite{Hasuo2007Generic,
  Kerstan2013Coalgebraic}---but this must be done in an \emph{ad hoc}
manner which requires thought. However, the category $\Kl{\mnd T}$ can
always be embedded into the much better-behaved category $\C^{\mnd T}$
of Eilenberg--Moore algebras---as the free algebras therein---and this
better behaviour now guarantees, under very mild side-conditions, the
existence of a final $A$-ary comagma, which we may thus declare to be
our object of traces.

\begin{Ex}
  \label{ex:7}
 When 
    $\mnd T = \mnd{id}_\cat{Set}$, the category of $A$-ary comagmas in
    $\cat{Set}^\mnd{id}$ is precisely the category of deterministic
    generative systems
    $S \rightarrow A \times S$. So the object of traces is simply
    the object of behaviours 
    $(\mathsf{Stream}(A), (\mathsf{hd}, \mathsf{tl}))$.
\end{Ex}
\begin{Ex}
  \label{ex:8}
  When $\mnd T = \mnd{P}_f^+$, an $A$-ary comagma in
  $\cat{Set}^{\mnd{P}_f^+}$ is a join-semilattice $S$ endowed with a
  $\vee$-preserving map $S \rightarrow S \star_1 \cdots \star_1 S$,
  where $\star_1$ is as in Example~\ref{ex:14}. In light of the
  alternative presentation~\eqref{eq:32} of this monoidal product,
  such a map can equally be described as a $\vee$-preserving map
  $\sigma \colon S \rightarrow (S_\bot)^A$---where $S_\bot$ is $S$
  with a new bottom element adjoined---such that no
  $\sigma(s)$ is constant at $\bot$. In this context, the following
  proposition describes the object of traces.
\end{Ex}

\begin{Prop}
  \label{prop:23}
  The final $A$-ary comagma in the category $\cat{Set}^{\mnd{P}_+^f}$
  of join-semilattices is the set $\mathsf{Tr}$ of
  non-empty closed subsets of $\mathsf{Stream}(A)$, with
  join-semilattice structure given by union, and with comagma
  structure map
  \begin{align*}
    \tau \colon \mathsf{Tr} &\rightarrow (\mathsf{Tr} \cup \{\emptyset\})^A\\
    \U & \mapsto a_0 \mapsto \{(a_1, a_2, \dots) \mid (a_0, a_1, \dots) \in \U \}\rlap{ .}
  \end{align*}
\end{Prop}
The \emph{closed} sets here are those which are closed in the product
topology. To make this explicit, recall that for
$\vec a \in \mathsf{Stream}(A)$, we write $a_{\leqslant n}$ for the
initial segment $(a_0, \dots, a_n)$. If, for
$V \subseteq \mathsf{Stream}(A)$, we also write $V_{\leqslant n}$
for $\{a_{\leqslant n} : \vec a \in V\}$, then a set
$V \subseteq \mathsf{Stream}(A)$ is closed if $\vec a \in V$ whenever
$a_{\leqslant n} \in V_{\leqslant n}$ for all $n \in \mathbb{N}$.
\begin{proof}
  Consider an $A$-ary comagma $\sigma \colon S \rightarrow (S_\bot)^A$
  in $\cat{Set}^{\mnd{P}_+^f}$. We may extend $\sigma$ to a $\vee$-
  and $\bot$-preserving map $S_\bot \rightarrow (S_\bot)^A$ by
  defining $\sigma(\bot)(a) = \bot$; and now, given $s_0 \in S$ and
  $\vec a \in \mathsf{Stream}(A)$, we may define a stream
  $\sigma^\ast(s_0, \vec a) = (s_0, s_1, \dots) \in
  \mathsf{Stream}(S_\bot)$ by taking $s_{i+1} = \sigma(s_i)(a_i)$ for
  each $i$. We will say that $s_0$ \emph{generates} $\vec a$ if no
  element of $\sigma^\ast(s_0, \vec a)$ is $\bot$.

  With these conventions, we can now define the map
  $u \colon S \rightarrow \mathsf{Tr}$ by
  \begin{equation}\label{eq:27}
    u(s) = \{\vec a \in \mathsf{Stream}(A) \mid s \text{ generates } \vec a\}\rlap{ .}
  \end{equation}

  To see that $u(s)$ is closed, consider
  $\vec a \in \mathsf{Stream}(A)$ with
  $a_{\leqslant n} \in \mathsf{Tr}_{\leqslant n}$ for each $n$. Then 
  for each $n$, there exists
  $\vec b = (a_0, \dots, a_n, b_{n+1}, \dots) \in \mathsf{Stream}(A)$
  such that $s$ generates $\vec b$; in particular,
  $\sigma^\ast(s, \vec b)_{\leqslant n}$ contains no $\bot$ for each
  $n$, whence also
  $\sigma^\ast(s, \vec a)_{\leqslant n} = \sigma^\ast(s, \vec
  b)_{\leqslant n}$ contains no $\bot$ for each
  $n$. Thus $s$ generates $\vec a$ so that $\vec a \in u(s)$ as required.
  
  To see that $u$ preserves $\vee$, consider $s, t \in S$ and
  $\vec a \in \mathsf{Stream}(A)$. Since $\sigma$ preserves $\vee$'s,
  we see that $\sigma^\ast(s \vee t, \vec a)$ is the pointwise join of
  $\sigma^\ast(s, \vec a)$ and $\sigma^\ast(t, \vec a)$; in
  particular, this means that $s \vee t$ generates $\vec a$ if and
  only if either $s$ generates $\vec a$ or $t$ generates $\vec a$:
  which is to say that $u(s \vee t) = u(s) \cup u(t)$, as desired.

  Next, to show that $u$ is a comagma homomorphism, we
  must verify that, for any $s_0 \in S$ and $a_0 \in A$, we have
  \begin{equation*}
    \{(a_1, a_2, \dots) : (a_0, a_1, \dots) \in u(s_0)\} =
    \begin{cases}
      u\bigl(\sigma(s_0)(a_0)\bigr) & \text{ if
        $\sigma(s_0)(a_0) \neq \bot$;}\\
      \emptyset & \text{ if
        $\sigma(s_0)(a_0) = \bot$.}
    \end{cases}
  \end{equation*}
  We can rephrase this as saying
  \begin{equation*}
    (a_0, a_1, \dots) \in u(s_0) \quad \iff \quad \sigma(s_0)(a_0) = s_1 \neq \bot \text{ and } (a_1, a_2, \dots) \in  u(s_1)\rlap{ ,}
  \end{equation*}
  which is clear on observing that if
  $\sigma(s_0)(a_0) = s_1 \neq \bot$, then
  $\sigma^\ast(s_0, (a_0, a_1, \dots)) = \sigma^\ast(s_1, (a_1, a_2,
  \dots))$. This proves that $u$ is a comagma homomorphism
  $S \rightarrow \mathsf{Tr}$, and it remains to prove it is
  \emph{unique}. So let $f$ be another such homomorphism; then for
  each $s_0 \in S$ we have
  \begin{equation}\label{eq:61}
    (a_0, a_1, \dots) \in f(s_0) \ \iff \ \sigma(s_0)(a_0) = s_1 \neq \bot \text{ and } (a_1, a_2, \dots) \in  f(s_1)\rlap{ ,}
  \end{equation}
  and we must show $f(s_0) = u(s_0)$.
  
  In one direction, if $\vec a \in f(s_0)$, then
  $\sigma^\ast(s_0, \vec a)$ is clearly never $\bot$, and so
  $\vec a \in u(s_0)$. Conversely, suppose that $\vec a \in u(s_0)$.
  If the sequence $\sigma^\ast(s_0, \vec a)$ is $(s_0, s_1, \dots)$, then for
  each $n$, the set $f(s_{n+1})$ is non-empty; so letting
  $(b_{n+1}, b_{n+2}, \dots)$ be any element of it, we may apply the
  leftward implication in~\eqref{eq:61} $n+1$ times to see that
  $(a_0, a_1, \dots, a_n, b_{n+1}, \dots) \in f(s_0)$. But this means
  that $a_{\leqslant n} \in f(s_0)_{\leqslant n}$ for each $n$; since
  $f(s_0)$ is closed, we conclude that $\vec a \in f(s_0)$ as
  required.
\end{proof}

Returning to the general situation, we now explain how every
generative $\mnd{T}$-system has a trace map valued in the object of
traces. As we noted above, any such system
$\sigma \colon S \rightarrow T(\Sigma_{a \in A} S)$ is equally an
$A$-ary comagma in the Kleisli category $\cat{Kl}(\mnd T)$ which may
in turn be seen as an $A$-ary comagma on the free
algebra $F^{\mnd T}(S)$ in $\C^{\mnd T}$. Now using the finality of
the object of traces yields the desired trace map, as the following
definition makes precise:

\begin{Defn}
  \label{def:33}
  Let $\sigma \colon S \rightarrow T(\Sigma_{a \in A} S)$ be a
  generative $\mnd T$-system over $A$. The \emph{associated free
    $A$-ary comagma} in $\C^{\mnd T}$ is $(F^\mnd{T}S, \sigma^\sharp)$, where
  \begin{equation*}
    \sigma^{\sharp} \defeq F^{\mnd T}(S) \xrightarrow{T\sigma} F^{\mnd T}(T(\Sigma_{a \in A} S)) \xrightarrow{\mu_{\Sigma_{a} S}} F^{\mnd T}(\Sigma_a S) \cong \Sigma_{a \in A} F^{\mnd T}(S)\rlap{ .}
  \end{equation*}
  This assignment is the action on objects of a functor
  \begin{equation}\label{eq:57}
    (\thg)^\sharp \colon \cat{Gen}_A(\mnd T) \rightarrow \cat{Comag}_A(\C^\mathsf T)
  \end{equation}
  which on maps sends $f \colon (S, \sigma) \rightarrow (U, \upsilon)$ to  $F^\mathsf{T}f \colon (F^\mnd{T}
  S, \sigma^\sharp) \rightarrow (F^\mnd{T} U, \upsilon^\sharp)$.
  
  If $\mnd T$ has an object of traces $(\mathsf{Tr},\tau)$, then
  the \emph{trace map} of a generative $\mnd T$-system $(S, \sigma)$
  is the map $\mathsf{tr} \colon S \rightarrow \mathsf{Tr}$
  obtained as the restriction along
  $\eta_S \colon S \rightarrow U^\mathsf{T}F^\mathsf{T}S$ of the
  unique comagma map
  $(F^{\mnd T}S, \sigma^\sharp) \rightarrow (\mathsf{Tr}, \tau)$.
\end{Defn}

\begin{Ex}
  \label{ex:6}
  When $\mathsf{T} = \mathsf{id}_\cat{Set}$, the object of traces is just
  the object of behaviours, and the trace map is just
    the behaviour map.
\end{Ex}

\begin{Ex}
  \label{ex:10}
  When $\mnd T = \mnd{P}_f^+$, a generative system is a labelled
  transition system $\sigma \colon S \rightarrow \P_f^+(A \times S)$.
  The associated $A$-comagma in the category of join-semilattices is
  given by
  \begin{align*}
    \sigma^\sharp \colon {\P}_f^+(S) & \rightarrow ({\P}_f^+(S) \cup \{\emptyset\})^A \\
    V & \mapsto a \mapsto \{s_1 \in S : (a,s_1) \in \sigma(s_0) \text{ for some } s_0 \in V\}\rlap{ .}
  \end{align*}
  It follows from this description that $\{s_0\}$ generates $\vec a$
  if and only if there exist $s_1, s_2, \dots \in S$ such
  that $(a_i, s_{i+1}) \in \sigma(s_i)$ for each $i$; which is to say
  that the non-empty subset
  $\mathsf{tr}(s_0) \subseteq \mathsf{Stream}(A)$ associated to $s_0$
  is the trace of~\eqref{eq:55}.
\end{Ex}

\subsection{Normalisation by trace evaluation for behaviours}
\label{sec:norm-trace-eval}

As is well known, bisimilarity is a finer equivalence on states of
labelled transition systems than trace equivalence; for example, we
may consider the labelled transition systems
\begin{align*}
  \sigma \colon \{s,t,u\} &\rightarrow \P_{f}^+(\{0,1\} \times \{s, t,u\})&
  \sigma' \colon \{s',t',u'\} &\rightarrow \P_f^+(\{0,1\} \times \{s',t',u'\})\\
  s & \mapsto \{(0,t), (1,t), (1,u)\}  & s' & \mapsto \{(0,t'),(1,u')\}\\
  t & \mapsto \{(0,t)\} & t' & \mapsto \{(0,t')\}\\
  u & \mapsto \{(1,t)\} & u' &\mapsto \{(0,t'), (1,t')\}\rlap{ .}
\end{align*}

The states $s$ and $s'$ both have trace
$\{00\dots, 100\dots,1100\dots\}$; however, their respective
behaviours are the trees
\begin{equation*}
  \cd[@-1em]{
  \vdots & \vdots & \vdots\\
  \bullet \ar@{-}[u]^-{0} &  
  \bullet \ar@{-}[u]^-{0} &
  \bullet \ar@{-}[u]^-{0} \\
  \bullet \ar@{-}[u]^-{0} &  
  \bullet \ar@{-}[u]^-{0} &
  \bullet \ar@{-}[u]^-{1} \\
  & \bullet \ar@{-}[ul]^-{0} \ar@{-}[u]^-{1} \ar@{-}[ur]_-{1}
} \qquad \text{and} \qquad 
  \cd[@-1em]{
  \vdots & \vdots & \vdots\\
  \bullet \ar@{-}[u]^-{0} &  
  \bullet \ar@{-}[u]^-{0} &
  \bullet \ar@{-}[u]^-{0} \\
  \bullet \ar@{-}[u]^-{0} &  
  \bullet \ar@{-}[u]^-{0} \ar@{-}[ur]_-{1} \\
  & \bullet \ar@{-}[ul]^-{0} \ar@{-}[u]^-{1} 
  }
\end{equation*}
which are \emph{not} bisimilar. The difference between $s$ and $s'$
which this captures is that an execution from state $s$ must
\emph{immediately} decide whether it will output zero, one or two 1's,
while an execution from state $s'$ need only decide the first bit now,
and can defer the decision on bit $2$ until the next computation step.

In fact, $s'$ above is ``minimal'' in the sense that, at each step,
it only requires a decision on the very next output bit.
We might hope that this minimality can be captured by somehow
``normalising'' the behaviour of $s$ to that of $s'$. What we will now
see is that this is in fact possible: we have an embedding-retraction
pair
\begin{equation}\label{eq:56}
  \cd{
    \mathsf{Beh} \ar@{->>}@/^1em/[r]^-{\mathsf{reflect}} &
    \mathsf{Tr} \ar@{ >->}@/^1em/[l]^-{\mathsf{reify}}
  }
\end{equation}
wherein $\mathsf{reflect}$ takes the trace of a behaviour, while its
section $\mathsf{reify}$ takes the minimal realisation of each trace
as a behaviour. In particular, the ``normalisation function''
$\mathsf{reify} \circ \mathsf{reflect}$ sends does the right thing for
our example above: $\mathsf{beh}(s')$ is mapped to $\mathsf{beh}(s)$.

We now explain how to construct the maps in~\eqref{eq:56} for
generative $\mnd T$-systems, where $\mnd T$ is \emph{any} monad which
admits hypernormalisation, i.e., any linear exponential monad. 
We first define the maps in each direction;
note that the easier direction,
$\mathsf{reflect}$, does not rely on hypernormalisation.

\begin{Defn}
  \label{def:35}
  If the object of behaviours $\mathsf{Beh}$ and the object of traces
  $\mathsf{Tr}$ exist for generative $\mnd{T}$-systems over $A$, then
  the \emph{reflection} map
  $\mathsf{reflect} \colon \mathsf{Beh} \rightarrow \mathsf{Tr}$ is
  the trace map of $\mathsf{Beh}$ \emph{qua} generative
  $\mnd T$-system.
\end{Defn}

\begin{Defn}
  \label{def:36}
  Let $\mnd T$ be a linear exponential monad on a symmetric monoidal
  category $(\C, \otimes, I)$ with finite coproducts. If
  $\sigma \colon S \rightarrow S \otimes \cdots \otimes S$ is an $A$-ary
  comagma in $\C^\mathsf{T}$, then the \emph{associated generative
    $\mnd T$-system} is $(S, \sigma^\flat)$, where 
  \begin{equation*}
    \sigma^\flat = S \xrightarrow{\sigma} S \otimes \cdots \otimes S \xrightarrow{\eta_S \otimes \cdots \otimes \eta_S} T(S) \otimes \cdots \otimes T(S) \xrightarrow{\varphi^{-1}} T(S + \cdots + S)\rlap{ .}
  \end{equation*}
  This assignment is the action on objects of a functor
  \begin{equation}\label{eq:58}
    (\thg)^\flat\colon \cat{Comag}_A(\mathsf{T}) \rightarrow \cat{Gen}_A(\mathsf{T})
  \end{equation}
  which on maps sends $f \colon (S, \sigma) \rightarrow (U, \upsilon)$
  to $f \colon (S, \sigma^\flat) \rightarrow (U, \upsilon^\flat)$.
  
  If the object of behaviours $\mathsf{Beh}$ and the object of traces
  $\mathsf{Tr}$ exist for generative $\mnd{T}$-systems over $A$, then
  the \emph{reification} map
  $\mathsf{reify} \colon \mathsf{Tr} \rightarrow \mathsf{Beh}$ is the
  behaviour map of the associated generative $\mnd T$-system of
  $\mathsf{Tr}$.
\end{Defn}
\begin{Ex}
  \label{ex:12}
  Recall that when $\mnd T = \mnd{P}_f^+$, the set $\mathsf{Beh}$ is
  the set of purely infinite, finitely branching $A$-labelled trees
  modulo bisimilarity, while $\mathsf{Tr}$ is the set of closed
  subsets of $\mathsf{Stream}(A)$. In this case, the reflection map
  sends a tree $\tau$ to the set of $A$-streams that label the
  infinite paths from the root of $\tau$. On the other hand, the
  reflection map sends a closed subset $V \subseteq
  \mathsf{Stream}(A)$ to the tree
  \begin{equation*}
    \tau = \cd[@-1em]{ \tau_1 \ar@{-}[dr]_-{a_1}& \dots &
        \tau_n\rlap{, } \ar@{-}[dl]^-{a_n} \\ & \bullet } 
  \end{equation*}
  where $a_1, \dots, a_n$ are the distinct elements of the set
  $V_{\leqslant 0}$, and where $\tau_i$ is (coinductively) the tree
  associated to the closed subset
  $\{(b_0, b_1, \dots) : (a_i, b_0, b_1, \dots) \in V\}$.
\end{Ex}
In the preceding example, it is easy to see that reification is a
section of reflection; and in fact, this is true in general:

\begin{Prop}
  \label{prop:21}
  In the situation of Definition~\ref{def:36}, we have that
  $\mathsf{reflect} \circ \mathsf{reify} = \mathsf{id}$.
\end{Prop}
\begin{proof}
  Let the coalgebra structures of $\mathsf{Tr}$ and $\mathsf{Beh}$ be
  $\tau \colon \mathsf{Tr} \rightarrow \mathsf{Tr} \otimes \cdots
  \otimes \mathsf{Tr}$ and $\beta \colon \mathsf{Beh} \rightarrow
  T(\mathsf{Beh} + \cdots + \mathsf{Beh})$ respectively. By
  definition, $\mathsf{reflect}$ is the 
  unique homomorphism
  \begin{equation*}
    u \colon (\mathsf{Tr}, \tau^{\flat}) \rightarrow (\mathsf{Beh}, \beta) \qquad \text{in } \cat{Gen}_A(\mnd T)
  \end{equation*}
  while $\mathsf{reflect}$ is the unique homomorphism
  \begin{equation*}
    v \colon (F^\mathsf{T}(\mathsf{Beh}), \beta^{\sharp}) \rightarrow (\mathsf{Tr}, \tau) \qquad \text{in } \cat{Comag}_A(\C^\mnd T)
  \end{equation*}
  precomposed by $\eta_{\mathsf{Beh}}$. So $\mathsf{reflect} \circ
  \mathsf{reify} = v \circ \eta_{\mathsf{Beh}} \circ u = v \circ Tu
  \circ \eta_{\mathsf{Tr}}$, which is equally the precomposition by
  $\eta_{\mathsf{Tr}}$ of the composite homomorphism
  \begin{equation}\label{eq:59}
    (F^\mathsf{T}(\mathsf{Tr}), (\tau^\flat)^\sharp) \xrightarrow{u^\sharp = Tu}  (F^\mathsf{T}(\mathsf{Beh}), \beta^\sharp) \xrightarrow{v} (\mathsf{Tr}, \tau)\rlap{ .}
  \end{equation}

  We claim that~\eqref{eq:59} is in fact equal to the algebra
  structure map $\alpha \colon T(\mathsf{Tr}) \rightarrow \mathsf{Tr}$
  of the Eilenberg--Moore $\mnd T$-algebra $\mathsf{Tr}$; since this
  structure map satisfies
  $\alpha \circ \eta_{\mathsf{Tr}} = \mathsf{id}$, this will complete
  the proof. To prove the claim, note that, since~\eqref{eq:59} is the
  unique map into a terminal object, it suffices to show that $\alpha$
  is itself a map
  $(F^\mathsf{T}(\mathsf{Tr}), (\tau^{\flat})^\sharp) \rightarrow
  (\mathsf{Tr}, \tau)$. Now, $(\tau^\flat)^\sharp$ is, by
  Definitions~\ref{def:33} and~\ref{def:36}, the following composite
  \begin{equation*}
    \cd{
      T(\mathsf{Tr}) \ar[d]_-{T\tau} \\
      T(\otimes_{a} \mathsf{Tr}) \ar[r]^-{T(\otimes_a \eta)} &
      T(\otimes_{a} T(\mathsf{Tr})) \ar[r]^-{T\varphi^{-1}} &
      T^2(\Sigma_{a} \mathsf{Tr}) \ar[r]^-{\mu} & 
      T(\Sigma_{a} \mathsf{Tr}) \ar[r]^-{\varphi} &
      \otimes_a T(\mathsf{Tr})\rlap{ ;}
    }
  \end{equation*}
  however, because $\mnd T$ is a linear exponential monad, the
  composite along the bottom row is,
  by~\cite[Theorem~3.1.6]{Blute2015Cartesian}, exactly the opmonoidal
  structure map
  $\upsilon \colon T(\otimes_{a} \mathsf{Tr}) \rightarrow \otimes_{a}
  T(\mathsf{Tr})$, and so to say that
  $\alpha \colon (F^\mathsf{T}(\mathsf{Tr}), (\tau^{\flat})^\sharp)
  \rightarrow (\mathsf{Tr}, \tau)$ is to say that the following
  diagram commutes:
  \begin{equation*}
    \cd{
      T(\mathsf{Tr}) \ar[rr]^-{\alpha} \ar[d]_-{T\tau} & & 
      \mathsf{Tr} \ar[d]^-{\tau} \\
      T(\otimes_a \mathsf{Tr}) \ar[r]^-{\upsilon} & \otimes_a T(\mathsf{Tr}) \ar[r]^-{\otimes_a \alpha} &\otimes_a \mathsf{Tr}\rlap{ .}
    }
  \end{equation*}
  As the bottom row is, by~\eqref{eq:60}, the $\mnd{T}$-algebra
  structure of $\otimes_a (\mathsf{Tr}, \tau)$, commutativity of
  this diagram is precisely the fact that $\tau \colon (\mathsf{Tr},
  \tau) \rightarrow \otimes_a (\mathsf{Tr}, \tau)$ in $\C^\mnd T$.
\end{proof}

As indicated above, the composite
$\mathsf{reify} \circ \mathsf{reflect}$ is an idempotent endofunction
of the set of behaviours, which implements ``normalisation by trace
evaluation''. In the case of $\mathsf{P}_f^+$, it normalises each
purely infinitely, finitely branching $A$-labelled tree to a
trace-equivalent one in which the children of any given node have
\emph{distinct} $A$-labels: thus, behaviours which, at
each step, only decide the next output token.

\subsection{Stream processors}
\label{sec:stream-processors}
In the final sections of this paper, we investigate the
reification--reflection pair for other computationally meaningful
linear exponential monads $\mnd T$. We begin with the free $B$-ary
magma monad $\mnd T_B$ from Section~\ref{sec:norm-cont-funct}.

\emph{Generative $\mnd T_B$-systems}. A generative $\mnd T_B$-system
$\sigma \colon S \rightarrow T_B(A \times S)$ is a \emph{stream
  processor}~\cite{Hancock2009Representations}. The
function $\sigma$ assigns to each state $s$ a decision tree which, by
consuming an initial segment of a stream of $B$-values, decides an
output $A$-value and a next state in $S$. When continued indefinitely,
this process turns a stream of $B$-values into a stream
of $A$-values; whence the name.

\emph{Behaviours}. In this case, the object of behaviours
$\mathsf{Beh}$ is the final coalgebra $\nu X.\, T_B(A \times X)$;
since it is a final coalgebra of a polynomial endofunctor, it has an
entirely standard description, which can be stated as follows: it is
the set of (possibly infinite) $B$-ary branching trees wherein each
non-leaf vertex is labelled by a finite list of $A$-elements; each
leaf is labelled by a stream of $A$-elements; and there is no infinite
simple path of nodes labelled by the empty string.

We introduce the following notation: given a tree $t \in \mathsf{Beh}$
and some $a \in A$, we write $a \cdot t$ for the tree $t$ with the
element $a$ adjoined to the front of the string labelling the root.
Using this, we may describe the coalgebra structure of $\mathsf{Beh}$
as follows. For a tree $t \in \mathsf{Beh}$, a \emph{cut vertex} is
one labelled by a non-empty string, but whose proper ancestors are all
unlabelled. For a cut vertex $v$, we define $\mathsf{hd}(v) \in A$ and
$\mathsf{tl}(v) \in \mathsf{Beh}$ to be such that the subtree of
descendents of $v$ is $\mathsf{hd}(v) \cdot \mathsf{tl}(v)$. Now
$\beta \colon \mathsf{Beh} \rightarrow T_B(A \times \mathsf{Beh})$
takes $t$ to the well-founded $B$-ary tree obtained by replacing each
cut vertex $v \in t$ with the leaf labelled by
$(\mathsf{hd}(v), \mathsf{tl}(v))$.

\emph{Behaviour map}. Given a generative $\mathsf{T}_B$-system
$\sigma \colon S \rightarrow T_B(A \times S)$, the behaviour
$\mathsf{beh}(s)$ of $s \in S$ is the tree obtained by coinductively
replacing each $(a,t)$-labelled leaf of $\sigma(s)$ by the tree
$a \cdot \mathsf{beh}(t)$.

\emph{Traces}. An $A$-ary comagma in $\cat{Set}^{\mnd{T_B}}$ is a
$B$-ary magma $(X, \xi)$ endowed with a $B$-ary magma homomorphism
\begin{equation*}
  (X, \xi) \rightarrow \overbrace{(X, \xi) \odot \cdots \odot (X, \xi)}^{\text{$A$ copies}}
\end{equation*}
where $\odot$ is as in Lemma~\ref{lem:9}. In this
case,~\cite[Theorem~37]{Garner2021Stream} proves that the object
of traces $\mathsf{Tr}$ is the set of continuous functions
$\mathsf{Stream}(B) \rightarrow \mathsf{Stream}(A)$, endowed with
the $B$-ary magma structure $\zeta$ given by
\begin{equation*}
  \Bigl(\zeta(f_b : b \in B)\Bigr)(c_0, c_1, \dots) = f_{c_0}(c_1, c_2, \dots)
\end{equation*}
and the $A$-ary comagma structure $\tau$ found as follows. Given a continuous
function $f \colon \mathsf{Stream}(B) \rightarrow \mathsf{Stream}(A)$,
we call a string $b_0 \cdots b_k \in B^\ast$ \emph{determining} if the
composite $f(\thg)_0 \colon \mathsf{Stream}(B) \rightarrow A$ is
constant on the clopen set
\begin{equation}\label{eq:63}
  [b_0 \cdots b_k] = \{ \vec b \in \mathsf{Stream}(B) : b_{\leqslant k} = (b_0, \dots, b_k)\}\rlap{ .}
\end{equation}
We call $(b_0, \dots, b_k)$ \emph{minimal} if it is determining, but
$(b_0, \dots, b_{k-1})$ is not so. Now by continuity of $f$, we can
find minimal strings $m_1, \dots, m_k \in B^\ast$, elements
$a_1, \dots, a_k \in A$ and continuous functions
$f_1, \dots, f_k \in \mathsf{Tr}$ such that the following clauses
totally define $f$:
\begin{equation*}
  f(m_1 \cdot \vec b) = a_1 \cdot f_1(\vec b) \qquad \cdots \qquad 
  f(m_k \cdot \vec b) = a_k \cdot f_k(\vec b)
\end{equation*}
where $\cdot$ denotes concatenation. The finite strings
$m_1, \dots, m_k$ can be seen as vertex addresses in the purely
infinite $B$-ary tree: and now replacing the subtree rooted at $m_i$
by a leaf vertex labelled by $(a_i, f_i) \in A \times \mathsf{Tr}$
yields an element of $T_B(A \times \mathsf{Tr})$; and the
\emph{minimality} of each $m_i$ ensures that this tree in fact lies in
$\mathsf{Tr} \odot \cdots \odot \mathsf{Tr} \subseteq T_B(A \times
\mathsf{Tr})$, so that we may define it to be $\tau(f)$.

\emph{Trace map}. For a given stream processor
$\sigma \colon S \rightarrow T_B(A \times S)$, the trace of a state
$s \in S$ is determined as in~\cite[\sec
3]{Hancock2009Representations}. First, for each set $V$, we write
$\theta_V$ for the function
$\mathsf{Stream}(B) \times T_B(V) \rightarrow \mathsf{Stream}(B)
\times V$ determined by
\begin{align*}
  \theta_V(\vec b, \bullet_{\textstyle v}) &= (\vec b, v)\\
  \theta_V\biggl((b_0, b_1, \dots),     \cd[@-1em]{
    \tau_b \ar@{-}[dr]^-{}&  \dots & \tau_{b'} \ar@{-}[dl]_-{} \\ & \bullet
  }
  \biggr) &= \theta\bigl((b_1, b_2,\dots), \tau_{b_0}\bigr)\rlap{ .}
\end{align*}
Now the trace $\mathsf{tr}(s)$ of $s \in S$ is the homomorphism
$\mathsf{Stream}(B) \rightarrow \mathsf{Stream}(A)$ determined by
\begin{equation*}
  \smash{\mathsf{tr}(s)(\vec b)_0 = a \quad \text{and} \quad \mathsf{tr}(s)(\vec b)_{i+1} = \mathsf{tr}(s')(\vec c)_i \text{ where $\theta(\vec b, \sigma(s)) = (\vec c, a, s')$.}}
\end{equation*}
The fact that this is a correct description of the trace map
is~\cite[Proposition~42]{Garner2021Stream}.

\emph{Reflection and reification}. The reflection map $\mathsf{Beh}
\rightarrow \mathsf{Tr}$ is simply the instantiation of the above
trace map at $(\mathsf{Beh}, \beta)$. On the other hand, the
generative $\mathsf{T}_B$-system structure on $\mathsf{Tr}$ is
the composite
\begin{equation*}
  \mathsf{Tr} \xrightarrow{\tau} \mathsf{Tr} \odot \cdots \odot \mathsf{Tr} \xrightarrow{\subseteq} T_B(A \times \mathsf{Tr})
\end{equation*}
and the reflection map $\mathsf{Tr} \rightarrow \mathsf{Beh}$ is the
associated behaviour map of this $\mathsf{T}_B$-system. We illustrate
this with an example drawn from~\cite[Example~22]{Garner2021Stream}. Consider the following
two stream processors:
\begin{align*}
  \sigma \colon \{s\} & \rightarrow T_B(A \times \{s\}) &
  \tau \colon \{t\} & \rightarrow T_B(A \times \{t\})\\
  s & \mapsto (a, \bullet_{\textstyle s}) & t & \mapsto     \cd[@-1em]{
      (a, t) \ar@{-}[dr]^-{}&  \dots & (a, t)\rlap{. } \ar@{-}[dl]_-{} \\ & \bullet
    }
\end{align*}

The traces $\mathsf{tr}(s)$ and $\mathsf{tr}(t)$ are both the
continuous function
$\mathsf{Stream}(B) \rightarrow \mathsf{Stream}(A)$ sending every
stream to $(a,a,a,\dots)$. On the other hand, as shown in
\emph{loc.~cit.}, their behaviours are \emph{distinct}: indeed,
$\mathsf{beh}(s)$ is the one-vertex tree whose root is labelled by
$(a, a, a, \dots)$, while $\mathsf{beh}(t)$ is the purely infinite
$B$-ary tree whose every vertex is labelled with a single $a$. The
difference which this captures is that $s$ produces its stream of
$a$'s while completely ignoring its input, and on the other, $t$
inefficiently consumes a single input token before outputting each
$a$. In fact, $s$ is the \emph{most} efficient realisation of this
function, and the normalisation-by-trace-evaluation map
$\mathsf{reify} \circ \mathsf{reflect}$ sends $\mathsf{beh}(t)$ to
$\mathsf{beh}(s)$. In general, the image of this normalisation
function comprises those trees $t$ in which no vertex has child
subtrees of the form $a \cdot t_1, \dots, a \cdot t_k$ for the
\emph{same} $a$.

\subsection{Probabilistic and logical generative systems}
\label{sec:prob-logic-gener}

In this section, we now consider the case where $\mnd T$ is either the
discrete distribution monad $\mnd D$ on $\cat{Set}$, or else the
logical subdistribution monad $\sDb$ on $\cat{Set}_B$.

\emph{Generative $\mnd D$- and $\sDb$-systems}. When
$\mnd T = \mnd D$, a generative $\mnd D$-system is a (non-terminating)
\emph{generative probabilistic system} in the sense
of~\cite{Glabbeek1990Reactive}; its transition map
$\sigma \colon S \rightarrow \D(A \times S)$ associates to each state
a finite probability distribution over output tokens and next states.
On the other hand, when $\mnd T = \sDb$, a generative $\sDb$-system
$\sigma \colon S \rightarrow \sdb(\Sigma_{a \in A} S)$ assigns to each
state $s \in S$ of domain $b$ a finitely supported $B$-valued logical
subdistribution $\omega$ of domain $b$ on the $B$-set $A \times S$
(with domains given by $\abs{(a,s)}_{A \times S} = \abs{s}_S$).

\emph{Behaviours}. The objects of behaviours can calculated, as in
Example~\ref{ex:5}, by following the approach
of~\cite{Aczel1989Final}: we first exhibit a \emph{weakly final}
coalgebra, and then quotient it by bisimulation. For generative
probabilistic systems, we consider the set $\mathsf{Beh}'$ of finitely
branching, purely infinite trees, whose edges are labelled by elements
of $(0,1] \times A$, and where the $(0,1]$-weights on the children of
any vertex sum to $1$. There is an evident structure map
$\beta' \colon \mathsf{Beh}' \rightarrow \D(A \times \mathsf{Beh}')$
and the object of behaviours is its quotient by the largest
bisimulation in the sense of~\cite{Larsen1991Bisimulation}, i.e., the
largest equivalence relation $R$ satisfying
\begin{equation}\label{eq:62}
  x \mathrel{R} y \implies \beta'(x)(\{a\} \times C) = \beta'(y)(\{a\} \times C) \text{ for all } a \in A, \ C \in \mathsf{Beh}' / R\rlap{ .}
\end{equation}

For generative $\sDb$-systems, we take $\mathsf{Beh}'$ to be the
$B$-labelled set of finitely branching, purely infinite trees whose
vertices are labelled by elements of $B \setminus \bot$, whose edges
are labelled in $(B \setminus \bot) \times A$, and where the
$B$-weights on the children of a vertex labelled by $b$ are a finite
partition of $b$; the domain of such a tree is the label of its root.
For the analogous notion of bisimulation to~\eqref{eq:62}, we again
obtain $\mathsf{Beh}$ as the quotient of $\mathsf{Beh}'$ by the
largest bisimulation.

\emph{Behaviour map}. Given a generative probabilistic system
$\sigma \colon S \rightarrow \D(A \times S)$ and $s \in S$, the
behaviour $\mathsf{beh}(s)$ is given coinductively by
\begin{equation*}
  \mathsf{beh}(s) = \!\!\cd[@-1em]{ \mathsf{beh}(t_1) \ar@{-}[dr]_-{(p_1, a_1)}& \dots &
    \mathsf{beh}(t_1) \ar@{-}[dl]^-{(p_n, a_n)} \\ & \bullet } \  \text{where } \sigma(s) = p_1 \cdot (a_1, t_1) + \cdots + p_n \cdot (a_n, t_n)\text{ .}
\end{equation*}
The behaviour map for a generative $\sDb$-system
is entirely analogous.

\emph{Traces}. An $A$-ary comagma in $\cat{Set}^{\mnd{D}}$ is an
abstract convex space $X$ endowed with a convex map
$X \rightarrow X \star \cdots \star X$. In this case, we have the
following characterisation of the \emph{final} such comagma. In the
statement and proof, we make use of the basic clopen sets of
$\mathsf{Stream}(A)$ defined in~\eqref{eq:63}.

\begin{Prop}
  \label{prop:20}
  The final $A$-ary comagma in $\cat{Set}^{\mnd D}$ is the set
  $\R(\mathsf{Stream}(A))$ of (necessarily) Radon probability distributions on
  $\mathsf{Stream}(A)$, endowed with the usual convex space structure
  and with the comagma structure
  \begin{equation*}
    \R(\mathsf{Stream}(A)) \xrightarrow{(\mathsf{hd}, \mathsf{tl})_!} \R(A \times \mathsf{Stream}(A)) \xrightarrow{\cong} \R(\mathsf{Stream}(A)) \star \cdots \star \R(\mathsf{Stream}(A))\rlap{ ;}
  \end{equation*}
  more explictly, this is the comagma structure given by
  \begin{equation*}
    \omega \mapsto \sum_{\substack{a \in A \\ \omega([a]) \neq 0}} \omega([a]) \cdot \iota_a(\overline{\omega_a})
  \end{equation*}
  where $\overline{\omega_a}$ is the distribution with
  $\overline{\omega_a}([a_1 a_2 \cdots]) = \omega([aa_1a_2
  \cdots])/\omega([a])$.
\end{Prop}
The reader should compare this result
to~\cite[Theorem~3.33]{Kerstan2013Coalgebraic} which, among other
things, proves that $\mathsf{Stream}(A)$ is terminal in the category
$\cat{Comag}_A(\cat{Kl}(\mnd P))$, where $\mnd P$ is the probability
monad on the category of measurable spaces. Neither result implies the
other, but one might expect, as a common generalisation, that
$\P(\mathsf{Stream}(A))$ is a final $A$-ary comagma in
$\cat{Meas}^{\mnd P}$.
\begin{proof}
  Let $(S, \sigma \colon S \rightarrow S \star \cdots \star S)$ be an
  $A$-ary comagma in the category of convex spaces. We define a map
  $u \colon S \rightarrow \R(\mathsf{Stream}(A))$ by
  \begin{equation*}
    u(s)([\, ]) = 1 \quad \text{and} \quad
    u(s)([a_0 \cdots a_n]) =
    \begin{cases}
      p_{a_0} \cdot u(s_{a_0})([a_1 \cdots a_n]) & \text{if $a_0
        \in A'$;}\\
      0 & \text{otherwise,}
    \end{cases}
  \end{equation*}
  where here we assume that
  $\sigma(s) = \sum_{a \in A'} p_a \cdot \iota_a(s_a)$ for some
  $A' \subseteq A$. It is easy to see that $u(s)$ extends to a
  finitely additive map on the Boolean algebra of all clopen sets of
  $\mathsf{Stream}(A)$, and so
  by~\cite[Lemma~IV.9.11]{Dunford1958Linear} extends uniquely to a
  probability measure on $\mathsf{Stream}(A)$. So $u$ is well-defined;
  it is now straightforward to verify, following the pattern of
  Proposition~\ref{prop:23}, that it is a convex map and a map of
  $A$-ary comagmas.

  Suppose now that
  $h \colon S \rightarrow \R(\mathsf{Stream}(A))$ is a convex map
  and a map of $A$-ary comagmas; we must show that $u = h$.
  For this it suffices to show that
  \begin{equation}\label{eq:64}
    u(s)([a_0 \cdots a_n]) = h(s)([a_0 \cdots a_n])
  \end{equation}
  for all states $s$ and finite strings $a_0 \cdots a_n \in A^\ast$.
  We do so by induction on $n$. For the base case $n = 0$, suppose
  that $\sigma(s) = \sum_{a \in A'} p_a \cdot \iota_a(s_a)$ as above;
  now as $h$ is a comagma homomorphism, we must have that
  \begin{equation}\label{eq:65}
    \sum_{\substack{a \in A \\ h(s)([a]) \neq 0}} h(s)([a]) \cdot \iota_a(\overline{h(s)_a}) = \sum_{a \in A'}p_a \cdot \iota_a(h(s_a))\rlap{ ;}
  \end{equation}
  so in particular, $u(s)([a]) \neq 0$ iff $a \in A'$ iff
  $h(s)([a]) \neq 0$, and when these equivalent conditions hold,
  we have $u(s)([a]) = p_a = h(s)([a])$, as required.

  For the inductive step, we suppose we have verified the equality for
  $n-1$, and will now verify it for $n$ as in~\eqref{eq:64}. If $a_0
  \notin A'$ then $u(s)([a_0]) = 0 = h(s)([a_0])$ and both
  sides of~\eqref{eq:64} are zero. If $a_0 \in A'$,
  then from~\eqref{eq:65} we must have
  $\overline{h(s)_{a_0}} = h(s_{a_0})$. But since 
  $\overline{h(s)_{a_0}}([a_1 \cdots a_n]) = h(s)([a_0 \cdots a_n]) / h(s)([a_0])$
  we conclude using the inductive hypothesis that
  \begin{align*}
    h(s)([a_0 \cdots a_n]) &= 
    h(s)([a_0]) \cdot h(s_{a_0})([a_1 \cdots a_n]) \\ &= p_{a_0} \cdot u(s_{a_0})([a_1 \cdots a_n]) = u(s)([a_0 \cdots a_n])\rlap{ .} \qedhere
  \end{align*}
\end{proof}

In the logical case, an $A$-ary comagma in $(\cat{Set}_B)^{\sDb}$ is a
sheaf $X$ endowed with a sheaf map
$\mathbf{X} \rightarrow \mathbf{X} \star \cdots \star \mathbf{X}$. An
entirely similar proof to the above shows that:
\begin{Prop}
  \label{prop:22}
  The final $A$-ary comagma in $\cat{Set}_B^{\sDb}$ is the $B$-sheaf
  of continuous functions valued in $\cat{Stream}(A)$. Via Stone
  duality, this is equally well the $B$-sheaf $\mathsf{Tr}$ whose
  elements of domain $b \in B$ are homomorphisms from the Boolean
  algebra of clopen sets of $\mathsf{Stream}(A)$ to the Boolean
  algebra $B/b$ of elements below $b$ in $B$, with comagma structure
  $\mathsf{Tr} \rightarrow \mathsf{Tr} \star \cdots \star \mathsf{Tr}$
  given by
    \begin{equation*}
      \omega \mapsto \sum_{\substack{a \in A \\ \omega([a]) \neq \bot}} \omega([a]) \cdot \iota_a(\omega_a)
    \end{equation*}
    where $\omega_a$ is the homomorphism
    $\mathrm{Clopen}(\mathsf{Stream}(A)) \rightarrow B / \omega([a])$
    defined by
    $\omega_a([a_1 a_2 \cdots]) = \omega([aa_1a_2 \cdots])$.
\end{Prop}

\emph{Trace map}. The following result characterises traces for
probabilistic generative systems. The construction of the
probabilities in~\eqref{prop:25} goes back at least
to~\cite{Jou1990Equivalences}.
\begin{Prop}
  \label{prop:24}
  Let $\sigma \colon S \rightarrow \D(A \times S)$ be a probabilistic
  generative system. The probability distribution $\mathsf{tr}(s)$
  associated to each state $s$ is characterised by
  \begin{equation}\label{prop:25}
    \mathsf{tr}(s)([\,]) = 1 \qquad 
    \mathsf{tr}(s)([a_0 a_1 \cdots a_n]) = \sum_{t \in S} \sigma(s)(a_0, t) \cdot \mathsf{tr}(t)([a_1 \cdots a_n])\rlap{ .}
  \end{equation}
\end{Prop}
\begin{proof}
  Let
  $\overline{\mathsf{tr}} \colon \D(S) \rightarrow
  \R(\mathsf{Stream}(A))$ be the unique extension of $\mathsf{tr}$ to
  a convex map satisfying $\overline{\mathsf{tr}} \circ \eta_S =
  \mathsf{tr}$. It suffices to check that $\overline{\mathsf{tr}}$ is
  a map of $A$-ary comagmas
  $(F^{\mnd D}(S), \sigma^\sharp) \rightarrow (\mathsf{Beh}, \beta)$;
  for then $\overline{\mathsf{tr}}$ is the unique such homomorphism,
  and its precomposition with $\eta_S$, which is $\mathsf{tr}$, is the
  desired trace map. Now, to check that $\overline{tr}$ is a map of
  $A$-ary comagmas, it suffices to do so on Dirac distributions
  $\eta(s) \in \D(S)$: in other words, we need only check
  commutativity in:
  \begin{equation*}
    \cd{
      S \ar[d]_-{\mathsf{tr}} \ar[r]^-{\sigma} & \D(A \times S) \ar[r]^-{\cong} & \D(S) \star \cdots \star \D(S) \ar[d]^-{\overline{\mathsf{tr}} \star \dots \star \overline{\mathsf{tr}}} \\
      \R(\mathsf{Stream}(A)) \ar[r]^-{(\mathsf{hd}, \mathsf{tl})_!} & 
      \R(A \times \mathsf{Stream}(A)) \ar[r]^-{\cong} &
      \R(\mathsf{Stream}(A)) \star \cdots \star \R(\mathsf{Stream}(A))\rlap{ .}
    }
  \end{equation*}
  Around the top of this diagram we have
  \begin{equation*}
    \cd[@-1em]{
      s \ar@{|->}[r]^-{} & \displaystyle \sum_{\substack{a \in A \\ \sigma(s)(a, S) \neq 0}} \sigma(s)(a, S) \cdot \iota_a(\overline{\sigma(s)(a, \thg)}) \ar@{|->}[d]^-{}\\
      & \displaystyle \sum_{\substack{a \in A \\ \sigma(s)(a,S) \neq 0}} \sigma(s)(a, S) \cdot \iota_a\Bigl(\sum_{t \in S} \frac{\sigma(s)(a, t)}{\sigma(s)(a,S)} \cdot \mathsf{tr}(t)\Bigr)
    }
  \end{equation*}
  where we write $\sigma(s)(a,S)$ for $\sum_{t \in S} \sigma(s)(a,t)$;
  while around the bottom we have
  \begin{equation*}
    \cd[@-1em]{
      s \ar@{|->}[d]^-{} \\
      \mathsf{tr}(s) \ar@{|->}[r]^-{} & \displaystyle \sum_{\substack{a \in A\\ \mathsf{tr}(s)([a]) \neq 0}} \mathsf{tr}(s)([a]) \cdot \iota_a(\overline{\mathsf{tr}(s)_a})
    }
  \end{equation*}
  Thus, we need only verify that for all $a \in A$ and $s \in S$ we
  have $\sigma(s)(a, S) =
  \mathsf{tr}(s)([a])$ for all $a,s$---which is clear
  from~\eqref{prop:25}---and that, if $\sigma(s)(a,S) \neq 0$, we have
  \begin{equation*}
\overline{\mathsf{tr}(s)_a} =
  \sum_{t \in S} \frac{\sigma(s)(a, t)}{\sigma(s)(a,S)} \cdot
  \mathsf{tr}(t)\rlap{ .}
\end{equation*}
But evaluating at $[a_1 \cdots a_n]$, this is the
condition that
  \begin{equation*}
    \frac{\mathsf{tr}(s)[aa_1 \cdots a_n]}{\mathsf{tr}(s)([a])} = 
    \sum_{t \in S} \frac{\sigma(s)(a, t)}{\sigma(s)(a,S)} \cdot
    \mathsf{tr}(t)[a_1 \cdots a_n]
  \end{equation*}
  which since $\mathsf{tr}(s)([a]) = \sigma(s)(a, S)$, follows from~\eqref{prop:25}.
\end{proof}
By a corresponding argument, we can show:
\begin{Prop}
  \label{prop:26}
  For a generative $\sDb$-system
  $\sigma \colon S \rightarrow \sdb(\sum_{a \in A} S)$, the trace
  $\mathsf{tr}(s)$ of some $s \in S$ of domain $b$ is the homomorphism
  $\mathrm{Clopen}(\mathsf{Stream}(A)) \rightarrow B/b$ specified on
  the generating basic clopen sets by
  \begin{equation*}
    \mathsf{tr}(s)([\,]) = b \quad \ \ 
    \mathsf{tr}(s)([a_0 a_1 \cdots a_n]) = \bigvee_{t \in S} \sigma(s)(a_0, t) \wedge \mathsf{tr}(t)([a_1 \cdots a_n])\rlap{ .}
  \end{equation*}
\end{Prop}

\emph{Reflection and reification}. In both the examples of this
section, the reflection map $\mathsf{Beh} \rightarrow \mathsf{Tr}$ is
the trace map of the object of behaviours. As for the reification
maps, we concentrate on the probabilistic case; the logical case is
entirely analogous. The probabilistic generative system
structure $\mathsf{Tr} \rightarrow \D(A \times \mathsf{Tr})$ on the
object of traces is given by
  \begin{equation*}
    \omega \mapsto \sum_{\substack{a \in A \\ \omega([a]) \neq 0}} \omega([a]) \cdot (a,\overline{\omega_a})\rlap{ ,}
  \end{equation*}
  and as such, the reflection map
  $\mathsf{Tr} \rightarrow \mathsf{Beh}$ realises each trace as the
  behaviour which at each step makes the minimal specialisation of the
  sample space required to resolve the next output token. For example,
  we may consider the following refinement of the example from
  Section~\ref{sec:norm-trace-eval}:
\begin{align*}
  \sigma \colon \{s,t,u\} &\rightarrow \D(\{0,1\} \times \{s, t,u\})&
  \!\!\!\!\!\!\!\tau \colon \{s',t',u'\} &\rightarrow \D(\{0,1\} \times \{s',t',u'\})\\
  s & \mapsto \tfrac 1 3 \cdot (0,t) + \tfrac 1 3 \cdot (1,t) + \tfrac 1 3 \cdot (1,u)  & s' & \mapsto \tfrac 1 3 \cdot (0,t') + \tfrac 2 3 \cdot (1,u')\\
  t & \mapsto 1 \cdot (0,t) & t' & \mapsto 1 \cdot (0,t')\\
  u & \mapsto 1 \cdot (1,t) & u' &\mapsto \tfrac 1 2 \cdot (0,t') + \tfrac 1 2 \cdot (1,t')\rlap{ .}
\end{align*}
The traces of $s$ and $s'$ are both the probability distribution which
picks uniformly between the infinite binary strings $00\dots$,
$100\dots$ and $1100\dots$; however, their respective behaviours are
the trees
\begin{equation*}
  \cd[@-0.5em]{
  \vdots & \vdots & \vdots\\
  \bullet \ar@{-}[u]^-{(1,0)} &  
  \bullet \ar@{-}[u]|-{(1,0)} &
  \bullet \ar@{-}[u]_-{(1,0)} \\
  \bullet \ar@{-}[u]^-{(1,0)} &  
  \bullet \ar@{-}[u]|-{(1,0)} &
  \bullet \ar@{-}[u]_-{(1,1)} \\
  & \bullet \ar@{-}[ul]^-{(\tfrac 1 3,0)} \ar@{-}[u]|-{(\tfrac 1 3,1)} \ar@{-}[ur]_-{(\tfrac 1 3,1)}
} \qquad \text{and} \qquad 
  \cd[@-0.5em]{
  \vdots & \vdots & \vdots\\
  \bullet \ar@{-}[u]^-{(1,0)} &  
  \bullet \ar@{-}[u]|-{(1,0)} &
  \bullet \ar@{-}[u]_-{(1,0)} \\
  \bullet \ar@{-}[u]^-{(1,0)} &  
  \bullet \ar@{-}[u]|-{(\tfrac 1 2, 0)} \ar@{-}[ur]_-{(\tfrac 1 2, 1)} \\
  & \bullet \ar@{-}[ul]^-{(\tfrac 1 3, 0)} \ar@{-}[u]_-{(\tfrac 2 3, 1)} 
  }
\end{equation*}
which are \emph{not} bisimilar; but nonetheless, applying the
normalisation function $\mathsf{reify} \circ \mathsf{reflect}$ sends
$\mathsf{beh}(s)$ to $\mathsf{beh}(s')$.

\subsection{Concluding remarks}
\label{sec:concluding-remarks}

The application developed in this section can be extended in two
directions. On the one hand, we can apply it as-is to other examples
of monads which admit hypernormalisation, for instance the
continuous probability monads of Section~\ref{sec:examples}.

More interestingly, we can change the kind of automaton we consider.
We considered \emph{generative} $\mnd T$-systems, which produce a
stream of output values with as only external stimulus the
computational effects encoded by the monad $\mnd T$. However, we can
just as easily consider automata that both consume and produce tokens:
indeed, if we define a \emph{Mealy $\mnd T$-machine} with input
alphabet $I$ and output alphabet $O$ to be a map of the form
$\Sigma_{i \in I}S \rightarrow T(\Sigma_{o \in O} S)$, then the
notions of behaviour, trace, reflection and reification given above
adapt unproblematically.


However, there are other directions in which we might wish to
generalise our automata. Most obviously, we might wish to introduce
the possibility of termination, by looking at coalgebras of the form
$S \rightarrow T(A \times S + 1)$ or
$S \rightarrow T(A \times S) + 1$. To capture the notion of behaviour
and trace in such situations, we must look not at $A$-ary comagmas but
\emph{comodels} of more general algebraic theories, which involve not
only ``cooperations'' $S \rightarrow \sum_{a \in A} S$ but also
equations between derived cooperations; see, for
example~\cite{Power2004From, Uustalu2015Stateful, Ahman2020Runners,
  Garner2020The-costructure-cosemantics}. \emph{Prima facie}, it is
not at all clear that our normalisation-by-trace-evaluation extends to
this context, since coequations between derived operations do not play
well with the hypernormalisation structure. However, this is not to
say that a more nuanced approach might not be possible: but this is to
be left for further work.

\bibliographystyle{acm}
\bibliography{bibdata}

\end{document}